\documentclass[english,a4paper,nolineno]{socg-lipics-v2021}

\usepackage{mathtools}
\usepackage{amssymb}
\usepackage{amscd}
\usepackage{amsthm}
\usepackage{eucal}
\usepackage{tikz-cd}
\usepackage{setspace}
\usepackage[utf8]{inputenc}
\usepackage{blkarray}
\usepackage{hyperref}
\usepackage{cleveref}
\usepackage[normalem]{ulem}
\usepackage{soul}
\usepackage{extarrows}
\hideLIPIcs

\DeclareMathSymbol{\minus}{\mathbin}{AMSa}{"39}

\newcommand{\tamal}[1]		{{ \textcolor{red} {#1}}}

\newcommand{\N}{\mathbb{N}_0}

\newcommand{\im}{\text{im}\hspace{1pt}}

\newcommand{\field}{\Bbbk}

\newcommand{\cancel}[1]

\title{Efficient Algorithms for Complexes of Persistence Modules with Applications}

\author{Tamal K. Dey}{Department of Computer Science, Purdue University, West Lafayette, USA}{tamaldey@purdue.edu}{}{}

\author{Florian Russold}{Institute of Geometry, Graz University of Technology, Graz, Austria}{russold@tugraz.at}{}{}
\author{Shreyas N. Samaga}{Department of Computer Science, Purdue University, West Lafayette, USA}{ssamaga@purdue.edu}{}{}

 \authorrunning{T. Dey, F. Russold, S. Samaga}

\Copyright{Tamal K. Dey and Florian Russold and Shreyas N. Samaga}


\funding{This work is supported partially by NSF grants CCF 2049010 and 2301360 and by the Austrian Science Fund (FWF): W1230}
\category{}

\supplement{}

\begin{document}

\maketitle

\begin{abstract}
We extend the persistence algorithm, viewed as an algorithm computing the homology of a complex of free persistence or graded modules, to complexes of modules that are not free. We replace persistence modules by their presentations and develop an efficient algorithm to compute the homology of a complex of presentations. To deal with inputs that are not given in terms of presentations, we give an efficient algorithm to compute a presentation of a morphism of persistence modules. This allows us to compute persistent (co)homology of instances giving rise to complexes of non-free modules. Our methods lead to a new efficient algorithm for computing the persistent homology of simplicial towers and they enable efficient algorithms to compute the persistent homology of cosheaves over simplicial towers and cohomology of persistent sheaves on simplicial complexes. We also show that we can compute the cohomology of persistent sheaves over arbitrary finite posets by reducing the computation to a computation over simplicial complexes.
\end{abstract}

\section{Introduction}
The theory of persistence, a central building block of topological data analysis, is concerned with the study of persistent objects and their persistent homology. A \emph{persistent object} $\Vec{A}\colon\mathbb{N}_0\rightarrow \mathbf{A}$ is a sequence of objects
\begin{equation} \label{persistent_object}
\begin{tikzcd}
\Vec{A}:\,A_0 \arrow[r,"f_0"] & A_1 \arrow[r,"f_1"] & A_2 \arrow[r,"f_2"] & \cdots
\end{tikzcd}
\end{equation}
in a category $\mathbf{A}$ with morphisms $f_i$. Let $\field$ be a field. Assuming that there is a chain complex of $\field$-vector spaces $C_\bullet(A_i)$ (in a category denoted $\mathbf{Ch}(\mathbf{A})$) associated to these objects inducing their homology, a persistent object provides a persistent chain complex $C_\bullet(\Vec{A})\colon\mathbb{N}_0\rightarrow \mathbf{Ch}(\mathbf{A})$ (\eqref{persistent_complex} right) or equivalently a chain complex of persistence modules (\eqref{persistent_complex} left).   
\begin{equation} \label{persistent_complex}
\begin{tikzcd}[column sep=large, row sep=large]
C_{k+1}(\Vec{A}): \arrow[d,"\partial_{k+ 1}"] &[-35pt] C_{k+ 1}(A_0) \arrow[r,"C_{k+ 1}(f_0)"] \arrow[d,"\partial_{k+ 1}^{0}"] &[10pt] C_{k+ 1}(A_1) \arrow[r,"C_{k+ 1}(f_1)"] \arrow[d,"\partial_{k+ 1}^1"] &[10pt] C_{k+ 1}(A_2) \arrow[r,"C_{k+ 1}(f_2)"] \arrow[d,"\partial_{k+ 1}^{2}"] &[10pt] \cdots\\
C_k(\Vec{A}): \arrow[d,"\partial_k"] & C_{k}(A_0) \arrow[r,"C_{k}(f_0)"] \arrow[d,"\partial_{k}^{0}"] & C_{k}(A_1) \arrow[r,"C_{k}(f_1)"] \arrow[d,"\partial_{k}^{1}"] & C_{k}(A_2) \arrow[r,"C_{k}(f_2)"] \arrow[d,"\partial_{k}^{2}"] & \cdots \\
C_{k\minus 1}(\Vec{A}): &[-30pt] C_{k\minus 1}(A_0) \arrow[r,"C_{k\minus 1}(f_0)"] &[5pt] C_{k\minus 1}(A_1) \arrow[r,"C_{k\minus 1}(f_1)"] &[5pt] C_{k\minus 1}(A_2) \arrow[r,"C_{k\minus 1}(f_2)"] &[-10pt] \cdots
\end{tikzcd}
\end{equation}
We call the sequence $C_{k+1}(\Vec{A})\xrightarrow{\partial_{k+1}} C_k(\Vec{A})\xrightarrow{\partial_k}C_{k\minus 1}(\Vec{A})$ where $C_k(\Vec{A})$ is a persistence module and $\partial_{k}\circ\partial_{k+1}=0$ a \emph{complex of persistence modules}. The \emph{persistent homology} $H_k(\Vec{A})\coloneqq\ker \partial_k/\im\partial_{k+1}$ of the complex of persistence modules \eqref{persistent_complex} is given by the persistence module 
\begin{equation*}
\begin{tikzcd}[column sep=large]
H_k(\Vec{A}):\, H_k(A_0) \arrow[r,"H_k(f_0)"] & H_k(A_1) \arrow[r,"H_k(f_1)"] & H_k(A_2) \arrow[r,"H_k(f_2)"] & \cdots
\end{tikzcd}
\end{equation*}
where $H_k(A_i)\coloneqq\ker\partial^i_k/\im\partial^i_{k+1}$ and $H_k(f_i)$ are induced by the chain maps $C_k(f_i)$. The goal of this paper is to present an efficient general purpose algorithm to compute the homology of a complex of persistence modules and demonstrate some of its applications. 

The inspiration for our approach comes from the well-known persistence algorithm \cite{ELZ02}. Suppose the persistent object in \eqref{persistent_object} is a finite filtration of simplicial complexes, i.e.\ $A_i$ is a finite simplicial complex, $f_i$ is an inclusion and $C_\bullet(\Vec{A})$ is the complex of persistent simplicial chains. In \cite{ZC05} the authors observed that a persistence module can equivalently be viewed as a graded module over the polynomial ring $\field[t]$. Using this perspective, $C_k(\Vec{A})$ is a free module with a basis given by the $k$-simplices in the filtration and the boundary morphisms $\partial_k$ can be represented by matrices w.r.t.\ these bases. The persistence algorithm leverages this compressed representation of $C_\bullet(\Vec{A})$ by its generators and computes the persistent homology of the whole filtration at once. This makes it much more efficient than naively computing $H_k(A_i)$ for each index $i$, where simplices would be considered multiple times.  

If we allow the $f_i$'s to be arbitrary simplicial maps, we obtain what is called a \emph{simplicial tower}. It turns out that this apparently simple modification brings a significant change at the algebraic level because the persistence modules $C_k(\Vec{A})$ in the complex may no longer remain free (relations among generators may appear). Thus, in general, they do not admit a basis and we can not straightforwardly represent $\partial_k$ by matrices and compute the homology using linear algebra over $\field[t]$. In~\cite{DFW14,KS17} the authors tackle this algebraic difficulty on a topological level, by expanding a tower into a filtration and then applying efficient algorithms for filtrations. We tackle the problem directly at the algebraic level by designing algorithms
that can handle complexes of non-free modules. These algorithms enable
a further generalization obtained by additionally considering algebraic information over a simplicial tower. If this algebraic information is provided by a cosheaf, we obtain the case of persistent cosheaf homology \cite{russold}. The persistent cosheaf homology is again the homology of a complex of not necessarily free persistence modules. But cosheaf homology is not a homotopy invariant, so the methods for the plain tower~\cite{DFW14,KS17} do not work in this case. Here we are forced to tackle the problem at an algebraic level. 

One of our main observations is that we can compute the homology of a complex of non-free modules if we consider relations in addition to the generators. This brings \emph{presentations} of modules into the picture which
are morphisms from the free modules of relations to the free modules of
generators. So, we design: 1) an efficient algorithm (Section~\ref{sec_presentation_algorithm}) to compute a presentation of a morphism of persistence modules by free modules which allows us to convert a complex of persistence modules into a \emph{complex of presentations}, and 2) an efficient algorithm (Section~\ref{sec_cohomology_presentations}) to compute the homology of a complex of presentations. In the spirit of the persistence algorithm our method considers a generator only once even though it may exist over a wide range of indices. In fact, our algorithm is a direct generalization of the persistence algorithm to which it specializes in the case of free modules. 

At this point, we note that we are not aware of any computational approach
in the TDA literature that deals with complexes of persistence modules
and complexes of presentations in the full generality as our approach does. The closest along this line is the recent work in \cite{casas} where the author introduces a new framework of barcode bases and operations on them to compute a barcode basis of the homology of a complex of persistence modules in the context of distributed persistent homology computations. Our presentation algorithm does not require such a specific barcode form and can process general presentations. The paper also does not discuss how to compute barcode bases and maps between them from general morphisms of persistence modules which we address.

\textbf{Applications:} A central motivation of this paper is to provide computational tools for the ever-growing body of ideas to use methods from algebraic topology in applications. In recent years, the idea of using methods from sheaf theory in applications has gained traction in the field of TDA~\cite{BGO19,curry,ghristbook,kashiwara,robinsonbook}. Instead of just considering a space by itself, sheaves allow us to study the behaviour of data over a space. Sheaves and their cohomology have been used in various applications, see e.g.~\cite{pmlr-v196-barbero22a,NEURIPS2022_75c45fca,brown,GH11,hansen2020sheaf,HG21,robinson4}. Persistent versions of sheaves and cosheaves have also appeared in the TDA literature, see e.g.~\cite{decoratedmapper,covid,wei2022persistent}. Moreover persistent (co)sheaves have been used as a framework for a distributed computation of (persistent) homology \cite{morse,casas,yoonghrist}.
In \cite{russold}, a general theory of persistent sheaf cohomology has been developed for which this paper establishes a complete computational framework. 

We have already mentioned that our approach provides a novel efficient algorithm to compute the barcode of a given tower (Section~\ref{sec_application_tower}). We demonstrate that it provides efficient algorithms to compute various flavors of persistent cosheaf homology and sheaf cohomology. We consider the persistent homology of a cosheaf over a varying simplicial complex (Section~\ref{sec_applications_cosheaf_tower}) and the cohomology of a persistent sheaf on a fixed simplicial complex (Section~\ref{sec_pers_shv_simp}). We also show that we can reduce the computation of persistent sheaf cohomology over an arbitrary finite poset to a computation over a simplicial complex (Section~\ref{sec_poset}). 

We have a preliminary implementation of the two basic algorithms
mentioned before. Experimental results (Appendix~\ref{app_experiments}) suggest that our approach is not merely theoretical, but has the potential to be useful in practice (\url{https://github.com/TDA-Jyamiti/Algos-cplxs-pers-modules/}).



\section{Persistence modules, graded modules, and presentations} \label{sec_background}
In this section we recall basic notions of persistence modules, graded modules and their presentations. A persistence module, as depicted in the top row of \eqref{MPM}, is a functor $M\colon \mathbb{N}_0\rightarrow \mathbf{vec}$ where $\mathbf{vec}$ denotes the category of finite dimensional vector spaces. It is of \emph{finite type}, if there exists an $m\in\mathbb{N}_0$ such that $M(i\leq j)$ is an isomorphism for all $i\geq m$. A morphism of persistence modules $\phi\colon M\rightarrow N$, as depicted in \eqref{MPM},
\begin{equation} \label{MPM}
\begin{tikzcd}[column sep=huge]
M: \arrow[d,swap,shift left=-2pt,"\phi"] &[-45pt] M(0) \arrow[r,"M(0\leq 1)"] \arrow[d,"\phi(0)"] &[-5pt] M(1) \arrow[r,"M(1\leq 2)"] \arrow[d,"\phi(1)"] &[-5pt] \cdots \arrow[r,"M(m\minus 1\leq m)"] &[5pt] M(m) \arrow[d,"\phi(m)"] \arrow[r,"M(m\leq m+1)"] &[5pt] \cdots \\
N: & N(0) \arrow[r,"N(0\leq 1)"] & N(1) \arrow[r,"N(1\leq 2)"] & \cdots \arrow[r,"N(m\minus 1\leq m)"] & N(m) \arrow[r,"N(m\leq m+1)"] & \cdots
\end{tikzcd}
\end{equation}
is a natural transformation of functors $\mathbb{N}_0\rightarrow \mathbf{vec}$. Let $\mathbf{pMod}$ denote the category of persistence modules of finite type.

\begin{example} \label{exmp:morphism_persistence_modules}
The following diagram shows a concrete example of a morphism of persistence modules of finite type over $\field=\mathbb{Z}_2$ (the field with 2 elements): 
\begin{equation} \label{eq:concrete_morphism}
\begin{tikzcd}[every label/.append style = {font = \tiny}]
M \arrow[d,swap,"\phi"] &[-30pt] : &[-15pt] \mathbb{Z}_2^3 \arrow[r,"\begin{pmatrix}1 \hspace{1pt} 0 \hspace{1pt} 0 \\ 0 \hspace{1pt} 1 \hspace{1pt} 0 \\ 0 \hspace{1pt} 0 \hspace{1pt} 0 \end{pmatrix}"] \arrow[d,"\begin{pmatrix} 0 \hspace{1pt} 0 \hspace{1pt} 0 \\ 0 \hspace{1pt} 1 \hspace{1pt} 1 \\ 1 \hspace{1pt} 1 \hspace{1pt} 1 \end{pmatrix}"] & \mathbb{Z}_2^3 \arrow[r,"\begin{pmatrix}0 \hspace{1pt} 0 \hspace{1pt} 1 \\ 0 \hspace{1pt} 0 \hspace{1pt} 0 \\ 0 \hspace{1pt} 0 \hspace{1pt} 0 \end{pmatrix}"] \arrow[d,"\begin{pmatrix} 0 \hspace{1pt} 0 \hspace{1pt} 1 \\ 0 \hspace{1pt} 0 \hspace{1pt} 1 \\ 0 \hspace{1pt} 0 \hspace{1pt} 0 \end{pmatrix}"] & \mathbb{Z}_2^3 \arrow[d,"\begin{pmatrix}1 \hspace{1pt} 1 \hspace{1pt} 1 \\ 1 \hspace{1pt} 1 \hspace{1pt} 1 \\ 0 \hspace{1pt} 1 \hspace{1pt} 0 \end{pmatrix}"] \arrow[r,"\cong"] & \cdots \\
N & : & \mathbb{Z}_2^3 \arrow[r,swap,"\begin{pmatrix}1 \hspace{1pt} 0 \hspace{1pt} 0 \\ 0 \hspace{1pt} 0 \hspace{1pt} 0 \\ 0 \hspace{1pt} 0 \hspace{1pt} 0 \end{pmatrix}"] & \mathbb{Z}_2^3 \arrow[r,swap,"\begin{pmatrix}1 \hspace{1pt} 0 \hspace{1pt} 0 \\ 0 \hspace{1pt} 1 \hspace{1pt} 0 \\ 0 \hspace{1pt} 0 \hspace{1pt} 1 \end{pmatrix}"] & \mathbb{Z}_2^3 \arrow[r,"\cong"] & \cdots
\end{tikzcd}
\end{equation}
\end{example}

\noindent
An $\mathbb{N}_0$-graded $\field[t]$-module $M$ is a direct sum of $\field$-vector spaces $\bigoplus_{i\in\mathbb{N}_0}M_i$ with the usual $\field$-action on the summands and a $t^i$-action for all $i\in\mathbb{N}_0$ such that $t^i\cdot M_j\subseteq M_{i+j}$. If $m\in M_i$ it is called a homogeneous element and we define its \emph{degree} by $\text{deg}(m)\coloneqq i$. A (homogeneous) morphism $\phi\colon M\rightarrow N$ of graded $\field[t]$-modules $M$ and $N$ is a map of $\field[t]$-modules such that $\phi(M_i)\subseteq N_i$. Let $\mathbf{grMod}_{\field[t]}$ denote the category of finitely generated $\mathbb{N}_0$-graded $\field[t]$-modules. We know the following equivalence~\cite[Corollary 10]{corbet}\cite[Theorem 3.1]{ZC05} which allows us to use persistence modules and graded modules interchangeably in the following.
\begin{proposition} \label{equivalence_persmod}
$\mathbf{pMod} \cong \mathbf{grMod}_{\field[t]}$. 
\end{proposition}

\noindent
In what follows, we assume that all persistence modules and graded modules are of finite type or finitely generated, respectively. Persistence modules as well as morphisms between them can be represented by collections of matrices (Example \ref{exmp:morphism_persistence_modules}). Their equivalent counterpart graded $\field[t]$-modules can not be handled by matrices directly since they can have torsion in general. To handle them with methods of linear algebra over $\field[t]$ we use presentations.    

\begin{definition}[Presentation of module] \label{def:pre_module}
Let $M$ be a finitely generated graded $\field[t]$-module. A presentation of $M$ is an exact sequence of the form 
\begin{equation*} \label{pres_module}
\begin{tikzcd}
0 & M \arrow[l]  & P_0 \arrow[l,swap,"\mu"] & P_1 \arrow[l,swap,"p"]  
\end{tikzcd}
\end{equation*}
where $P_0$ and $P_1$ are free finitely generated graded $\field[t]$-modules. We call a presentation \emph{reduced}, if $p$ is a monomorphism.
\end{definition}

\begin{definition}[Presentation of morphism]
Let $\phi\colon M\rightarrow N$ be a morphism of finitely generated graded $\field[t]$-modules $M$ and $N$. A presentation of $\phi$ is a commutative diagram 
\begin{equation}
\begin{tikzcd} \label{presentation}
0 & M \arrow[l] \arrow[d,"\phi"] &  P_0 \arrow[l,swap,"\mu"] \arrow[d,"f_0"] & P_1 \arrow[l,swap,"p"] \arrow[d,"f_1"] \\
0 & N \arrow[l] &  Q_0 \arrow[l,swap,"\nu"] &  Q_1 \arrow[l,swap,"q"]
\end{tikzcd}
\end{equation}
where the rows are presentations of $M$ and $N$ respectively.
\end{definition}

\noindent
In the following, we also refer to the morphism of free modules $P_0\xleftarrow{p} P_1$ in Definition \ref{def:pre_module} as a presentation of $M$. By the exactness assumption $M\cong \text{coker }p$ and $\mu\cong (P_0\xrightarrow{\pi} \text{coker }p)$. We also refer to the right square in \eqref{presentation} as a morphism of presentations.

Given $a< b\in\mathbb{N}_0\cup\{\infty\}$, let $\mathbb{I}_{[a,\infty)}$ denote the free graded $\field[t]$-module generated by a single generator of degree $a$. Moreover, we denote by $\mathbb{I}_{[a,b)}$ the quotient module $\mathbb{I}_{[a,\infty)}/ \mathbb{I}_{[b,\infty)}$. By the equivalence of Proposition \ref{equivalence_persmod}, $\mathbb{I}_{[a,b)}$ corresponds to the indecomposable interval persistence module starting at index $a$ and ending at index $b$. By the Theorems of Krull-Remak-Schmidt~\cite[Theorem 1]{azumaya1950corrections} and Gabriel~\cite[Chapter 2.2]{Gabriel72} and Proposition \ref{equivalence_persmod}, every finitely generated graded $\field[t]$-module $M$ is isomorphic to a finite direct sum of indecomposable interval modules $M\cong\bigoplus_{i=1}^d \mathbb{I}_{[a_i,b_i)} \hspace{2pt}$. We call the multiset of intervals $\{[a_i,b_i)\vert 1\leq i\leq d\}$ the \emph{barcode} of $M$. 

\begin{example}\label{expm:barcodes}
For the persistence modules $M$ and $N$ in Example \ref{exmp:morphism_persistence_modules} we get $M\cong \mathbb{I}_{[0,2)}\oplus \mathbb{I}_{[0,2)}\oplus \mathbb{I}_{[0,1)}\oplus \mathbb{I}_{[1,\infty)}\oplus \mathbb{I}_{[2,\infty)}\oplus \mathbb{I}_{[2,\infty)}$ and $N\cong\mathbb{I}_{[0,\infty)}\oplus \mathbb{I}_{[0,1)}\oplus \mathbb{I}_{[0,1)}\oplus \mathbb{I}_{[1,\infty)}\oplus \mathbb{I}_{[1,\infty)}$. This can be inferred from the ranks of the matrices in the upper and lower row. In the upper row we have $\text{rank }M(0\leq 1)=2$, $\text{rank }M(1\leq 2)=1$ and $\text{rank }M(0\leq 2)=0$. This implies that only two of the three bars starting in index 0 reach index 1 and none of them reaches index 2. Moreover, only one of the three bars at index 1 reaches index 2.  
\end{example}

\noindent
For two interval modules $\mathbb{I}_{[a,b)}$ and $\mathbb{I}_{[c,d)}$ we have
\begin{equation} \label{hom_pmod}
\text{Hom}(\mathbb{I}_{[a,b)},\mathbb{I}_{[c,d)})\cong \begin{cases} \field \hspace{5pt} \text{if } c\leq a<d\leq b \\ 0 \hspace{5pt} \text{else} \end{cases} \hspace{2pt} .
\end{equation}
In other words, if two bars overlap in a certain way as stated in \eqref{hom_pmod}, then, for every scalar $\lambda\in \field$, there is a unique morphism defined by multiplication with $\lambda$ where the bars overlap. Moreover, given two finitely generated graded modules $M$ and $N$, we have
\begin{equation} \label{hom_sum}
\text{Hom}(M,N)\cong \text{Hom}(\bigoplus_i \mathbb{I}_{[a_i,b_i)},\bigoplus_j \mathbb{I}_{[c_j,d_j)})\cong \bigoplus_i\bigoplus_j\text{Hom}(\mathbb{I}_{[a_i,b_i)},\mathbb{I}_{[c_j,d_j)}) \hspace{2pt} .
\end{equation}
We call the following reduced presentation of $\mathbb{I}_{[a,b)}$ an \emph{elementary presentation}
\begin{equation} \label{elemenatry_presentation}
\begin{tikzcd}
0 & \mathbb{I}_{[a,b)} \arrow[l] & \mathbb{I}_{[a,\infty)} \arrow[l,swap] &[10pt] \mathbb{I}_{[b,\infty)} \arrow[l,swap,"\cdot t^{b\minus a}"] & 0 \arrow[l]
\end{tikzcd}
\end{equation}
where we set $\mathbb{I}_{[b,\infty)}$ and the map to zero if $b=\infty$.
The direct sum of presentations is defined by the pointwise direct sum of graded modules. Thus, assuming $b_i\neq \infty$ iff $1\leq i \leq d'\leq d$, we obtain the following reduced presentation of $M$ (up to isomorphism) 
\begin{equation*}
\begin{tikzcd}
0 & \bigoplus_{i=1}^d \mathbb{I}_{[a_i,b_i)} \arrow[l] & \bigoplus_{i=1}^d \mathbb{I}_{[a_i,\infty)} \arrow[l,swap] & \bigoplus_{i=1}^{d'} \mathbb{I}_{[b_i,\infty)} \arrow[l,swap,"p"] & 0 \arrow[l]
\end{tikzcd}
\end{equation*}
as a direct sum of elementary presentations where $p$ is of the form 
\begin{equation} \label{pform}
p=\begin{pmatrix}
t^{b_1\minus a_1} & 0 & \cdots & 0  \\
0 & t^{b_2\minus a_2} & \cdots  & 0  \\
\vdots & \vdots & \ddots & \vdots  \\
0 & 0 & \cdots  & t^{b_{d'}\minus a_{d'}}  \\
0 & 0 & \cdots  & 0  \\
\vdots & \vdots & \ddots & \vdots  \\
0 & 0 & \cdots  & 0  \\
\end{pmatrix}.
\end{equation}

\noindent
We call a presentation \emph{canonical} if it is a direct sum of elementary presentations and has no elementary summands \eqref{elemenatry_presentation} with $a=b$. The matrix $p$ of a canonical presentation has a form as in \eqref{pform}, i.e.\ every relation in $P_1$ maps to a unique generator in $P_0$. We obtain a natural inclusion and projection of elementary summands, i.e.\ a commutative diagram of the form: 
\begin{equation*}
\begin{tikzcd}
0 & \mathbb{I}_{[a_i,b_i)} \arrow[l] \arrow[d,"\iota"]& \mathbb{I}_{[a_i,\infty)} \arrow[l,swap] \arrow[d,"\iota_0"] &[10pt] \mathbb{I}_{[b_i,\infty)} \arrow[l,swap,"\cdot t^{b_i\minus a_i}"]  \arrow[d,"\iota_1"] & 0 \arrow[l] \\
0 & \underset{l}{\bigoplus} \mathbb{I}_{[a_l,b_l)} \arrow[l] \arrow[d,"\pi"] & \underset{l}{\bigoplus} \mathbb{I}_{[a_l,\infty)} \arrow[l,swap] \arrow[d,"\pi_0"] & \underset{l}{\bigoplus} \mathbb{I}_{[b_l,\infty)} \arrow[l,swap,"p"] \arrow[d,"\pi_1"] & 0 \arrow[l] \\
0 & \mathbb{I}_{[a_j,b_j)} \arrow[l] & \mathbb{I}_{[a_j,\infty)} \arrow[l,swap] & \mathbb{I}_{[b_j,\infty)} \arrow[l,swap,"\cdot t^{b_j\minus a_j}"]  & 0 \arrow[l] 
\end{tikzcd}
\end{equation*}
Presentations of finitely generated graded $\field[t]$-modules and their morphisms can be represented by the matrices corresponding to the morphisms of free modules $p,q,f_0,f_1$ in \eqref{presentation} with labeled rows and columns recording the degree of the generators. 

\begin{example} \label{exmp:presentation}
In Example \ref{expm:barcodes} we computed the barcodes of the persistence modules $M$ and $N$ in Example \ref{exmp:morphism_persistence_modules}. To get a presentation of $M$ we put a generator $\mathbb{I}_{[a,\infty)}$ for every interval module $\mathbb{I}_{[a,b)}$ and a relation $\mathbb{I}_{[b,\infty)}$ for every interval module with $b<\infty$ such that $\mathbb{I}_{[b,\infty)}$ maps to $\mathbb{I}_{[a,\infty)}$. Thus, we obtain $P_0=\mathbb{I}_{[0,\infty)}\oplus \mathbb{I}_{[0,\infty)}\oplus \mathbb{I}_{[0,\infty)}\oplus \mathbb{I}_{[1,\infty)}\oplus \mathbb{I}_{[2,\infty)}\oplus \mathbb{I}_{[2,\infty)}$ and $P_1=\mathbb{I}_{[2,\infty)}\oplus \mathbb{I}_{[2,\infty)}\oplus \mathbb{I}_{[1,\infty)}$. Analogously, for the generators $Q_0$ and relations $Q_1$ of $N$ we obtain $Q_0=\mathbb{I}_{[0,\infty)}\oplus \mathbb{I}_{[0,\infty)}\oplus \mathbb{I}_{[0,\infty)}\oplus \mathbb{I}_{[1,\infty)}\oplus \mathbb{I}_{[1,\infty)}$ and $Q_1=\mathbb{I}_{[1,\infty)}\oplus \mathbb{I}_{[1,\infty)}$. The following diagram is a presentation of the morphism $\phi\colon M\rightarrow N$ in \eqref{eq:concrete_morphism} 
\begin{equation*}
\begin{tikzcd}[ampersand replacement=\&,every label/.append style = {font = \tiny},column sep=large, row sep=large]
0 \&[-10pt] M \arrow[d,"\phi"] \arrow[l] \& P_0 \arrow[l] \arrow{d}[xshift=-2pt,yshift=-3pt]{\begin{blockarray}{c@{\hspace{7pt}}cccccc}
& 0 & 0 & 0 & 1 & 2 & 2 \\
\begin{block}{c@{\hspace{7pt}}(cccccc)}
0 & 0 & 0 & 0 & t & t^2 & t^2 \\
0 & 0 & 1 & 1 & 0 & 0 & 0 \\
0 & 1 & 1 & 1 & 0 & 0 & 0 \\
1 & 0 & 0 & 0 & 1 & t & t \\
1 & 0 & 0 & 0 & 0 & t & 0 \\
\end{block}
\end{blockarray}} \&[60pt] P_1 \arrow{l}[swap,xshift=-2pt,yshift=-3pt]{\begin{blockarray}{c@{\hspace{7pt}}ccc}
& 2 & 2 & 1  \\
\begin{block}{c@{\hspace{7pt}}(ccc)}
0 & t^2 & 0 & 0 \\
0 & 0 & t^2 & 0 \\
0 & 0 & 0 & t \\
1 & 0 & 0 & 0 \\
2 & 0 & 0 & 0 \\
2 & 0 & 0 & 0 \\
\end{block}
\end{blockarray}} \arrow{d}[xshift=-2pt,yshift=-3pt]{\begin{blockarray}{c@{\hspace{7pt}}cccccc}
& 2 & 2 & 1 \\
\begin{block}{c@{\hspace{7pt}}(cccccc)}
1 & 0 & t & 1 \\
1 & t & t & 1 \\
\end{block}
\end{blockarray}} \&[-10pt] 0 \arrow[l] \\[20pt]
0 \& N \arrow[l] \& Q_0 \arrow[l] \& Q_1 \arrow{l}[xshift=-2pt,yshift=-3pt]{\begin{blockarray}{c@{\hspace{7pt}}cc}
& 1 & 1  \\
\begin{block}{c@{\hspace{7pt}}(cc)}
0 & 0 & 0 \\
0 & t & 0 \\
0 & 0 & t \\
1 & 0 & 0 \\
1 & 0 & 0 \\
\end{block}
\end{blockarray}}  \& 0 \arrow[l] 
\end{tikzcd}
\end{equation*}
where the vertical matrices can be inferred by following the generators corresponding to the bars in \eqref{eq:concrete_morphism}. 
\end{example}

\noindent
When recording the degree of the generators, we can use coefficients in $\field$ instead of coefficients in $\field[t]$ since the degree of column and row determines the polynomial factor $t^r$ of an entry. If we have a morphism of presentations as in \eqref{presentation} where $p$ and $q$ are in canonical form, then, by the structure of $p$ and $q$ and by commutativity, $f_1$ is uniquely determined by $f_0$. To represent such a morphism of presentations it is enough to store the matrix $f_0$ and for each row and column the degree of the generator in $P_0$ or $Q_0$ and its unique corresponding relation in $P_1$ or $Q_1$.
\begin{example} \label{exmp:kcoef_pres}
The following matrix completely represents the presentation in example \ref{exmp:presentation}:
\begin{equation*} \footnotesize
\begin{blockarray}{ccccccc}
& \text{[}0,2) & \text{[}0,2) & \text{[}0,1) & \text{[}1,\infty) & \text{[}2,\infty) & \text{[}2,\infty) \\
\begin{block}{c(cccccc)}
\text{[}0,\infty) & 0 & 0 & 0 & 1 & 1 & 1 \\
\text{[}0,1) & 0 & 1 & 1 & 0 & 0 & 0 \\
\text{[}0,1) & 1 & 1 & 1 & 0 & 0 & 0 \\
\text{[}1,\infty) & 0 & 0 & 0 & 1 & 1 & 1 \\
\text{[}1,\infty) & 0 & 0 & 0 & 0 & 1 & 0 \\
\end{block}
\end{blockarray}    
\end{equation*}
\end{example}


\section{Computing homology of complexes of presentations} \label{sec_cohomology_presentations}

In this section we develop an algorithm that computes the barcode of the homology of a complex of finitely generated graded modules where we assume that the complex is given by presentations. Consider the following sequence of finitely generated graded $\field[t]$-modules
\begin{equation*} \label{cochain_comp_mod}
L\xlongrightarrow{\phi} M \xlongrightarrow{\psi} N
\end{equation*}
where $\psi\circ\phi=0$. Suppose that we are given reduced presentations of this sequence, i.e., a commutative diagram with exact rows of the form:
\begin{equation} \label{comp_presentations}
\begin{tikzcd}
0 & L \arrow[l] \arrow[d,"\phi"] & P_0 \arrow[l,swap,"\lambda"] \arrow[d,"f_0"] & P_1 \arrow[l,swap,"p"] \arrow[d,"f_1"] & 0 \arrow[l] \\
0 & M \arrow[l] \arrow[d,"\psi"] & Q_0 \arrow[l,swap,"\mu"] \arrow[d,"g_0"] & Q_1 \arrow[l,swap,"q"] \arrow[d,"g_1"] & 0 \arrow[l] \\
0 & N \arrow[l] & R_0 \arrow[l,swap,"\nu"] & R_1 \arrow[l,swap,"r"] & 0 \arrow[l] 
\end{tikzcd}
\end{equation}
such that $g_0\circ f_0=0$ and $g_1\circ f_1=0$. The condition $g_0\circ f_0=0$ and $g_1\circ f_1=0$ means that \eqref{comp_presentations} is a complex of presentations. In Example \ref{homology_algo_example} one can find a concrete instance of such a complex of presentations. Our goal is to compute a presentation of $\ker\psi / \im\phi$ from the presentations in \eqref{comp_presentations}, i.e., an exact sequence of the form:
\begin{equation*}
\begin{tikzcd}
0 & \ker\psi / \im\phi \arrow[l] & W_0 \arrow[l] & W_1 \arrow[l,swap,"w"] 
\end{tikzcd}
\end{equation*}
where $W_0\xleftarrow{w} W_1$ is a morphism of free modules. We first construct a presentation of $\ker\psi$ through the maps
\begin{equation*}
\begin{aligned}
\begin{pmatrix}g_0 & \minus r\end{pmatrix}&\colon Q_0\times R_1\rightarrow R_0 \hspace{3pt},\hspace{5pt} \begin{pmatrix}g_0 & \minus r\end{pmatrix}(x,y)\coloneqq g_0(x)-r(y) \\
\begin{pmatrix}q \\ g_1\end{pmatrix}&\colon Q_1\rightarrow Q_0\times R_1 \hspace{3pt},\hspace{5pt} \begin{pmatrix}q \\ g_1\end{pmatrix}(a)\coloneqq (q(a),g_1(a))
\end{aligned} 
\end{equation*}
where $\ker \begin{pmatrix}g_0 & \minus r\end{pmatrix}=\{(x,y)\in Q_0\times R_1\hspace{3pt}|\hspace{3pt} g_0(x)=r(y)\}$. Note that the kernel of a morphism of free modules is again a free module. Moreover, we define the map 
\begin{equation*}
\begin{aligned}
\kappa&\colon \ker(g_0\minus r)\rightarrow M \hspace{3pt},\hspace{5pt} \kappa(x,y)\coloneqq \mu(x) \hspace{1pt}.
\end{aligned} 
\end{equation*}

\begin{proposition} \label{presentation_kernel}
The following exact sequence
\begin{equation*}
\begin{tikzcd}
0 & \ker\psi \arrow[l] & \ker(g_0\minus r) \arrow[l,swap,"\kappa"] &[15pt] Q_1 \arrow[l,swap,"\begin{pmatrix}q \\ g_1\end{pmatrix}"]
\end{tikzcd} 
\end{equation*}
is a presentation of $\ker\psi$.
\end{proposition}

\begin{proof}
\setstretch{1.2}
\begin{enumerate}
\item Let $(x,y)\in\ker \begin{pmatrix}g_0 & \minus r \end{pmatrix}$, then $\psi\circ \kappa(x,y)=\psi\circ \mu (x)=\nu\circ g_0(x)=\nu\circ r(y)=0$. Thus, $\im\kappa\subseteq  \ker\psi$ and $\kappa\colon\ker \begin{pmatrix}g_0 & \minus r \end{pmatrix}\rightarrow \ker\psi$.
\item Let $a\in\ker\psi$. Since $\mu$ is an epimorphism, $\exists x\in Q_0$ s.t.\ $\mu(x)=a$ and $\psi\circ \mu(x)=0=\nu\circ g_0(x)$. We get $g_0(x)\in\ker \nu=\im r$ and $\exists y\in R_1$ s.t.\ $r(y)=g_0(x)$. Thus, $\begin{pmatrix}g_0 & \minus r \end{pmatrix}(x,y)=0$ and $\kappa(x,y)=\mu(x)=a$. Hence, $\kappa$ is an epimorphism.
\item Let $z\in Q_1$, then $\begin{pmatrix}g_0 & \minus r\end{pmatrix}\circ \begin{pmatrix}q \\ g_1\end{pmatrix}(z)=g_0\circ q(z)-r\circ g_1(z)=0$. Thus, $\im \begin{pmatrix}q \\ g_1\end{pmatrix}\subseteq \ker \begin{pmatrix}g_0 & \minus r\end{pmatrix}$ and $\begin{pmatrix}q \\ g_1\end{pmatrix}\colon Q_1\rightarrow \ker \begin{pmatrix}g_0 & \minus r \end{pmatrix}$.
\item Let $(x,y)\in \ker \kappa$. Then, $\kappa(x,y)=\mu(x)=0$ and $x\in\ker \mu=\im q$. Thus, $\exists a\in Q_1$ s.t.\ $q(a)=x$. Since, $g_0(x)=g_0\circ q(a)=r\circ g_1(a)=r(y)$, we get $g_1(a)-y\in\ker r=0$. This implies, $g_1(a)=y$ and $\begin{pmatrix}q \\ g_1\end{pmatrix}(a)=(q(a),g_1(a))=(x,y)$. Hence, $\ker \kappa\subseteq \im \begin{pmatrix}q \\ g_1\end{pmatrix}$.

Conversely, let $(x,y)\in\im \begin{pmatrix}q \\ g_1\end{pmatrix}$. Then $\exists a\in Q_1$ s.t.\ $\begin{pmatrix}q \\ g_1\end{pmatrix}(a)=(q(a),g_1(a))=(x,y)$. We get $\kappa(x,y)=\kappa(q(a),g_1(a))=\mu\circ q(a)=0$. Hence, $(x,y)\in\ker \kappa$ and $\ker \kappa=\im \begin{pmatrix}q \\ g_1\end{pmatrix}$.
\end{enumerate}
\end{proof}

\noindent
Next we show that the presentations of $L$, $\ker\psi$ and $\phi$ are compatible via the map
\begin{equation*}
\begin{pmatrix}f_0 \\ 0\end{pmatrix}\colon P_0\rightarrow Q_0\times R_1 \hspace{3pt},\hspace{5pt} \begin{pmatrix}f_0 \\ 0\end{pmatrix}(x)\coloneqq(f_0(x),0) \hspace{1pt}.
\end{equation*}

\begin{proposition} \label{kernel_lemma}
The following diagram commutes:
\begin{equation*}
\begin{tikzcd}
0 & L \arrow[l] \arrow[d,"\phi"] &[10pt] P_0 \arrow[l,swap,"\lambda"] \arrow[d,"\begin{pmatrix}f_0 \\ 0\end{pmatrix}"] &[15pt] P_1 \arrow[l,swap,"p"] \arrow[d,"f_1"] \\[10pt]
0 & \ker\psi \arrow[l] & \ker(g_0\minus r) \arrow[l,swap,"\kappa"] & Q_1 \arrow[l,swap,"\begin{pmatrix}q \\ g_1\end{pmatrix}"]
\end{tikzcd} 
\end{equation*}
\end{proposition}

\begin{proof}
\setstretch{1.2}
\begin{enumerate}
\item Let $x\in P_0$. Since $g_0\circ f_0=0$, we have $\begin{pmatrix}g_0 & \minus r \end{pmatrix}(f_0(x),0)=g_0(f_0(x))=0$. Thus, $\im\begin{pmatrix}f_0 \\ 0\end{pmatrix}\subseteq \ker \begin{pmatrix}g_0 & \minus r \end{pmatrix}$ and $\begin{pmatrix}f_0 \\ 0\end{pmatrix}\colon P_0\rightarrow \ker \begin{pmatrix}g_0 & \minus r \end{pmatrix}$.
\item Let $x\in P_1$. Since $g_1\circ f_1=0$ we get $\begin{pmatrix}q \\ g_1\end{pmatrix}\circ f_1(x)=\big(q(f_1(x)),g_1(f_1(x))\big)=\big(f_0(p(x)),0\big)=\begin{pmatrix}f_0 \\ 0\end{pmatrix}\circ p(x)$. Hence, the right square commutes.
\item Let $x\in P_0$. Then $\kappa\circ \begin{pmatrix}f_0 \\ 0\end{pmatrix}(x)=\kappa(f_0(x),0)=\mu(f_0(x))=\phi\circ \lambda(x)$. Thus, also the left square commutes.
\end{enumerate}
\end{proof}

\noindent
Finally, we define the presentation map from the relations to the generators 
\begin{equation*}
\begin{pmatrix}f_0 & \minus q \\ 0 & \minus g_1\end{pmatrix}\colon P_0\times Q_1\rightarrow \ker(g_0\minus r) \hspace{3pt},\hspace{5pt} \begin{pmatrix}f_0 & \minus q \\ 0 & \minus g_1\end{pmatrix}(x,y)\coloneqq \begin{pmatrix}f_0 \\ 0\end{pmatrix}(x)-\begin{pmatrix}q \\ g_1\end{pmatrix}(y)
\end{equation*}
and denote by $\pi\colon \ker\psi\rightarrow \ker\psi/\im\phi$ the quotient map.

\begin{theorem} \label{thm_cohomology_presentation}
The following exact sequence 
\begin{equation*}
\begin{tikzcd}
0 & \ker\psi / \im\phi \arrow[l] &[5pt] \ker(g_0\minus r) \arrow[l,swap,"\pi\circ \kappa"] &[35pt] P_0\times Q_1 \arrow[l,swap,"\begin{pmatrix}f_0 \hspace{4pt} \minus q \\ \hspace{2pt} 0 \hspace{6pt} \minus g_1\end{pmatrix}"]
\end{tikzcd} 
\end{equation*}
is a presentation of $\ker\psi/\im\phi$.
\end{theorem}

\begin{proof}
\setstretch{1.2}
\begin{enumerate}
\item Since $\kappa$ is an epimorphism by Proposition
\ref{presentation_kernel} and $\pi$ is an epimorphism as a quotient map, $\pi\circ \kappa$ is also an epimorphism.
\item Let $(x,y)\in \ker(\pi\circ \kappa)$ i.e.\ $\pi\circ \kappa(x,y)=0$ and $\kappa(x,y)\in\ker\pi=\im\phi$. Hence, $\exists a\in L$ s.t.\ $\phi(a)=\kappa(x,y)$ and, moreover, $\exists b\in P_0$ s.t.\ $\lambda(b)=a$. By Proposition
\ref{kernel_lemma}, we get $\phi\circ \lambda(b)=\kappa\circ \begin{pmatrix}f_0 \\ 0\end{pmatrix}(b)=\kappa(x,y)$ and $\begin{pmatrix}f_0 \\ 0\end{pmatrix}(b)-(x,y)\in\ker \kappa=\im\begin{pmatrix}q \\ g_1\end{pmatrix}$. Thus, $\exists z\in Q_1$ s.t.\ $\begin{pmatrix}q \\ g_1\end{pmatrix}(z)=\begin{pmatrix}f_0 \\ 0\end{pmatrix}(b)-(x,y)$ and $(x,y)=\begin{pmatrix}f_0 \\ 0\end{pmatrix}(b)-\begin{pmatrix}q \\ g_1\end{pmatrix}(z)=\begin{pmatrix}f_0 & \minus q \\ 0 & \minus g_1 \end{pmatrix}(b,z)$. Therefore, $(x,y)\in\im\begin{pmatrix}f_0 & \minus q \\ 0 & \minus g_1 \end{pmatrix}$ and $\ker \pi\circ\kappa\subseteq \im\begin{pmatrix}f_0 & \minus q \\ 0 & \minus g_1 \end{pmatrix}$.

Conversely, let $(x,y)\in\im\begin{pmatrix}f_0 & \minus q \\ 0 & \minus g_1 \end{pmatrix}$ i.e.\ $(x,y)=\begin{pmatrix}f_0 & \minus q \\ 0 & \minus g_1 \end{pmatrix}(a,b)$ for some $(a,b)\in P_0\times Q_1$. We get $\kappa(x,y)=\kappa(f_0(a),0)-\kappa(q(b),g_1(b))=\mu\circ f_0(a)-\mu\circ q(b)=\phi\circ \lambda(a)\in\im\phi$. Hence, $\pi\circ \kappa(x,y)=0$, $(x,y)\in\ker\pi\circ \kappa$ and $\ker \pi\circ\kappa=\im\begin{pmatrix}f_0 & \minus q \\ 0 & \minus g_1 \end{pmatrix}$.
\end{enumerate}
\end{proof}

\noindent
Theorem \ref{thm_cohomology_presentation} leads us to the following algorithm \textbf{PresHom} for computing the barcode of the persistence module or graded module $\ker\psi/\im\phi$: 

\vspace{0.1in}
\noindent \textbf{Input:} A complex of presentations as in \eqref{comp_presentations} given by the matrices $p,q,r,f_0,f_1,g_0,g_1$. 

\noindent
\textbf{Step 1:} Build the matrices $\begin{pmatrix}g_0 & \minus r\end{pmatrix}$ and $\begin{pmatrix}f_0 & \minus q \\ 0 & \minus g_1\end{pmatrix}$ and order the columns by degree. 

\noindent
\textbf{Step 2:} Column reduce $\begin{pmatrix}g_0 & \minus r\end{pmatrix}$ from left to right. The generators corresponding to the zero columns of the reduced matrix span $\ker \begin{pmatrix}g_0 & \minus r\end{pmatrix}$.

\noindent
\textbf{Step 3:} Remove the rows of $\begin{pmatrix}f_0 & \minus q \\ 0 & \minus g_1\end{pmatrix}$ corresponding to non-zero columns of the reduced matrix $\begin{pmatrix}g_0 & \minus r\end{pmatrix}$. 

\noindent
\textbf{Step 4:} Column reduce the resulting matrix from left to right.

\noindent
\textbf{Output:} The barcode can be read off from the final matrix $\begin{pmatrix}f_0 & \minus q \\ 0 & \minus g_1\end{pmatrix}$ in the following way: If the row $r_j$ is a pivot row with pivot in column $c_{j_l}$, output the bar $[\text{deg}(r_j),\text{deg}(c_{j_l})\big)$ otherwise output the bar $[\text{deg}(r_j),\infty\big)$.

\begin{theorem}
The algorithm \textbf{PresHom} computes the barcode of $\ker\psi/\im\phi$. If the matrices $p,q,r,f_0,f_1,g_0,g_1$ are of size $O(n)\times O(n)$ the algorithm takes $O(n^3)$ time. 
\label{thm:perscohom}
\end{theorem}

\begin{proof}
From the input we can construct the matrices $\begin{pmatrix}g_0 & \minus r\end{pmatrix}$ and $\begin{pmatrix}f_0 & \minus q \\ 0 & \minus g_1 \end{pmatrix}$ and order the columns by degree of the generators. After a column reduction from left to right, the zero columns of the reduced matrix span $\ker\begin{pmatrix}g_0 & \minus r\end{pmatrix}$. Removing the rows of  $\begin{pmatrix}f_0 & \minus q \\ 0 & \minus g_1 \end{pmatrix}$ corresponding to the non-zero columns of the reduced matrix yields 
\begin{equation} \label{appendix_presentation_matrix_cohomology}
\ker \begin{pmatrix}g_0 & \minus r\end{pmatrix} \xleftarrow{\begin{pmatrix}f_0 & \minus q \\ 0 & \minus g_1 \end{pmatrix}} P_0\times Q_1
\end{equation}
Note that we do not need any transformations on $\begin{pmatrix}f_0 & \minus q \\ 0 & \minus g_1 \end{pmatrix}$. When we perform the change of basis, we never add a zero column of $\ker\begin{pmatrix}g_0 & \minus r\end{pmatrix}$ to another column. Therefore, in the other matrix there is never a row added to a row corresponding to a zero column. Hence the only rows we care about stay unchanged.

Theorem~\ref{thm_cohomology_presentation} shows that \eqref{appendix_presentation_matrix_cohomology} is a presentation of the homology $\ker\psi/\im\phi$. After column reducing the presentation matrix in \eqref{appendix_presentation_matrix_cohomology} from left to right, every column is either zero or has a unique pivot. We could now row reduce the presentation matrix and bring it to canonical form. After this step we could read off the elementary summands defining the barcode. Every row $r_j$ corresponds to a bar that is born at $\text{deg}(r_j)$. If it is a pivot row with pivot in column $c_{j_l}$ the bar dies at index $\text{deg}(c_{j_l})$. If it is a zero row the bar lives forever. Since the row reduction does not change the pivots and the bars only depend on the pivots, we can avoid the row reduction and read of the bars directly. Hence, after the column reduction every pivot row contributes a bar $[\text{deg}(r_j),\text{deg}(c_{j_l})\big)$ and every non-pivot row contributes a bar $[\text{deg}(r_j),\infty\big)$. All required matrix reductions can be done in time $O(n^3)$.   
\end{proof}

\noindent
If we have a complex of free graded modules, then $P_1$,$Q_1$,$R_1$ and the morphisms $p,q,r$ in \eqref{comp_presentations} can be chosen as zero. In this case, the algorithm \textbf{PresHom} described above reduces to the persistence algorithm as described in \cite{ZC05}. 

\begin{example} \label{homology_algo_example}
Consider the following instance of the complex of presentations in \eqref{comp_presentations} 
\begin{equation*}
\begin{tikzcd}[ampersand replacement=\&,every label/.append style = {font = \tiny},column sep=large, row sep=large]
0 \& L \arrow[l] \arrow[d,swap,"\phi"] \&[10pt] P_0 \arrow[l] \arrow{d}[swap,xshift=-2pt,yshift=-3pt]{\begin{blockarray}{c@{\hspace{7pt}}ccc}
& 2 & 1 & 1 \\
\begin{block}{c@{\hspace{7pt}}(ccc)}
0 & 1 & 1 & 0 \\
0 & 0 & 1 & 1 \\
1 & 1 & 0 & 1 \\
\end{block}
\end{blockarray}} \&[20pt] P_1 \arrow{l}[swap,xshift=-2pt,yshift=-7pt]{\begin{blockarray}{c@{\hspace{7pt}}ccc}
& 6 & 5 & 7 \\
\begin{block}{c@{\hspace{7pt}}(ccc)}
2 & 1 & 0 & 0 \\
1 & 0 & 1 & 0 \\
1 & 0 & 0 & 1 \\
\end{block}
\end{blockarray}} \arrow{d}[xshift=-2pt,yshift=-3pt]{\begin{blockarray}{c@{\hspace{7pt}}ccc}
& 6 & 5 & 7 \\
\begin{block}{c@{\hspace{7pt}}(ccc)}
5 & 1 & 1 & 0 \\
5 & 0 & 1 & 1 \\
6 & 1 & 0 & 1 \\
\end{block}
\end{blockarray}} \&[10pt] 0 \arrow[l] \\[7pt]
0 \& M \arrow[l] \arrow[d,swap,,"\psi"] \& Q_0 \arrow[l] \arrow{d}[swap,xshift=-2pt,yshift=-3pt]{\begin{blockarray}{c@{\hspace{5pt}}ccc}
& 0 & 0 & 1 \\
\begin{block}{c@{\hspace{5pt}}(ccc)}
0 & 1 & 1 & 1 \\
\end{block}
\end{blockarray}} \& Q_1 \arrow{l}[swap,xshift=-2pt,yshift=-7pt]{\begin{blockarray}{c@{\hspace{7pt}}ccc}
& 5 & 5 & 6 \\
\begin{block}{c@{\hspace{7pt}}(ccc)}
0 & 1 & 0 & 0 \\
0 & 0 & 1 & 0 \\
1 & 0 & 0 & 1 \\
\end{block}
\end{blockarray}} \arrow{d}[xshift=-2pt,yshift=-3pt]{\begin{blockarray}{c@{\hspace{5pt}}ccc}
& 5 & 5 & 6 \\
\begin{block}{c@{\hspace{5pt}}(ccc)}
3 & 1 & 1 & 1 \\
\end{block}
\end{blockarray}} \& 0 \arrow[l] \\[3pt]
0 \& N \arrow[l] \& R_0 \arrow[l] \& R_1 \arrow{l}[swap,xshift=-2pt,yshift=-7pt]
{\begin{blockarray}{c@{\hspace{4pt}}c}
& 3 \\
\begin{block}{c@{\hspace{4pt}}(c)}
0 & 1  \\
\end{block}
\end{blockarray}}\& 0 \arrow[l]
\end{tikzcd}
\end{equation*}
where the matrices are labeled by column and row degrees and have coefficients in $\mathbb{Z}_2$. We build the matrices
\begin{equation*} \footnotesize
\begin{pmatrix}g_0 & \minus r\end{pmatrix}=\begin{blockarray}{ccccc}
& 0 & 0 & 1 & 3 \\
\begin{block}{c(cccc)}
0 & 1 & 1 & 1 & 1 \\
\end{block}
\end{blockarray}
\hspace{2pt}, \hspace{5pt}
\begin{pmatrix}f_0 & \minus q \\ 0 & \minus g_1\end{pmatrix}=\begin{blockarray}{ccccccc}
& 2 & 1 & 1 & 5 & 5 & 6 \\
\begin{block}{c(cccccc)}
0 & 1 & 1 & 0 & 1 & 0 & 0 \\
0 & 0 & 1 & 1 & 0 & 1 & 0 \\
1 & 1 & 0 & 1 & 0 & 0 & 1 \\
3 & 0 & 0 & 0 & 1 & 1 & 1 \\
\end{block}
\end{blockarray}
\end{equation*}
column reduce $(g_0 \hspace{2pt} \minus r)$, delete the rows corresponding to non-zero columns from $\begin{pmatrix}f_0 & \minus q \\ 0 & \minus g_1\end{pmatrix}$ and order its columns by degree
\begin{equation*} \footnotesize
\begin{pmatrix}g_0 & \minus r\end{pmatrix}=\begin{blockarray}{ccccc}
& 0 & 0 & 1 & 3 \\
\begin{block}{c(cccc)}
0 & 1 & 0 & 0 & 0 \\
\end{block}
\end{blockarray}
\hspace{2pt}, \hspace{5pt}
\begin{pmatrix}f_0 & \minus q \\ 0 & \minus g_1\end{pmatrix}=\begin{blockarray}{ccccccc}
& 1 & 1 & 2 & 5 & 5 & 6 \\
\begin{block}{c(cccccc)}
0 & 1 & 1 & 0 & 0 & 1 & 0 \\
1 & 0 & 1 & 1 & 0 & 0 & 1 \\
3 & 0 & 0 & 0 & 1 & 1 & 1 \\
\end{block}
\end{blockarray}
\end{equation*}
Now we column reduce $\begin{pmatrix}f_0 & \minus q \\ 0 & \minus g_1\end{pmatrix}$ and obtain the reduced matrix 
\begin{equation} \label{example_shv_computation} \footnotesize
\begin{pmatrix}f_0 & \minus q \\ 0 & \minus g_1\end{pmatrix}=\begin{blockarray}{ccccccc}
& 1 & 1 & 2 & 5 & 5 & 6 \\
\begin{block}{c(cccccc)}
0 & 1 & 1 & 0 & 0 & 0 & 0 \\
1 & 0 & 1 & 0 & 0 & 0 & 0 \\
3 & 0 & 0 & 0 & 1 & 0 & 0 \\
\end{block}
\end{blockarray}
\end{equation}
Finally we read off the barcode of $\ker\psi /\im\phi$ from \eqref{example_shv_computation} in the following way: Row $1$ contributes the bar $[0,1)$, row $2$ contributes the empty bar $[1,1)$, and row $3$ contributes the bar $[3,5)$.   
\end{example}


\section{Computing presentations of morphisms of persistence modules} \label{sec_presentation_algorithm}

The algorithm described in the previous section requires as input a complex of presentations of graded modules. In practice, when working with, for example, simplicial towers or persistent sheaves (see Appendix \ref{app_sheaves}) we cannot always assume that the input is given as a complex of presentations. Thus, to make use of our algorithm in various settings, we develop an efficient algorithm \textbf{PresPersMod} to compute a presentation of a morphism of persistence modules. Given a morphism of persistence modules $\phi\colon M\rightarrow N$ of finite type
\begin{equation} \label{morphism_persistence_modules}
\begin{tikzcd}
M_0 \arrow[r,"A_0"] \arrow[d,"C_0"] & M_1 \arrow[r,"A_1"] \arrow[d,"C_1"] & \cdots \arrow[r,"A_{m\minus 1}"] &[5pt] M_m \arrow[d,"C_m"] \arrow[r,"\cong"] & \cdots \\
N_0 \arrow[r,"B_0"] & N_1 \arrow[r,"B_1"] & \cdots \arrow[r,"B_{m\minus 1}"] & N_m \arrow[r,"\cong"] & \cdots
\end{tikzcd}
\end{equation}
where we use the notation $M_i\coloneqq M(i)$, $N_i\coloneqq N(i)$ and the morphisms $M(i\leq i+1)$, $N(i\leq i+1)$ and $\phi(i)$ are input by the respective matrices $A_{i}$, $B_{i}$ and $C_i$. Our goal is to compute a canonical presentation 
\begin{equation}  \label{morphism_presentations}
\begin{tikzcd}
P_0 \arrow[d,swap,"f_0"] & \arrow[l,swap,"p"] P_1 \arrow[d,"f_1"] \\
Q_0 & \arrow[l,swap,"q"] Q_1
\end{tikzcd}
\end{equation}
presenting a morphism of persistence modules isomorphic to \eqref{morphism_persistence_modules}. As explained in Section \ref{sec_background}, it is enough to compute the matrix $f_0$ where each row and column additionally stores the degree of the corresponding generator $b(-)$ and its relation $d(-)$. Each generator-relation pair in a canonical presentation corresponds to an interval summand in a persistence module isomorphic to the top or bottom row of \eqref{morphism_persistence_modules} . Hence, an alternative point of view is that the algorithm computes the barcodes of \eqref{morphism_persistence_modules} while keeping track of how the bars map to each other (cf.\ \eqref{hom_pmod} and \eqref{hom_sum}). The algorithm \textbf{PresPersMod} maintains a canonical presentation 
\begin{equation*} \label{ith_presentation}
\begin{tikzcd}
P_0^i \arrow[d,swap,"f_0^i"] & \arrow[l,swap,"p^i"] P_1^i \arrow[d,"f_1^i"] \\
Q_0^i & \arrow[l,swap,"q^i"] Q_1^i
\end{tikzcd}
\end{equation*}
of a morphism of persistence modules obtained by restricting the original modules up to index $i$
\begin{equation} \label{ith_restriction}
\begin{tikzcd}
M_0 \arrow[r,"A_0"] \arrow[d,"C_0"] & \cdots \arrow[r,"A_{i\minus 1}"] & M_i \arrow[d,"C_i"] \arrow[r,"\text{id}"] & M_i \arrow[r,"\text{id}"] \arrow[d,"C_i"] & \cdots  \\
N_0 \arrow[r,"B_0"]  & \cdots \arrow[r,"B_{i\minus 1}"] & N_i \arrow[r,"\text{id}"] & N_i \arrow[r,"\text{id}"] & \cdots
\end{tikzcd}
\end{equation}
while processing \eqref{morphism_persistence_modules} from left $i=0$ to right $i=m$. Iteratively, we build the matrix $f^i_0$ with the birth- and death-time annotations. In a generic step, when we arrive at $i$ from $i-1$, the matrices
$A_i$, $B_i$, and $C_i$ are already transformed to $A_i^t$, $B_i^t$, and
$C_i^t$ respectively. We reduce these matrices further to $A_i^r$, $B_i^r$, and $C_i^r$ which induce transformations of $A_{i+1}$, $B_{i+1}$, and $C_{i+1}$
as we go forward in \textbf{PresPersMod} described below.

\vspace{0.1in}
\textbf{Initialization at $i=0$:} We start with the matrices $A_0$, $B_0$ and $C_0$. Each column and row of $C_0$ corresponds to a basis element of $M_0$ and $N_0$. Since we are at index $0$, each of these basis elements has to start a new bar and they are mapped according to the matrix $C_0$. Therefore, we set $f^0_0:=C_0$ with $b(c_j)=0$, $d(c_j)=\infty$ and $b(r_k)=0$, $d(r_k)=\infty$ for every column $c_j$ and every row $r_k$ of $C_0$. Each bar at index $0$ now corresponds to a generator in $P^0_0$ or $Q^0_0$ and their maps are determined by the matrix of generators $f^0_0$.
 Initially, we set $A_0^t:=A_0$, $B_0^t:=B_0$ and $C_0^t:=C_0$ respectively.

\vspace{0.1in}
\textbf{Moving from $i$ to $i+1$:} Given the matrix $f^i_0$ and the matrices $A^t_i$ and $B^t_i$ from the previous step. Each column $c_j$ and row $r_k$ of $f^i_0$ corresponds to a bar in persistence modules isomorphic to \eqref{ith_restriction}. The bars with $d(c_j),d(r_k)<\infty$ already finished (died) and need no further processing. For the bars with $d(c_j),d(r_k)=\infty$ we need to determine which of them die entering index $i+1$. Moreover, we need to determine how many new bars are born at index $i+1$ and how they are mapped. 

Each column of $A^t_i$ and $B^t_i$ corresponds to a column and row of $f^i_0$ respectively and they are ordered w.r.t.\ their order in $f^i_0$. Column reduce $A^t_i$ and $B^t_i$ from left to right to have reduced matrices $A_i^r$ and $B_i^r$ respectively while performing the corresponding column and row operations on $f^i_0$ to obtain the new matrix $\overline{f}^i_0$.  If $a_j$ is a zero column in $A^r_i$ and $\overline{c}_{i_j}$ is the corresponding column of $\overline{f}^i_0$, we set $d(\overline{c}_{i_j})=i+1$. We proceed in the same way for the columns of $B^r_i$ and the rows of $\overline{f}^i_0$. 

Every row of $A^r_i$ or $B^r_i$ that is not a pivot row could be zeroed out by row reductions. Hence, after a change of basis, every non-pivot row of $A^r_i$ or $B^r_i$ is a generator of the respective cokernel  and thus generates a new bar born at index $i+1$. We do not perform the row reduction on $A^r_i$ or $B^r_i$ but we transform $A_{i+1}$, $B_{i+1}$ and $C_{i+1}$ according to this change of bases to maintain matrices that are consistent with the bases induced by our presentation. These transformed matrices constitute $A^t_{i+1}$, $B^t_{i+1}$ and $C^t_{i+1}$ 
for the next iteration.
Observe that every column of $A^t_{i+1}$ and $B^t_{i+1}$ and every column and row of $C^t_{i+1}$ corresponds either to a non-pivot row of $A^r_i$ or $B^r_i$ and thus to a new born bar at index $i+1$ or to a pivot row with pivot in the $j$-th column and thus to the bar of the $j$-th column in $A^r_i$ or $B^r_i$.  We sort the columns of $A^t_{i+1}$ and $B^t_{i+1}$ and the columns and rows of $C^t_{i+1}$ w.r.t.\ their order in $\overline{f}^i_0$ where columns and rows for new generators are added at the end in arbitrary order. The matrix $C^t_{i+1}$, restricted to the columns corresponding to non-pivot rows of $A^r_i$ (new generators), is added to $\overline{f}^i_0$ to obtain $f^{i+1}_0:=\overline{f}^i_0$. For each new column $g_{j}$ and row $h_{k}$ we set $b(g_{j})=i+1$, $d(g_{j})=\infty$ and $b(h_{k})=i+1$, $d(h_{k})=\infty$. 

\begin{theorem}
Given a morphism as in~\eqref{morphism_persistence_modules}, the algorithm \textbf{PresPersMod} computes a canonical presentation~\eqref{morphism_presentations} in $O(n^3)$ time where $\text{dim}(M_i),\text{dim}(N_i)=O(n_i)$ and $n=\sum_{i=0}^m n_i$.
\label{thm:bar-from-matrices}
\end{theorem}

\begin{proof}
See Appendix \ref{app_presentation_algo}.
\end{proof}

\begin{example}
Consider the following morphism of persistence modules over $\mathbb{Z}_2$
\begin{equation} \label{exmp:prespersmod}
\begin{tikzcd}[every label/.append style = {font = \tiny}]
M_0 \arrow[r,"\begin{pmatrix}0 \hspace{1pt} 1 \hspace{1pt} 0 \\ 1 \hspace{1pt} 1 \hspace{1pt} 1 \\ 1 \hspace{1pt} 1 \hspace{1pt} 1 \end{pmatrix}"] \arrow[d,"\begin{pmatrix}0 \hspace{1pt} 1 \hspace{1pt} 1 \\ 0 \hspace{1pt} 1 \hspace{1pt} 1 \\ 1 \hspace{1pt} 0 \hspace{1pt} 0 \end{pmatrix}"] & M_1 \arrow[r,"\begin{pmatrix}0 \hspace{1pt} 1 \hspace{1pt} 1 \\ 0 \hspace{1pt} 0 \hspace{1pt} 0 \\ 0 \hspace{1pt} 1 \hspace{1pt} 1 \end{pmatrix}"] \arrow[d,"\begin{pmatrix}0 \hspace{1pt} 0 \hspace{1pt} 0 \\ 0 \hspace{1pt} 1 \hspace{1pt} 1 \\ 0 \hspace{1pt} 0 \hspace{1pt} 0 \end{pmatrix}"] & M_2 \arrow[d,"\begin{pmatrix}1 \hspace{1pt} 0 \hspace{1pt} 0 \\ 1 \hspace{1pt} 0 \hspace{1pt} 0 \\ 0 \hspace{1pt} 1 \hspace{1pt} 0 \end{pmatrix}"] \arrow[r,"\cong"] & \cdots \\
N_0 \arrow[r,swap,"\begin{pmatrix}1 \hspace{1pt} 1 \hspace{1pt} 0 \\ 1 \hspace{1pt} 1 \hspace{1pt} 0 \\ 0 \hspace{1pt} 0 \hspace{1pt} 0 \end{pmatrix}"] & N_1 \arrow[r,swap,"\begin{pmatrix}1 \hspace{1pt} 1 \hspace{1pt} 1 \\ 0 \hspace{1pt} 1 \hspace{1pt} 1 \\ 0 \hspace{1pt} 0 \hspace{1pt} 1 \end{pmatrix}"] & N_2 \arrow[r,swap,"\cong"] & \cdots
\end{tikzcd}
\end{equation}
We initialize $f^0_0=C_0$ and set the birth- and death-time of every generator to $0$ and $\infty$: 
\begin{equation*} \scriptsize
f^0_0=\begin{blockarray}{cccc}
& c_1 & c_2 & c_3 \\
\begin{block}{c(ccc)}
r_1 & 0 & 1 & 1 \\
r_2 & 0 & 1 & 1 \\
r_3 & 1 & 0 & 0 \\
\end{block}
\end{blockarray}
\hspace{1pt}, \hspace{5pt}
\begin{blockarray}{ccccccc}
& c_1 & c_2 & c_3 & r_1 & r_2 & r_3  \\
\begin{block}{c(cccccc)}
b & 0 & 0 & 0 & 0 & 0 & 0 \\
d & \infty & \infty & \infty & \infty & \infty & \infty \\
\end{block}
\end{blockarray}
\end{equation*}
We column reduce $A_0$ and $B_0$ while performing the corresponding operations on $f^0_0$: To reduce $A_0$ we add column $c_1$ to column $c_2$ and $c_3$ and obtain the reduced matrix $A^r_0$. We also have to add column $c_1$ to column $c_2$ and $c_3$ in $f_0^0$. To reduce $B_0$ we add column $r_1$ to column $r_2$ and obtain the reduced matrix $B^r_0$. Now we have to add row $r_2$ to $r_1$ in $f_0^0$ to obtain $\overline{f}_0^0$.  
\begin{equation*} \scriptsize
A^r_0=\begin{blockarray}{ccc}
c_1 & c_2 & c_3 \\
\begin{block}{(ccc)}
0 & 1 & 0 \\
1 & 0 & 0 \\
1 & 0 & 0 \\
\end{block}
\end{blockarray}
\hspace{1pt}, \hspace{5pt}
B^r_0=\begin{blockarray}{ccc}
r_1 & r_2 & r_3 \\
\begin{block}{(ccc)}
1 & 0 & 0 \\
1 & 0 & 0 \\
0 & 0 & 0 \\
\end{block}
\end{blockarray}
\hspace{1pt}, \hspace{5pt}
\overline{f}^0_0=\begin{blockarray}{cccc}
& c_1 & c_2 & c_3 \\
\begin{block}{c(ccc)}
r_1 & 0 & 0 & 0 \\
r_2 & 0 & 1 & 1 \\
r_3 & 1 & 1 & 1 \\
\end{block}
\end{blockarray}
\end{equation*}
We perform the basis transformation corresponding to a row reduction of $A^r_0$ and $B^r_0$ on $A_1$, $B_1$ and $C_1$: To row reduce $A^r_0$ we would have to add the third row to the second one. Hence, we add the second column to the third one in $A_1$ and $C_1$. To row reduce $B^r_0$ we would have to add the second row to the first one. Hence, we add the first column of $B_1$ to the second one and the second row of $C_1$ to the first one. Finally we reorder the columns of $A_1$ and $B_1$ and the columns and rows of $C_1$ such that their order matches the order of the generators of $\overline{f}^0_0$. After row reducing $A^r_0$ and $B^r_0$, the second row of $A^r_0$ and the first and third row of $B^r_0$ would be zero. Thus, these rows span the respective cokernel and the corresponding columns of $A_1$ and $B_1$ start new bars and get sorted in at the end. 
\begin{equation*} \scriptsize
A^t_1=\begin{blockarray}{ccc}
c_1 & c_2 & c_4 \\
\begin{block}{(ccc)}
0 & 0 & 1 \\
0 & 0 & 0 \\
0 & 0 & 1 \\
\end{block}
\end{blockarray}
\hspace{1pt}, \hspace{5pt}
B^t_1=\begin{blockarray}{ccc}
r_1 & r_4 & r_5 \\
\begin{block}{(ccc)}
0 & 1 & 1 \\
1 & 0 & 1 \\
0 & 0 & 1 \\
\end{block}
\end{blockarray}
\hspace{1pt}, \hspace{5pt}
C^t_1=\begin{blockarray}{cccc}
& c_1 & c_2 & c_4 \\
\begin{block}{c(ccc)}
r_1 & 0 & 0 & 1 \\
r_4 & 0 & 0 & 1 \\
r_5 & 0 & 0 & 0 \\
\end{block}
\end{blockarray}
\end{equation*}
We update $\overline{f}^0_0$ to $f^1_0$ as follows: The column $c_4$ and the rows $r_4$ and $r_5$ of $C^t_1$ correspond to non-pivot rows of $A^r_0$ or $B^r_0$. We add them as new columns and rows of $f^1_0$ born at index $1$. The columns $c_3$ of $A^r_0$ and $r_2,r_3$ of $B^r_0$ are zero-columns, thus we set $d(c_3),d(r_2),d(r_3)=1$:
\begin{equation*} \scriptsize
f^1_0=\begin{blockarray}{ccccc}
& c_1 & c_2 & c_3 & c_4 \\
\begin{block}{c(cccc)}
r_1 & 0 & 0 & 0 & 1 \\
r_2 & 0 & 1 & 1 & 0 \\
r_3 & 1 & 1 & 1 & 0 \\
r_4 & 0 & 0 & 0 & 1 \\
r_5 & 0 & 0 & 0 & 0 \\
\end{block}
\end{blockarray}
\hspace{1pt}, \hspace{5pt}
\begin{blockarray}{cccccccccc}
& c_1 & c_2 & c_3 & c4 & r_1 & r_2 & r_3 & r_4 & r_5 \\
\begin{block}{c(ccccccccc)}
b & 0 & 0 & 0 & 1 & 0 & 0 & 0 & 1 & 1 \\
d & \infty & \infty & 1 & \infty & \infty & 1 & 1 & \infty & \infty \\
\end{block}
\end{blockarray}
\end{equation*}
Since $A^t_1$ and $B^t_1$ are already column reduced we can directly transform $C_2$: To row reduce $A^t_1$ we would have to add the third row to the first one. Thus, we add the first column of $C_2$ to the third one. To row reduce $B^t_1$ we would have to add the third row to the first one and the second one. Hence, we add the third row of $C_2$ to the first and the second one. Finally we reorder the columns and rows of $C_2$ according to the order of the corresponding generators in $f_0^1$ where new generators corresponding to non-pivot rows of $A^t_1$ and $B^t_1$ are sorted in at the end.
\begin{equation*}\scriptsize
C^t_2=\begin{blockarray}{cccc}
& c_4 & c_5 & c_6 \\
\begin{block}{c(ccc)}
r_1 & 1 & 1 & 1 \\
r_4 & 1 & 1 & 1 \\
r_5 & 0 & 1 & 0 \\
\end{block}
\end{blockarray}
\end{equation*}
We update $\overline{f}^1_0:=f^1_0$ by adding columns $c_5$, $c_6$ of $C^t_2$, corresponding to non-pivot rows of $A^t_1$, with $b(c_5),b(c_6)=2$ and $d(c_5),d(c_6)=\infty$. Since $c_1$ and $c_2$ are zero-columns of $A^t_1$, we set $d(c_1),d(c_2)=2$ and obtain the final result: 
\begin{equation*} \scriptsize
f^2_0=\begin{blockarray}{ccccccc}
& c_1 & c_2 & c_3 & c_4 & c_5 & c_6 \\
\begin{block}{c(cccccc)}
r_1 & 0 & 0 & 0 & 1 & 1 & 1 \\
r_2 & 0 & 1 & 1 & 0 & 0 & 0 \\
r_3 & 1 & 1 & 1 & 0 & 0 & 0 \\
r_4 & 0 & 0 & 0 & 1 & 1 & 1 \\
r_5 & 0 & 0 & 0 & 0 & 1 & 0 \\
\end{block}
\end{blockarray}
\hspace{1pt}, \hspace{2pt}
\begin{blockarray}{cccccccccccc}
& c_1 & c_2 & c_3 & c4 & c_5 & c_6 & r_1 & r_2 & r_3 & r_4 & r_5 \\
\begin{block}{c(ccccccccccc)}
b & 0 & 0 & 0 & 1 & 2 & 2 &  0 & 0 & 0 & 1 & 1 \\
d & 2 & 2 & 1 & \infty & \infty & \infty & \infty & 1 & 1 & \infty & \infty \\
\end{block}
\end{blockarray}
\end{equation*}
Note that the obtained presentation is equal to the presentation of \eqref{eq:concrete_morphism} in Example \ref{exmp:kcoef_pres}. This implies that the morphisms of persistence modules \eqref{exmp:prespersmod} is isomorphic to \eqref{eq:concrete_morphism}.
\end{example}


\section{Complexes of presentations} \label{sec_comp_presentations}

The algorithm introduced in the previous section allows us to compute a canonical presentation of a complex of persistence modules. To obtain a valid input for our homology algorithm the resulting presentations need to form a complex, i.e., a sequence of presentations connected by morphisms where composition of consecutive morphisms is zero. The problem is that an arbitrary presentation of a complex of graded modules is not necessarily a complex of presentations. Consider the following complex of persistence modules over $\mathbb{Z}_2$
\begin{equation*}
\mathbb{I}_{[1,2)}\xlongrightarrow{\phi} \mathbb{I}_{[0,2)}\xlongrightarrow{\psi} \mathbb{I}_{[0,1)}
\end{equation*}
where $\phi$ and $\psi$ are the identity map on the overlap of the intervals and $0$ elsewhere. Note that the composition is zero because the first and the third interval do not overlap. The following canonical presentation of the corresponding complex of $\field[t]$-modules
\begin{equation} \label{non_complex}
\begin{tikzcd}
0 & \arrow[l] \mathbb{I}_{[1,2)} \arrow[d,"\phi"] & \arrow[l] \mathbb{I}_{[1,\infty)} \arrow[d,"t"] & \arrow[l,swap,"t"] \mathbb{I}_{[2,\infty)} \arrow[d,"1"] & \arrow[l] 0 \\
0 & \arrow[l] \mathbb{I}_{[0,2)} \arrow[d,"\psi"] & \arrow[l] \mathbb{I}_{[0,\infty)} \arrow[d,"1"] & \arrow[l,swap,"t^2"] \mathbb{I}_{[2,\infty)} \arrow[d,"t"] & \arrow[l] 0 \\
0 & \arrow[l] \mathbb{I}_{[0,1)} & \arrow[l] \mathbb{I}_{[0,\infty)} & \arrow[l,swap,"t"] \mathbb{I}_{[1,\infty)} & \arrow[l] 0 
\end{tikzcd}
\end{equation}
is not a complex of presentations because $1\circ t\not=0$ and $t\circ 1\neq 0$. The following propositions help us to show that we can always modify a canonical presentation of a complex of graded modules to a complex of presentations. 

\begin{proposition} \label{appendix_prop_uniqueness_morphism}
Given a commutative diagram with exact rows of the form 
\begin{equation*}
\begin{tikzcd}
0 & \mathbb{I}_{[a,b)} \arrow[l] \arrow[d,"\phi"] & \mathbb{I}_{[a,\infty)} \arrow[l,swap,"\cdot 1"] \arrow[d,"f_0"] & \mathbb{I}_{[b,\infty)} \arrow[l,swap,"\cdot t^{b\minus a}"] \arrow[d,"f_1"] & 0 \arrow[l] \\
0 & \mathbb{I}_{[c,d)} \arrow[l] & \mathbb{I}_{[c,\infty)} \arrow[l,swap,"\cdot 1"] & \mathbb{I}_{[d,\infty)} \arrow[l,swap,"\cdot t^{d\minus c}"] & 0 \arrow[l] 
\end{tikzcd}  
\end{equation*}
such that $a<b$ and $c<d$. There are three cases:
\begin{enumerate}
    \item If $c\leq a<d\leq b$, then $(\phi=0 \iff f_0=0 \iff f_1=0)$.
    \item If $c<d\leq a<b$, then $\phi=0$ and $(f_0=0 \iff f_1=0)$.
    \item Else, $\phi=f_0=f_1=0$.
\end{enumerate}
\end{proposition}

\begin{proof}
As defined in \eqref{hom_pmod}, if $c\leq a<d\leq b$, there is a unique non-zero morphism, up to multiplication with a scalar, between two interval modules. We denote this morphism w.r.t.\ the unit $1$ of the field $\field$ by $(\cdot 1)$ and $(\cdot t^r)$. 
\begin{enumerate}
\item Since $a<b$, $c\leq a$, $d\leq b$ and $c<d$, we have $(\cdot 1)\circ f_0=0$ and $f_0\circ (\cdot t^{b\minus a})=0$ if and only if $f_0=0$ and $(\cdot t^{d\minus c})\circ f_1=0$ if and only if $f_1=0$. Moreover, since $a<d$, we have $\phi\circ (\cdot 1)=0$ if and only if $\phi=0$. 
\item Since $d\leq a$, we get $\phi=0$. Similiar to (1), we obtain $f_0=0 \iff f_1=0$.
\item If we are not in case (1), then $\phi=0$. If $a<c$, then $f_0=0$. If $f_1\neq 0$, then $d\leq b$ and $(\cdot t^{d\minus c})\circ f_1\neq 0=f_0\circ (\cdot t^{b\minus a})$. A contradiction. If $b<d$, then $f_1=0$. If $f_0\neq 0$, then $c\leq a$ and $f_0\circ (\cdot t^{b\minus a})\neq 0=(\cdot t^{d\minus c})\circ f_1$. A contradiction. If $c\leq a$ and $d\leq b$ 
we are in case (1) or (2).
\end{enumerate}
\end{proof}

\begin{proposition} \label{lemma1}
Consider a commutative diagram of canonical presentations of the form: 
\begin{equation}
\begin{tikzcd}
0 & \mathbb{I}_{[a,b)} \arrow[l] \arrow[d,"\iota"] & \mathbb{I}_{[a,\infty)} \arrow[l] \arrow[d,"\iota_0"] & [5pt] \mathbb{I}_{[b,\infty)} \arrow[l,swap,"\cdot t^{b\minus a}"] \arrow[d,"\iota_1"] & 0 \arrow[l] \\
0 & M \arrow[l] \arrow[d,"\phi"] & P_0 \arrow[l,swap,"\mu"] \arrow[d,"f_0"] & P_1 \arrow[l,swap,"p"] \arrow[d,"f_1"] & 0 \arrow[l] \\
0 & N \arrow[l] \arrow[d,"\pi"] & Q_0 \arrow[l,swap,"\nu"] \arrow[d,"\pi_0"] & Q_1 \arrow[l,swap,"q"] \arrow[d,"\pi_1"] & 0 \arrow[l] \\
0 & \mathbb{I}_{[c,d)} \arrow[l] & \mathbb{I}_{[c,\infty)} \arrow[l] & \mathbb{I}_{[d,\infty)} \arrow[l,swap,"\cdot t^{d\minus c}"] & 0 \arrow[l]
\end{tikzcd}  
\end{equation}
If not $(c<d\leq a<b)$, then $(\pi\circ\phi\circ\iota=0 \iff \pi_0\circ f_0\circ \iota_0=0 \iff \pi_1\circ f_1\circ \iota_1=0)$.
\end{proposition}

\begin{proof}
By Proposition \ref{appendix_prop_uniqueness_morphism}, in case (1), we have $\pi\circ\phi\circ\iota=0\iff \pi_0\circ f_0\circ\iota_0=0\iff \pi_1\circ f_1\circ\iota_1=0$. Moreover, in case (3), we get $\pi\circ\phi\circ\iota=\pi_0\circ f_0\circ\iota_0=\pi_1\circ f_1\circ\iota_1=0$. 
\end{proof}

\noindent
Proposition~\ref{lemma1} implies that given canonical presentations as in \eqref{comp_presentations}, the only case where the entry of the matrix $g_0\circ f_0$, corresponding to the bars $\mathbb{I}_{[a,b)}$ and $\mathbb{I}_{[c,d)}$, can be non-zero is when $c<d\leq a<b$. In this case, we can add a new bar $\mathbb{I}_{[e,e)}$ of length zero with $d\leq e\leq a$ to the middle presentation in \eqref{comp_presentations}, and map $\mathbb{I}_{[a,\infty)}\rightarrow\mathbb{I}_{ [e,\infty)}$ and $\mathbb{I}_{[e,\infty)}\rightarrow \mathbb{I}_{[c,\infty)}$ in such a way that the resulting map $\mathbb{I}_{[a,b)}\rightarrow \mathbb{I}_{[c,d)}$ is zero.  Given the morphisms of presentations in \eqref{non_complex}, we add $\mathbb{I}_{[1,\infty)} \xleftarrow{1} \mathbb{I}_{[1,\infty)}$ to the middle term to obtain the following complex of presentations:
\begin{equation*} \label{complex_repaired}
\begin{tikzcd}
0 & \arrow[l] \mathbb{I}_{[1,2)} \arrow[d,"\phi"] & \arrow[l] \mathbb{I}_{[1,\infty)} \arrow[d,"\begin{pmatrix}t\\1\end{pmatrix}"] &[10pt] \arrow[l,swap,"t"] \mathbb{I}_{[2,\infty)} \arrow[d,"\begin{pmatrix}1\\t\end{pmatrix}"] & \arrow[l] 0 \\[12pt]
0 & \arrow[l] \mathbb{I}_{[0,2)} \arrow[d,"\psi"] & \arrow[l] \mathbb{I}_{[0,\infty)}\oplus \mathbb{I}_{[1,\infty)} \arrow[d,"\begin{pmatrix}1\hspace{4pt} t\end{pmatrix}"] & \arrow[l,swap,"\begin{pmatrix}t^2 \hspace{2pt} 0\\ 0 \hspace{4pt} 1\end{pmatrix}"] \mathbb{I}_{[2,\infty)}\oplus \mathbb{I}_{[1,\infty)} \arrow[d,"\begin{pmatrix}t\hspace{4pt} 1\end{pmatrix}"] & \arrow[l] 0 \\[10pt]
0 & \arrow[l] \mathbb{I}_{[0,1)} & \arrow[l] \mathbb{I}_{[0,\infty)} & \arrow[l,swap,"t"] \mathbb{I}_{[1,\infty)} & \arrow[l] 0 
\end{tikzcd}
\end{equation*}

\begin{theorem} \label{modification_thm}
Let $L\xlongrightarrow{\phi} M \xlongrightarrow{\psi} N$ be a complex of graded $\field[t]$-modules, then there exists a complex of presentations presenting $L\xlongrightarrow{\phi} M \xlongrightarrow{\psi} N$. 
\end{theorem}

\begin{proof}
See Appendix \ref{app_comp_presentations}.
\end{proof}

\begin{corollary}
Given a complex of graded $\field[t]$-modules $L\xlongrightarrow{\phi} M \xlongrightarrow{\psi} N$ and a canonical presentation as in \eqref{comp_presentations} of size $n\coloneqq\text{max}\{\text{rank}(P_0),\text{rank}(Q_0),\text{rank}(R_0)\}$, we can modify \eqref{comp_presentations} such that the modified presentation is a reduced complex of presentations of size $O(n)$.
\label{cor:preserve-rank}
\end{corollary}

\begin{proof}
See Appendix \ref{app_comp_presentations}.
\end{proof}

\noindent
The modification described in the proof of Corollary \ref{cor:preserve-rank} can be implemented in the following way: \\

\noindent
\textbf{Input:} The matrices $f_0$ and $g_0$ with birth/death annotations representing canonical presentations of $\phi$ and $\psi$. \\

\noindent
\textbf{Step 1:} Compute the matrix product $g_0\circ f_0$. \\

\noindent
\textbf{Step 2:} For every non-zero column $c_j$ of $g_0\circ f_0$, add a row $(0,\ldots,0,\underset{j}{1},0,\ldots,0)$ with birth and death index $b(c_j)$ to the end of $f_0$ and the column $-c_j$ with birth and death index $b(c_j)$ to the end of $g_0$. \\

\noindent
\textbf{Output:} The modified matrices $f_0$ and $g_0$. \\

\noindent
This modification takes $O(n^3)$ time.

\cancel{
The modification described in the proof of Corollary \ref{cor:preserve-rank} can be implemented in the following way: Compute the matrix product $g_0\circ f_0$. For every non-zero column $c_j$ of $g_0\circ f_0$ we add a row of degree $\text{deg}(c_j)$ to the end of $f_0$ with $1$ in column $j$ and $0$ everywhere else and the column $-c_j$ of degree $\text{deg}(c_j)$ to the end of $g_0$. Moreover, we add a column and row of degree $\text{deg}(c_j)$ to the end of $q$ with entry $1$ in the last column and row index and $0$ elsewhere. Since $f_1$ and $g_1$ are uniquely determined by the canonical form (with length zero bars) and $f_0$ and $g_0$ we can modify the matrices $f_1$ and $g_1$ accordingly. This modification takes $O(n^3)$ time.
}


\section{Applications} \label{sec_applications}

\subsection{Persistent homology of simplicial towers} \label{sec_application_tower}

We consider a simplicial tower:  
\begin{equation} \label{application_tower}
\begin{tikzcd}
\Vec{K}:\, K_0 \arrow[r,"f_0"] & K_1 \arrow[r,"f_1"] & \cdots \arrow[r,"f_{m\minus 1}"] &[2pt] K_m
\end{tikzcd}
\end{equation}
where $K_i$ is a finite simplicial complex and $f_i$ an arbitrary simplicial map. Our goal is to compute the barcode of the homology persistence module
\begin{equation} \label{tower_complex_homology}
\begin{tikzcd}[column sep=large]
H_k(\Vec{K}):\, H_k(K_0) \arrow[r,"H_k(f_0)"] & H_k(K_1) \arrow[r,"H_k(f_1)"] & \cdots \arrow[r,"H_k(f_{m\minus 1})"] &[5pt] H_k(K_m) \arrow[r,"\cong"] &[-3pt] \cdots
\end{tikzcd}
\end{equation}
where we artificially extend finite persistence modules by isomorphisms to infinite persistence modules of finite type to fit them into the algebraic framework discussed in Section \ref{sec_background}. From an input tower~\eqref{application_tower}, we can construct the following complex of  persistence modules for computing its persistent homology
\begin{equation} \label{tower_complex}
\begin{tikzcd}
C_{k+1}(\Vec{K}) \arrow[r,"\partial_{k+1}"] & C_{k}(\Vec{K}) \arrow[r,"\partial_{k}"] & C_{k\minus 1}(\Vec{K}) 
\end{tikzcd}
\end{equation}
where $C_k(\Vec{K})$ is the persistence module of simplicial $k$-chains in $\Vec{K}$ defined by
\begin{equation} \label{simplex_modules}
\begin{tikzcd}[column sep=large]
C_k(\Vec{K}):\, C_k(K_0) \arrow[r,"C_k(f_0)"] &[0pt] C_k(K_1) \arrow[r,"C_k(f_1)"] &[0pt]  \cdots \arrow[r,"C_k(f_{m\minus 1})"] &[5pt]  C_k(K_m) \arrow[r,"\cong"] &[-3pt] \cdots
\end{tikzcd}
\end{equation}
Since simplices can be collapsed in a tower, the linear maps $C_k(f_i)$ are not necessarily injective. Thus, in general $C_k(\Vec{K})$ is not a free module, but we can use the pipeline developed in Sections \ref{sec_presentation_algorithm}, \ref{sec_comp_presentations} and \ref{sec_cohomology_presentations} to compute the barcode of \eqref{tower_complex_homology}. 

\textbf{Algorithm Tower:} In practice a tower is usually represented as a sequence of elementary inclusions (adding a single simplex) and elementary collapses (collapsing exactly two vertices). It is known that every tower can be represented in this way~\cite{DFW14}, so w.l.o.g.\ we assume that \eqref{application_tower} is a sequence of elementary inclusions and collapses. Every collapse can be thought of as making some simplices (chains) collapse trivially to $0$, and some other pairs of simplices merging together. The number of such trivial collapses and mergings cannot be more than the number of elementary inclusions. So, we define the size $n$ of the tower as the number of elementary inclusions in $\Vec{K}$. The special kind of persistence modules \eqref{simplex_modules}, arising from simplicial towers, significantly simplify the computation of a presentation of \eqref{tower_complex}. The reason for this simplification is that in the matrices $C_k(f_i)$ in \eqref{simplex_modules} each column is either zero or contains a single $1$ which implies that each column is reduced by a single other column and there is no need to explicitly construct these matrices. In fact, a canonical presentation of $C_k(\Vec{K})$ can be computed in $O(n)$ time directly by observing elementary inclusions (creating generators), trivial collapses (pairing simplices to itself), and mergings (pairing the relation to the simplex in the merge appearing later in the filtration). Corresponding to every merge, we have a basis change which needs to be reflected by a sum of two columns in the boundary matrices that connect two modules. This will incur quadratic complexity (see Appendix \ref{app_application_tower}, Proposition \ref{prop_modified_presentation_algo}). By Proposition \ref{lemma1} and the fact that the boundary of a simplicial chain can not die before the chain itself, this computation of a canonical presentation of the complex in \eqref{tower_complex} yields a complex of presentations.

\begin{theorem}
Given a simplicial tower \eqref{application_tower} of size $n$, the algorithm \textbf{Tower} outlined above computes the persistent homology, i.e., the barcode of \eqref{tower_complex_homology} in $O(n^3)$ time.
\label{thm:compute-tower}
\end{theorem}

\begin{proof}
See Appendix \ref{app_application_tower}.
\end{proof}

\subsection{Persistent cosheaf homology over simplicial towers} \label{sec_applications_cosheaf_tower}
In this section and the next, we consider cosheaves and sheaves
over finite simplicial complexes. Briefly, as depicted in Figure \ref{fig:cosheaf}, a (co)sheaf is an assignment of vector spaces on simplices of a simplicial complex and linear maps among them, which satisfy certain conditions of commutativity. Similar to the homology of a simplicial complex, the homology of a cosheaf over a simplicial complex can be computed from a chain complex where the boundary operators are determined by how the vector spaces over $k$-simplices map to the vector spaces over their $(k\minus 1)$-dimensional faces. For some background on cosheaves over persistent spaces, persistent sheaves, and, their persistent homology and cohomology, see Appendices~\ref{app_sheaves} and~\ref{app_cosheaves}.

\begin{figure}
    \centering
    \includegraphics[scale=0.6]{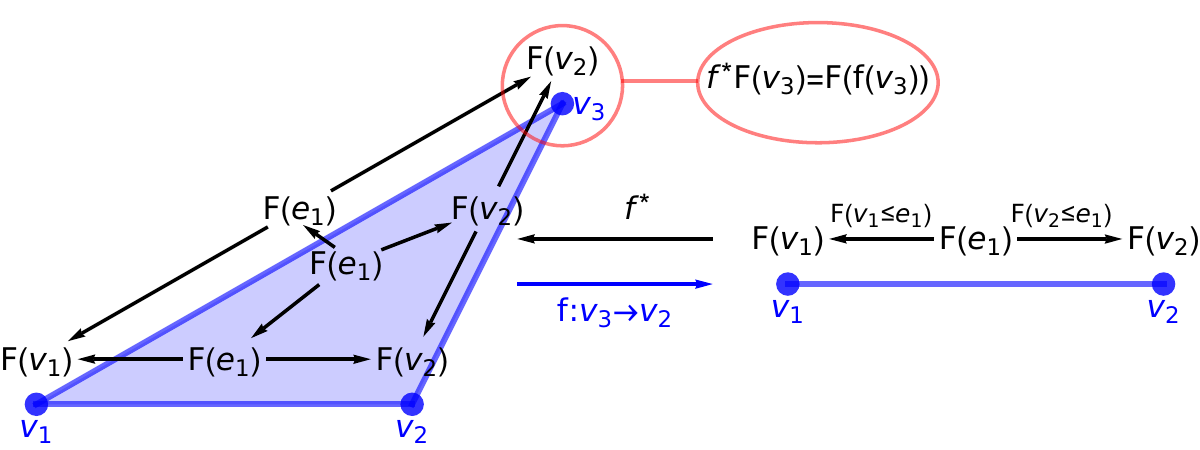}
    \caption{The pullback $f^*F$ of a cosheaf $F$ along an elementary collapse $f:\text{v}_3\rightarrow \text{v}_2$.}
    \label{fig:cosheaf}
\end{figure}

The persistent homology of a tower can be viewed as a special case of persistent cosheaf homology over a tower (Appendix~\ref{app_cosheaves}) where the cosheaf is the constant cosheaf. We consider a simplicial tower as in the lower row of \eqref{application_cosheaf_tower} and a cosheaf $F\colon K_m^{\text{op}}\rightarrow \mathbf{vec}$ over $K_m$.
\begin{equation} \label{application_cosheaf_tower}
\begin{tikzcd}
&[-25pt] F^0 & F^1 \arrow[l,swap,mapsto,"f_{0}^*"] & \cdots \arrow[l,swap,mapsto,"f_{1}^*"] & F^m \arrow[l,swap,mapsto,"f_{m\minus 1}^*"] \\[-13pt]
\Vec{K}:& K_0 \arrow[r,"f_0"] & K_1 \arrow[r,"f_1"] & \cdots \arrow[r,"f_{m\minus 1}"] &[2pt] K_m
\end{tikzcd}
\end{equation}
As shown in Figure \ref{fig:cosheaf}, we use the inverse image functors $f^*_i$ to iteratively pull back $F=F^m$ to cosheaves on all the complexes $K_i$ by defining $F^i\coloneqq f_i^*F^{i+1}$. We now use the induced maps $H_k(f_i)\colon H_k(K_i,F^i)\rightarrow H_k(K_{i+1},F^{i+1})$ to construct the cosheaf homology persistence module
\begin{equation} \label{tower_cosheaf_complex_homology} \small
\begin{tikzcd}[column sep=large]
H_k(\Vec{K},F):\, H_k(K_0,F^0) \arrow[r,"H_k(f_0)"] & H_k(K_1,F^1) \arrow[r,"H_k(f_1)"] & \cdots \arrow[r,"H_k(f_{m\minus 1})"] &[7pt] H_k(K_m,F^m) \arrow[r,"\cong"] &[-10pt] \cdots
\end{tikzcd}
\end{equation}
The persistent cosheaf homology is the homology of the complex of persistence modules 
\begin{equation*} \label{tower_cosheaf_complex}
\begin{tikzcd}
C_{k+1}(\Vec{K},F) \arrow[r,"\partial_{k+1}"] & C_{k}(\Vec{K},F) \arrow[r,"\partial_{k}"] & C_{k\minus 1}(\Vec{K},F) 
\end{tikzcd}
\end{equation*}
where $C_k(\Vec{K},F)$ is the persistence module of finite type defined by 
\begin{equation*}\small
\begin{tikzcd}
C_k(\Vec{K},F):\, C_k(K_0,F^0) \arrow[r,"C_k(f_0)"] &[8pt] C_k(K_1,F^1) \arrow[r,"C_k(f_1)"] &[8pt]  \cdots \arrow[r,"C_k(f_{m\minus 1})"] &[16pt]  C_k(K_m,F^m) \arrow[r,"\cong"] & \cdots
\end{tikzcd}
\end{equation*}
\textbf{Algorithm Cosheaf:} As in the case of simplicial towers the persistence modules $C_k(\Vec{K},F)$ are not necessarily free. We can again assume that the tower is given by a sequence of elementary inclusions and collapses and define the size $n$ of the instance $(\Vec{K},F)$ as $n\coloneqq \text{\# elementary inclusions}\times\underset{\sigma\in K_m}{\text{max}}\text{dim}\big(F(\sigma)\big)$. We can now use our pipeline developed in Sections \ref{sec_presentation_algorithm}, \ref{sec_comp_presentations} and \ref{sec_cohomology_presentations} to compute the barcode of \eqref{tower_cosheaf_complex_homology} where the same simplifications as in the tower case (Section \ref{sec_application_tower}) apply
(see Appendix \ref{sec_appendix_cosheaf_tower}, Proposition \ref{prop_modified_presentation_algo_cosheaf}).
\begin{theorem}
Given an instance $(\Vec{K},F)$ of size $n$ as in \eqref{application_cosheaf_tower}, the algorithm \textbf{Cosheaf} above computes the persistent cosheaf homology, i.e.,\ the barcode of \eqref{tower_cosheaf_complex_homology} in $O(n^3)$ time.
\label{thm:compute-sheaf}
\end{theorem}

\begin{proof}
See Appendix \ref{sec_appendix_cosheaf_tower}.
\end{proof}

\begin{example}
We consider the following simplicial tower $\Vec{K}$:
\begin{equation*}
\begin{tikzcd}[row sep=small, column sep=small]
\phantom{\bullet} & \phantom{\bullet} & F^8 & \phantom{\bullet} & \phantom{\bullet} & \phantom{\bullet} & F^9 \arrow[llll,swap,mapsto,"z\rightarrow w"] & \phantom{\bullet} & \phantom{\bullet} & F^{10} \arrow[lll,swap,mapsto,"w\rightarrow v"] & \phantom{\bullet} & F^{11} \arrow[ll,swap,mapsto,"v\rightarrow u"] & \phantom{\bullet} \\[-12pt]
\phantom{\bullet} \arrow[rrrrrrrrrrrr,dash,shorten=-7pt] & \phantom{\bullet} & \phantom{\bullet} & \phantom{\bullet} & \phantom{\bullet} & \phantom{\bullet} & \phantom{\bullet} & \phantom{\bullet} & \phantom{\bullet} & \phantom{\bullet} & \phantom{\bullet} & \phantom{\bullet} & \phantom{\bullet} \\[-8pt]
\phantom{\bullet} & \field \arrow[d,swap,"1",<-] \arrow[r,"1",<-] & \field & \field \arrow[l,swap,"1",<-] \arrow[dl,swap,"1",<-] \arrow[d,<-,"1"] & \phantom{\bullet} & \field \arrow[d,swap,"1",<-] \arrow[r,"1",<-] & \field & \field \arrow[l,swap,"1",<-] \arrow[dl,swap,"1",<-] & \phantom{\bullet} & \field \arrow[d,swap,"1",<-] & \phantom{\bullet} & \phantom{\bullet} & \phantom{\bullet} \\[5pt]
\phantom{\bullet} & \field & \field & \field & \phantom{\bullet} & \field & \field & \phantom{\bullet} & \phantom{\bullet} & \field & \phantom{\bullet} & \phantom{\bullet} & \phantom{\bullet} \\[5pt]
\phantom{\bullet} & \field \arrow[u,"1",<-] \arrow[r,"1",<-] \arrow[ur,"1",<-] & \field & \field \arrow[u,swap,"1",<-] \arrow[l,swap,"1",<-] & \phantom{\bullet} & \field \arrow[u,"1",<-] \arrow[ur,"1",<-] & \phantom{\bullet} & \phantom{\bullet} & \phantom{\bullet} & \field \arrow[u,"1",<-] & \phantom{\bullet} & \field & \phantom{\bullet} \\[-8pt]
\phantom{\bullet} \arrow[rrrrrrrrrrrr,dash,shorten=-7pt] & \phantom{\bullet} & \phantom{\bullet} & \phantom{\bullet} & \phantom{\bullet} & \phantom{\bullet} & \phantom{\bullet} & \phantom{\bullet} & \phantom{\bullet} & \phantom{\bullet} & \phantom{\bullet} & \phantom{\bullet} & \phantom{\bullet} \\[-10pt]
\phantom{\bullet} & v & \phantom{\bullet} & w & \phantom{\bullet} & v & \phantom{\bullet} & w & \phantom{\bullet} & v & \phantom{\bullet} & \phantom{\bullet} & \phantom{\bullet} \\[-12pt]
\phantom{\bullet} & \bullet \arrow[rr,dash,shorten=-5pt] & \phantom{\bullet} & \bullet & \phantom{\bullet} & \bullet \arrow[rr,dash,shorten=-5pt] & \phantom{\bullet} & \bullet & \phantom{\bullet} & \bullet & \phantom{\bullet} & \phantom{\bullet} & \phantom{\bullet} \\
\phantom{\bullet} & \phantom{\bullet} & \phantom{\bullet} & \phantom{\bullet} & \phantom{\bullet} & \phantom{\bullet} & \phantom{\bullet} & \phantom{\bullet} & \phantom{\bullet} & \phantom{\bullet} & \phantom{\bullet} & \phantom{\bullet} & \phantom{\bullet} \\
\phantom{\bullet} & \bullet \arrow[rr,dash,shorten=-5pt] \arrow[uu,dash,shorten=-5pt] \arrow[uurr,dash,shorten=-9pt] & \phantom{\bullet}& \bullet \arrow[uu,dash,shorten=-5pt] & \phantom{\bullet} & \bullet \arrow[uu,dash,shorten=-5pt] \arrow[uurr,dash,shorten=-9pt] & \phantom{\bullet} & \phantom{\bullet} & \phantom{\bullet} & \bullet \arrow[uu,dash,shorten=-5pt] & \phantom{\bullet} & \bullet & \phantom{\bullet} \\[-12pt]
\phantom{\bullet} & u & \phantom{\bullet} & z & \phantom{\bullet} & u & \phantom{\bullet} & \phantom{\bullet} & \phantom{\bullet} & u & \phantom{\bullet} & u & \phantom{\bullet} \\[-10pt]
\phantom{\bullet} \arrow[rrrrrrrrrrrr,dash,shorten=-7pt] \arrow[uuuuuuuuuu,dash,shorten=-6pt] & \phantom{\bullet} & \phantom{\bullet} & \phantom{\bullet} & \phantom{\bullet} \arrow[uuuuuuuuuu,dash,shorten=-6pt] & \phantom{\bullet} & \phantom{\bullet} & \phantom{\bullet} & \phantom{\bullet} \arrow[uuuuuuuuuu,dash,shorten=-6pt] & \phantom{\bullet} & \phantom{\bullet} \arrow[uuuuuuuuuu,dash,shorten=-6pt] & \phantom{\bullet} & \phantom{\bullet} \arrow[uuuuuuuuuu,dash,shorten=-6pt] \\[-12pt]
\phantom{\bullet} & \phantom{\bullet} & K_8 \arrow[rrrr,"z\rightarrow w"] & \phantom{\bullet} & \phantom{\bullet} & \phantom{\bullet} & K_9 \arrow[rrr,"w\rightarrow v"] & \phantom{\bullet} & \phantom{\bullet} & K_{10} \arrow[rr,"v\rightarrow u"] & \phantom{\bullet} & K_{11} & \phantom{\bullet} 
\end{tikzcd}
\end{equation*}
where we assume that $K_8$ is obtained by simplex-wise inclusions in order
\begin{equation*}
\begin{pmatrix} 0 & 1 & 2 & 3 & 4 & 5 & 6 & 7 & 8 \\ u & v & w & z & uv & vw & wz & uz & uw \end{pmatrix} .   
\end{equation*}
Moreover, we consider the constant cosheaf $F=F^{11}$ on $K_{11}$ defined by $F(u)\coloneqq\field$ and set $\field\coloneqq \mathbb{Z}_2$. By pulling back $F$ along the maps $f_i\colon K_i\rightarrow K_{11}$ of the tower we obtain cosheaves $F^i=f_i^*F$ defined by $F^i(\sigma)=F\big(f_i(\sigma)\big)=F(u)=\field$ for all $\sigma\in K_i$. Thus we obtain the constant cosheaf on all $K_i$ in $ \Vec{K}$. We now demonstrate how to compute a presentation of the morphism $\partial_1\colon C_1(\Vec{K},F)\rightarrow C_0(\Vec{K},F)$. We process the cosheaf-tower from $(K_0,F^0)$ to $(K_{11},F^{11})$ and build the matrix $f_0$ with birth/death annotations. We start with an empty matrix $f_0=\emptyset$. If an edge $\tau$ is included at index $i$ we add the matrix $\partial_1(\tau)\colon C_1(K_i,F^i)|_{F^i(\tau)}\rightarrow C_0(K_i,F^i)$ consisting of $\text{dim}(F^i(\tau))=1$ columns with $b(\text{col})=i$ and $d(\text{col})=\infty$ to $f_0$. If a vertex $\sigma$ is included at index $j$ we add a zero-matrix consisting of $\text{dim}(F^j(\sigma))=1$ rows to $f_0$ with $b(\text{row})=j$ and $d(\text{row})=\infty$. Note that the boundary matrix of the constant cosheaf $\partial_1\colon C_1(K_i,F^i)\rightarrow C_0(K_i,F^i)$ agrees with the boundary matrix of the simplicial complex itself. Hence, after processing $K_0,\ldots, K_8$ we have 
\begin{equation*} \footnotesize
f_0=
\begin{blockarray}{ccccccc}
& \text{[}4,\infty) & \text{[}5,\infty) & \text{[}6,\infty) & \text{[}7,\infty) & \text{[}8,\infty) & \\
\begin{block}{c(ccccc)c}
\text{[}0,\infty) & 1 & 0 & 0 & 1 & 1 & u \\
\text{[}1,\infty) & 1 & 1 & 0 & 0 & 0 & v \\
\text{[}2,\infty) & 0 & 1 & 1 & 0 & 1 & w \\
\text{[}3,\infty) & 0 & 0 & 1 & 1 & 0 & z \\
\end{block}
& uv & vw & wz & uz & uw & \\
\end{blockarray}  
\end{equation*}
At index $9$ the vertex $z$ gets collapsed into $w$. As a consequence the edge $uz$ is merged into $uw$ and the edge $wz$ is collapsed to the vertex $w$. Hence, we have to add the column corresponding to $uz$, which is born first, to the column corresponding to $uw$ and set $d(uw)=9$ and $d(wz)=9$. Moreover, we have to add the row corresponding to $z$ to the one corresponding to $w$ and set $d(z)=9$. After this step we have 
\begin{equation*} \footnotesize
f_0=
\begin{blockarray}{cccccc}
& \text{[}4,\infty) & \text{[}5,\infty) & \text{[}6,9) & \text{[}7,\infty) & \text{[}8,9) \\
\begin{block}{c(ccccc)}
\text{[}0,\infty) & 1 & 0 & 0 & 1 & 0  \\
\text{[}1,\infty) & 1 & 1 & 0 & 0 & 0 \\
\text{[}2,\infty) & 0 & 1 & 0 & 1 & 0 \\
\text{[}3,9) & 0 & 0 & 1 & 1 & 1 \\
\end{block}
\end{blockarray}  
\end{equation*}
When moving from $9$ to $10$ the vertex $w$ is collapsed into $v$. As a consequence the edge $uw$ is merged into $uv$ and the edge $vw$ is collapsed into $v$. Hence, we have to add the column corresponding to $uv$ to the column corresponding to $uz$ and set $d(uz)=10$ and $d(vw)=10$. Moreover, we have to add the row corresponding to $w$ to the one corresponding to $v$ and set $d(w)=10$. After this step we have 
\begin{equation*} \footnotesize
f_0=
\begin{blockarray}{cccccc}
& \text{[}4,\infty) & \text{[}5,10) & \text{[}6,9) & \text{[}7,10) & \text{[}8,9) \\
\begin{block}{c(ccccc)}
\text{[}0,\infty) & 1 & 0 & 0 & 0 & 0 \\
\text{[}1,\infty) & 1 & 0 & 0 & 0 & 0 \\
\text{[}2,10) & 0 & 1 & 0 & 1 & 0 \\
\text{[}3,9) & 0 & 0 & 1 & 1 & 1 \\
\end{block}
\end{blockarray}  
\end{equation*}
When moving from $10$ to $11$ the vertex $v$ is collapsed into $u$. As a consequence the edge $uv$ is collapsed into $u$. Hence, we have to set $d(uv)=11$. Moreover, we have to add the row corresponding to $v$ to the one corresponding to $u$ and set $d(v)=11$. After this step we obtain the final result 
\begin{equation*} \footnotesize
f_0=
\begin{blockarray}{cccccc}
& \text{[}4,11) & \text{[}5,10) & \text{[}6,9) & \text{[}7,10) & \text{[}8,9) \\
\begin{block}{c(ccccc)}
\text{[}0,\infty) & 0 & 0 & 0 & 0 & 0 \\
\text{[}1,11) & 1 & 0 & 0 & 0 & 0 \\
\text{[}2,10) & 0 & 1 & 0 & 1 & 0 \\
\text{[}3,9) & 0 & 0 & 1 & 1 & 1 \\
\end{block}
\end{blockarray}  
\end{equation*}

\end{example}

\subsection{Cohomology of persistent sheaves over simplicial complexes} \label{sec_pers_shv_simp}
In this section we compute the cohomology of a persistent sheaf over a simplicial complex. We consider a persistent sheaf $\Vec{F}\colon\mathbb{N}_0\rightarrow \mathbf{Shv}(K,\mathbf{vec})$ (Appendix~\ref{app_sheaves}) of finite type on a finite simplicial complex $K$, i.e.\ a diagram of sheaves and sheaf morphisms (see Figure \ref{fig:sheaf_example_main}) over $K$
\begin{equation} \label{persistent_sheaf}
\begin{tikzcd}
\Vec{F}:\, F_0 \arrow[r,"\phi_0"] & F_1 \arrow[r,"\phi_1"] & \cdots \arrow[r,"\phi_{m\minus 1}"] &[3pt] F_m \arrow[r,"\cong"] & \cdots
\end{tikzcd}
\end{equation}
where $\phi_i$ is an isomorphism for all $i\geq m$. As depicted in Figure \ref{fig:psheaf-example_main}, a persistent sheaf can also be viewed as a sheaf of persistence modules. The figure shows parts of the sheaf in \eqref{pers_shv_example} of interval persistence modules on a simplicial complex containing a triangle and its faces where the maps are identity on the overlap of intervals and zero elsewhere:
\begin{figure}
\centering
\includegraphics[scale=0.35]{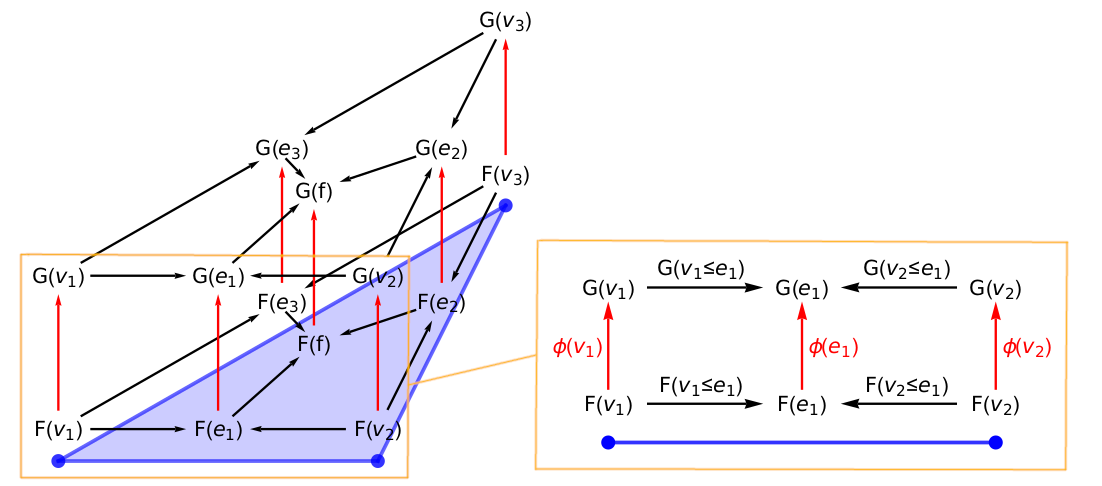}
\caption{A sheaf morphism $\phi\colon F\rightarrow G$ over a triangle and its faces.}
\label{fig:sheaf_example_main}
\end{figure}
\begin{figure}[htb] 
\includegraphics[width=\textwidth]{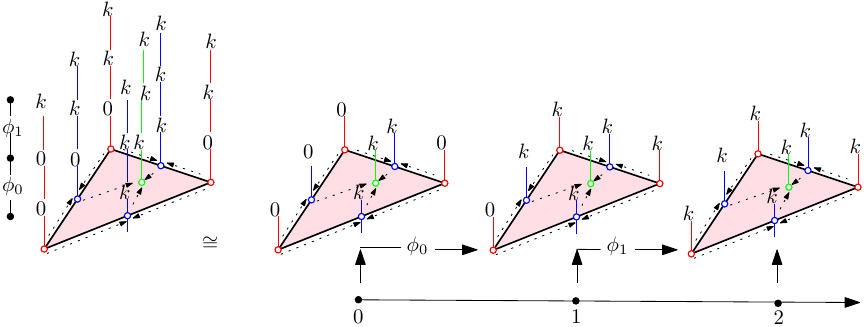}
\caption{Two equivalent viewpoints of a persistent sheaf over a triangle and its faces.}
\label{fig:psheaf-example_main}
\end{figure}
\begin{equation} \label{pers_shv_example}
\begin{tikzcd}
& & \mathbb{I}_{[1,7)} \arrow[dr] \arrow[dl] \\
& \mathbb{I}_{[1,6)} \arrow[r] & \mathbb{I}_{[0,3)} & \mathbb{I}_{[0,5)} \arrow[l] \\
\mathbb{I}_{[2,6)} \arrow[rr] \arrow[ur] & & \mathbb{I}_{[0,5)} \arrow[u] & & \mathbb{I}_{[1,5)} \arrow[ll] \arrow[ul]
\end{tikzcd}
\end{equation}
Our goal is to compute the persistent sheaf cohomology of $\Vec{F}$, i.e.\ the barcode of the following persistence module: 
\begin{equation} \label{persistent_sheaf_cohomology} \small
\begin{tikzcd}
H^k(X,\Vec{F}):\, H^k(K,F_0) \arrow[r,"H^k(K\text{,}\phi_0)"] &[15pt] H^k(K,F_1) \arrow[r,"H^k(K\text{,}\phi_1)"] &[15pt] \cdots \arrow[r,"H^k(K\text{,}\phi_{m\minus 1})"] &[20pt] H^k(K,F_m) \arrow[r,"\cong"] & \cdots
\end{tikzcd}
\end{equation}
\textbf{Algorithm PersistentSheaf:} We first assume that the input for the computation is given by a collection of matrices representing the internal maps $F_i(\sigma\leq\tau)$ of the sheaves for all $0\leq i\leq m$ and $\sigma <_1 \tau\in K$ and a collection of matrices representing the sheaf morphisms $\phi_i(\sigma)$ for all $0\leq i<m$ and $\sigma\in K$. We define the size of the input as $n\coloneqq \sum_{i=0}^m n_i$ where $n_i\coloneqq \sum_{\sigma\in K} \text{dim}\big(F_i(\sigma)\big)$.

From the input we can construct the following complex of persistence modules  
\begin{equation} \label{sheaf_complex_applications}
\begin{tikzcd}
C^{k\minus1}(K,\Vec{F}) \arrow[r,"\delta^{k\minus 1}"] & C^{k}(K,\Vec{F}) \arrow[r,"\delta^{k}"] & C^{k+1}(K,\Vec{F}) 
\end{tikzcd}
\end{equation}
where $C^k(K,\Vec{F})$ is a persistence module of finite type defined by 
\begin{equation*}
\begin{tikzcd}
C^k(K,\Vec{F}):\, C^k(K,F_0) \arrow[r,"C^k(X\text{,}\phi_0)"] &[12pt] C^k(K,F_1) \arrow[r,"C^k(K\text{,}\phi_1)"] &[12pt]  \cdots \arrow[r,"C^k(K\text{,}\phi_{m\minus 1})"] &[15pt]  C^k(K,F_m) \arrow[r,"\cong"] & \cdots
\end{tikzcd}
\end{equation*}
Using the pipeline developed in Sections \ref{sec_presentation_algorithm}, \ref{sec_comp_presentations} and \ref{sec_cohomology_presentations}, we can compute a presentation of \eqref{sheaf_complex_applications}, transform it into a complex of presentations and compute the barcode of the cohomology.

\begin{theorem}
Given a persistent sheaf \eqref{persistent_sheaf} of size $n$, the algorithm \textbf{PersistentSheaf} above computes the persistent sheaf cohomology, i.e.,\ the barcode of \eqref{persistent_sheaf_cohomology}, in $O(n^3)$ time.
\end{theorem}

\begin{proof}
This follows directly from the results in Section \ref{sec_presentation_algorithm}, \ref{sec_comp_presentations} and \ref{sec_cohomology_presentations}.
\end{proof}

\begin{example}
We demonstrate how to compute the cohomology in degree $k=1$ of the persistent sheaf in \eqref{pers_shv_example}. We assume $\field=\mathbb{Z}_2$ and start by constructing a complex of presentations of \eqref{sheaf_complex_applications} for $k=1$. As in \eqref{cochain_complex_def}, we can build $C^k(K,\Vec{F})$ by taking the direct sum of all the persistence modules $\Vec{F}(\sigma)$ over $k$-simplices $\sigma\in K$, i.e., $C^k(K,\Vec{F})\coloneqq\underset{\sigma\in K^k}{\bigoplus}\Vec{F}(\sigma)$. In the case of \eqref{pers_shv_example}, we obtain: 
\begin{equation*}
\begin{tikzcd}
C^0(K,\Vec{F}) \arrow[d,"\delta^0"] &[-15pt] = &[-15pt] \mathbb{I}_{[2,6)} \arrow[d] \arrow[drr] &[-25pt] \oplus &[-25pt] \mathbb{I}_{[1,5)} \arrow[drr] \arrow[dll,crossing over] &[-25pt] \oplus &[-25pt] \mathbb{I}_{[1,7)} \arrow[dll,crossing over] \arrow[d] \\
C^1(K,\Vec{F}) \arrow[d,"\delta^1"] & = & \mathbb{I}_{[0,5)} \arrow[drr] & \oplus & \mathbb{I}_{[1,6)} \arrow[d] & \oplus & \mathbb{I}_{[0,5)} \arrow[dll] \\
C^2(K,\Vec{F}) & = & & & \mathbb{I}_{[0,3)}
\end{tikzcd}
\end{equation*}
where the arrows depict morphisms between interval modules over $K$ which are the identity on the overlap of intervals and zero elsewhere. Note that the structures of the maps $\delta^0$ and $\delta^1$ resemble the coboundary morphisms of a triangle
\begin{equation*} \footnotesize
\delta^0=
\begin{blockarray}{cccc}
& \text{[}2,6) & \text{[}1,5) & \text{[}1,7)  \\
\begin{block}{c(ccc)}
\text{[}0,5) & 1 & 1 & 0 \\
\text{[}0,5) & 0 & 1 & 1 \\
\text{[}1,6) & 1 & 0 & 1 \\
\end{block}
\end{blockarray} 
\hspace{10pt} , \hspace{5pt}
\delta^1=
\begin{blockarray}{cccc}
& \text{[}0,5) & \text{[}0,5) & \text{[}1,6)  \\
\begin{block}{c(ccc)}
\text{[}0,3) & 1 & 1 & 1 \\
\end{block}
\end{blockarray} \hspace{1pt}.
\end{equation*}
Using these observations we can easily build the following complex of canonical presentations:
\begin{equation*}
\begin{tikzcd}[ampersand replacement=\&,every label/.append style = {font = \tiny},column sep=large, row sep=large]
0 \& C^0(K,\Vec{F}) \arrow[l] \arrow[d,swap,"\delta^0"] \&[10pt] P_0 \arrow[l] \arrow{d}[swap,xshift=-2pt,yshift=-3pt]{\begin{blockarray}{c@{\hspace{7pt}}ccc}
& 2 & 1 & 1 \\
\begin{block}{c@{\hspace{7pt}}(ccc)}
0 & 1 & 1 & 0 \\
0 & 0 & 1 & 1 \\
1 & 1 & 0 & 1 \\
\end{block}
\end{blockarray}} \&[20pt] P_1 \arrow{l}[swap,xshift=-2pt,yshift=-7pt]{\begin{blockarray}{c@{\hspace{7pt}}ccc}
& 6 & 5 & 7 \\
\begin{block}{c@{\hspace{7pt}}(ccc)}
2 & 1 & 0 & 0 \\
1 & 0 & 1 & 0 \\
1 & 0 & 0 & 1 \\
\end{block}
\end{blockarray}} \arrow{d}[xshift=-2pt,yshift=-3pt]{\begin{blockarray}{c@{\hspace{7pt}}ccc}
& 6 & 5 & 7 \\
\begin{block}{c@{\hspace{7pt}}(ccc)}
5 & 1 & 1 & 0 \\
5 & 0 & 1 & 1 \\
6 & 1 & 0 & 1 \\
\end{block}
\end{blockarray}} \&[10pt] 0 \arrow[l] \\[7pt]
0 \& C^1(K,\Vec{F}) \arrow[l] \arrow[d,swap,"\delta^1"] \& Q_0 \arrow[l] \arrow{d}[swap,xshift=-2pt,yshift=-3pt]{\begin{blockarray}{c@{\hspace{5pt}}ccc}
& 0 & 0 & 1 \\
\begin{block}{c@{\hspace{5pt}}(ccc)}
0 & 1 & 1 & 1 \\
\end{block}
\end{blockarray}} \& Q_1 \arrow{l}[swap,xshift=-2pt,yshift=-7pt]{\begin{blockarray}{c@{\hspace{7pt}}ccc}
& 5 & 5 & 6 \\
\begin{block}{c@{\hspace{7pt}}(ccc)}
0 & 1 & 0 & 0 \\
0 & 0 & 1 & 0 \\
1 & 0 & 0 & 1 \\
\end{block}
\end{blockarray}} \arrow{d}[xshift=-2pt,yshift=-3pt]{\begin{blockarray}{c@{\hspace{5pt}}ccc}
& 5 & 5 & 6 \\
\begin{block}{c@{\hspace{5pt}}(ccc)}
3 & 1 & 1 & 1 \\
\end{block}
\end{blockarray}} \& 0 \arrow[l] \\[3pt]
0 \& C^2(K,\Vec{F}) \arrow[l] \& R_0 \arrow[l] \& R_1 \arrow{l}[swap,xshift=-2pt,yshift=-7pt]
{\begin{blockarray}{c@{\hspace{4pt}}c}
& 3 \\
\begin{block}{c@{\hspace{4pt}}(c)}
0 & 1  \\
\end{block}
\end{blockarray}}\& 0 \arrow[l]
\end{tikzcd}
\end{equation*}
The cohomology of this complex of presentations is computed in Example \ref{homology_algo_example} in Section \ref{sec_cohomology_presentations}. From there we get $H^1(K,\Vec{F})=\ker\delta^1/\im\delta^0\cong\mathbb{I}_{[0,1)}\oplus \mathbb{I}_{[3,5)}$ for the persistent sheaf cohomology of $\Vec{F}$ over $K$.  
\end{example}

\noindent
In practice, it is unlikely that one is just handed a persistent sheaf in the form described above. In an application the persistent sheaf has to be constructed in the first place. To construct a persistent sheaf $\Vec{F}\colon \mathbb{N}_0\rightarrow \mathbf{Shv}(K,\mathbf{vec})$ or equivalently a sheaf of persistence modules $\Vec{F}\colon K\rightarrow \mathbf{pMod}$, one has to construct a persistence module $\Vec{F}(\sigma)\colon \mathbb{N}_0\rightarrow\mathbf{vec}$ over every simplex $\sigma \in K$ and connect them by a morphism of persistence modules $\Vec{F}(\sigma\leq \tau)$ if $\sigma<_1\tau$. The idea of sheaves is to organize local information over a space. We use this idea to compute presentations $0\leftarrow\Vec{F}(\sigma)\leftarrow P^\sigma_0\xleftarrow{p^\sigma} P^\sigma_1$ of the modules over each simplex and the morphisms $\Vec{F}(\sigma\leq\tau)$ locally in time $O\big(\text{max}\{\sum_{i=0}^m\text{dim}(F_i(\sigma)),\sum_{i=0}^m\text{dim}(F_i(\tau))\}^3\big)$. This can be computed in a distributed manner while constructing the persistent sheaf. From these local presentations we can assemble a presentation of \eqref{sheaf_complex_applications}. After this construction / preprocessing step we can define the size of the presentation as $\overline{n}\coloneqq \sum_{\sigma\in K}\text{rank}(P_0^\sigma)$. The size of $\overline{n}$ depends on the structure of the modules $\Vec{F}(\sigma)$. If all their summands restricted to $[0,m]$ are intervals of length $0$, then $\overline{n}=n$, but, if all their summands are intervals of length $m$, then $\overline{n}=\frac{n}{m}$. The remaining steps in our pipeline take $O(\overline{n}^3)$ time. Therefore, if one can take advantage of distributed computation of the input and local presentations, our approach can be significantly faster than $O(n^3)$.

\subsection{Cohomology of persistent sheaves over posets} \label{sec_poset}

So far our algorithms allow us to compute the persistent sheaf cohomology barcode of persistent sheaves over simplicial complexes. Discrete sheaves and their persistent version can be defined more generally on finite posets. The problem is that for arbitrary finite posets there is no direct analog to the explicit (co)chain complex (of polynomial size) computing the sheaf cohomology over a simplicial complex. For arbitrary posets one could use the general definition of sheaf cohomology via derived functors (see \cite{brown2022discrete} for the non-persistent case). We now show that we can reduce the computation of persistent sheaf cohomology over a finite poset to a computation over a simplicial complex. Let $X$ be a finite poset and $\Vec{F}\colon \mathbb{N}_0\rightarrow \mathbf{Shv}(X,\mathbf{vec})$ a persistent sheaf. The order complex $\mathcal{K}(X)$ is the simplicial complex defined by all totally ordered subsets of $X$, i.e.\ $\mathcal{K}(X)\coloneqq\{\{x_0,\ldots,x_k\}\vert x_0<\ldots<x_k\in X)\}$. It comes with a surjective projection morphism $f\colon\mathcal{K}(X)\rightarrow X$ defined by $f(x_0<\ldots<x_k)\coloneqq x_k$. We can now pull back $\Vec{F}$ along this projection $f$ to a persistent sheaf on $\mathcal{K}(X)$. Given $x_\bullet=(x_0<\cdots <x_k)\leq (y_0<\cdots <y_l)=y_\bullet\in\mathcal{K}(X)$ we define this pullback by $f^*\Vec{F}(x_\bullet)\coloneqq \Vec{F}(x_k)$ and $f^*\Vec{F}(x_\bullet\leq y_\bullet)\coloneqq \Vec{F}(x_k\leq y_l)$.  The following theorem shows that the persistent sheaf cohomology of $\Vec{F}$ on the poset $X$ is isomorphic to the persistent sheaf cohomology of $f^*\Vec{F}$ on the simplicial complex $\mathcal{K}(X)$.

\begin{theorem}
If $X$ is a finite poset, $\Vec{F}$ a persistent sheaf on $X$ and $f\colon \mathcal{K}(X)\rightarrow X$ the projection from the order complex, then $H^k(X,\Vec{F})\cong H^k(\mathcal{K}(X),f^*\Vec{F})$.
\label{thm:poset}
\end{theorem}

\begin{proof}
See Appendix \ref{app_order_complex}.
\end{proof}

\noindent
We are now able to compute the persistent sheaf cohomology over any finite poset by pulling back to the order complex and using our pipeline for simplicial complexes. Unfortunately this approach is only practical for small instances since the size of the order complex is exponential in the size of the poset. This raises the question if it is possible to compute the persistent sheaf cohomology over a general poset in polynomial time. In the following we give an example of a class of posets, which can have exponentially large order complexes, where we can compute the persistent cohomology in polynomial time. 

Let $X$ be a zigzag-poset, i.e. a poset with the following Hasse diagram
\begin{equation} \label{hasse}
\begin{tikzcd}
x_0 \arrow[r] & \arrow[l] x_1 \arrow[r] & \arrow[l] \cdots \arrow[r] & \arrow[l] x_{n\minus 1} \arrow[r] & \arrow[l] x_{n}
\end{tikzcd}
\end{equation}
where each arrow can point to the left or to the right. Let $\vec{F}$ be a persistent sheaf on $X$. We now construct a poset $X'$ in the following way: Suppose the points in $X$ are ordered from left to right as in \eqref{hasse}. Let $x_{i_0}$ be the minimal element in $X$ with the smallest index $i_0\in\{0,\ldots,n\}$. Let $x_{i_1}$ be the maximal element of $X$ with minimal index $i_1>i_0$. Let $x_{i_2}$ be the minimal element of $X$ with minimal index $i_2>i_1$. We proceed in this way until we reach the minimal element $x_{i_m}$ with maximal index $i_m$ and define $X'\coloneqq \{x_{i_0},\ldots,x_{i_m}\}$. The constructed zigzag-poset $X'$ is an alternating zigzag poset of the form:
\begin{equation*}
\begin{tikzcd}
& x_{i_1} & & \cdots & & x_{i_{m\minus 1}} \\
x_{i_0} \arrow[ur] & & x_{i_2} \arrow[ul] \arrow[ur] &  & x_{i_{m\minus 2}} \arrow[ur] \arrow[ul] & & x_{i_m} \arrow[ul] 
\end{tikzcd}
\end{equation*}
and is a subposet of $X$, i.e.\ there is an order-preserving inclusion $\iota\colon X'\xhookrightarrow{} X$. We can pull back the persistent sheaf $\Vec{F}$ along this inclusion to obtain a persistent sheaf $\iota^*\Vec{F}$ on $X'$. Since $X'$ is the poset of a one-dimensional simplicial complex, our algorithms are applicable to $\iota^*\Vec{F}$. Now, we can show that the cohomology of $\Vec{F}$ on $X$ agrees with the cohomology of $\iota^*\Vec{F}$ on $X'$.

\begin{theorem} \label{zigzag_reduction_thm}
Let $X$ be a zigzag-poset and $\iota\colon X'\xhookrightarrow{} X$ the inclusion of the alternating subposet. If $\Vec{F}$ is a persistent sheaf on $X$, then $ H^k(X,\Vec{F})\cong H^k(X',\iota^*\Vec{F})$. 
\end{theorem}

\begin{proof}
See Appendix \ref{app_zigzag_simplification}.
\end{proof}


\section{Conclusion} \label{sec_conclusion}
We developed efficient general purpose algorithms to compute the (co)homology of complexes of persistence modules. Our methods can deal with complexes of persistence modules given in raw form by matrices and by complexes of their presentations. They can be used to efficiently compute the persistent homology of towers, the persistent cosheaf homology over towers and the persistent sheaf cohomology over simplicial complexes. We also showed that we can use our pipeline to compute the cohomology of persistent sheaves over arbitrary finite posets by reducing it to a computation over simplicial complexes. Unfortunately, this reduction leads to an exponential time algorithm in general. For the class of zigzag posets we proved that we can compute the persistent cohomology in polynomial time. This result demonstrates that we can reduce persistent sheaves over some posets to ones over simplicial complexes of polynomial size in a way preserving the persistent cohomology. This leaves us with the open question: What is the class of posets for which the persistent sheaf cohomology can be computed in polynomial time and is the computation NP-hard for arbitrary posets?       

\bibliography{lib}

\appendix

\section{Sheaves and their persistent version} \label{app_sheaves}

A sheaf is a structure that organizes information over a topological space. For example, a sheaf of vector spaces over a topological space $X$ is a collection of vector spaces and linear maps parameterized by the open sets of $X$. In this paper we will only work with sheaves in the discrete setting of finite posets. The theory of sheaves over posets is connected to the general theory of sheaves on topological spaces by the Alexandrov topology on posets \cite[Theorem 4.2.10]{curry}. In the following we will also view a poset $X$ as a category with a unique morphism $\sigma\rightarrow \tau$ if and only if $\sigma\leq \tau\in X$.  

\begin{definition}[Sheaves and sheaf morphisms over poset]
A sheaf $F$ on a finite poset $X$ with values in a category $\mathbf{C}$ is a functor $F\colon X\rightarrow \mathbf{C}$. In other words, $F$ is an assignment of an object $F(\sigma)\in\mathbf{C}$ to every element $\sigma\in X$ and an assignment of a morphism $F(\sigma)\xrightarrow{F(\sigma\leq\tau)}F(\tau)$ to every relation $\sigma\leq \tau\in X$.

A morphism of sheaves $\phi\colon F\rightarrow G$ on $X$ is a natural transformation of the functors $F$ and $G$, i.e.\ a collection of morphisms $\phi(\sigma)\colon F(\sigma)\rightarrow G(\sigma)$ such that $G(\sigma\leq\tau)\circ \phi(\sigma)=\phi(\tau)\circ F(\sigma\leq\tau)$ for all $\sigma\leq\tau\in X$.
\end{definition}

\noindent
Our main example will be sheaves over finite abstract simplicial complexes where a simplicial complex $K$ is viewed as a poset with $\sigma\leq\tau\in K$ whenever $\sigma$ is a face of $\tau$. Figure \ref{fig:sheaf_example_main} shows an example of sheaves (black) on a complex comprising a triangle and its faces and a sheaf morphism (red). Sheaves over $X$ with values in $\mathbf{C}$ and their morphisms form a category which we denote by $\mathbf{Shv}\big(X,\mathbf{C}\big)$. In the following, we will work with the categories of sheaves $\mathbf{Shv}\big(X,\mathbf{vec}\big)$ of finite-dimensional vector spaces $\mathbf{vec}$ or sheaves $\mathbf{Shv}\big(X,\mathbf{grMod}_{\field[t]}\big)$ of finitely generated graded modules $\mathbf{grMod}_{\field[t]}$ over the polynomial ring $\field[t]$.

Sheaves with values in a suitable category $\mathbf{C}$ have a notion of cohomology, i.e., for every $k\geq 0$, there is a cohomology functor $H^k(X,-)\colon\mathbf{Shv}\big(X,\mathbf{C}\big)\rightarrow \mathbf{C}$ defined by $F\mapsto H^k(X,F)$ and $(F\xrightarrow{\phi}G)\mapsto \big(H^k(X,F)\xrightarrow{H^k(X,\phi)}H^k(X,G)\big)$. If $K$ is a simplicial complex, the cohomology of $F$ can be computed by a cochain complex similiar to the one computing simplicial cohomology \cite{curry}. In the following, we denote by $K^k$ the set of $k$-simplices of $K$ and by $[\sigma:\tau]$ a signed incidence relation on $K$. Such an incidence relation can be defined as follows. Choose an arbitrary ordering of the vertex set $K^0=(v_0,\ldots,v_d)$. A simplex $\tau\in K$ can be expressed as $\tau=(v_{i_1},\ldots , v_{i_t})$ such that $i_j<i_{j+1}$. If $\sigma$ is a codimension 1 face of $\tau$, denoted by $\sigma <_1 \tau$, then $\sigma=(v_{i_1},\ldots,v_{i_{j\minus 1}},v_{i_{j+1}},\ldots,v_{i_t})$. Now we define
\begin{equation*}
[\sigma:\tau]\coloneqq \begin{cases} (\minus 1)^{j} \hspace{5pt} \text{if }\sigma<_1 \tau \\
0 \hspace{23pt} \text{else}\end{cases} .
\end{equation*}
The cohomology $H^k(K,F)$ of a sheaf of finite-dimensional vector spaces $F$ on a finite simplicial complex $K$ is given by the cohomology of the following cochain complex
\begin{equation} \label{cochain_complex}
\begin{tikzcd}
0\arrow[r] & C^0(K,F) \arrow[r,"\delta^0"] &  C^1(K,F) \arrow[r,"\delta^1"] &  C^2(K,F) \arrow[r,"\delta^2"] & \cdots
\end{tikzcd}
\end{equation}
where
\begin{equation} \label{cochain_complex_def}
\begin{aligned}
C^k(K,F)&\coloneqq\bigoplus_{\sigma\in K^k} F(\sigma) \\
\delta^k_{\sigma,\tau}&\coloneqq [\sigma:\tau]F(\sigma\leq\tau) 
\end{aligned}
\end{equation}
and $\delta^k_{\sigma,\tau}$ denotes the block of the block-matrix $\delta^k\colon \underset{\sigma\in K^k}{\bigoplus}F(\sigma)\rightarrow\underset{\tau\in K^{k+1}}{\bigoplus}F(\tau)$ corresponding to the simplices $\sigma$ and $\tau$. Moreover, the morphism $H^k(K,\phi)\colon H^k(K,F)\rightarrow H^k(K,G)$ induced in cohomology by a sheaf morphism $\phi\colon F\rightarrow G$ is the morphism induced by the following morphism of cochain complexes
\begin{equation} \label{shv_cochain_morphism}
\begin{tikzcd}
0\arrow[r] & C^0(K,F) \arrow[r,"\delta^0"] \arrow[d,"C^0(K\text{,}\phi)"] &  C^1(K,F) \arrow[r,"\delta^1"] \arrow[d,"C^1(K\text{,}\phi)"] &  C^2(K,F) \arrow[r,"\delta^2"] \arrow[d,"C^2(K\text{,}\phi)"] & \cdots \\
0\arrow[r] & C^0(K,G) \arrow[r,"\delta^0"] &  C^1(K,G) \arrow[r,"\delta^1"] &  C^2(K,G) \arrow[r,"\delta^2"] & \cdots
\end{tikzcd}
\end{equation}
where $C^k(K,\phi)$ is the block matrix defined by
\begin{equation*}
C_k(K,\phi)_{\sigma\tau}\coloneqq
\begin{cases}
\phi(\sigma) \hspace{5pt} \text{if} \hspace{5pt} \sigma=\tau \\
0 \hspace{23pt} \text{else}
\end{cases}
\end{equation*}
for all $\sigma,\tau\in K^k$.

We now come to persistent sheaves. In topological data analysis we study the evolution of topological spaces under variation of a parameter. We can think of them as persistence modules of topological spaces or persistent spaces. A persistent sheaf describes the evolution of algebraic information over a fixed space $X$. 

\begin{definition}[Persistent sheaf]
A \emph{persistent sheaf} is a functor
\begin{equation*}
\Vec{F}\colon \mathbb{N}_0\rightarrow \mathbf{Shv}\big(X,\mathbf{vec}\big) \hspace{2pt}.
\end{equation*}
We call $\Vec{F}$ of \emph{finite type} if $\exists m\in \mathbb{N}_0$ so that $\Vec{F}(i\leq j)$ is an isomorphism $\forall m\leq i\leq j$. 
\end{definition}

\noindent 
The following theorem shows that a persistent sheaf (Figure \ref{fig_pers_shv} left) can also be viewed as a sheaf of graded modules (Figure \ref{fig_pers_shv} right), or by Proposition \ref{equivalence_persmod}, a sheaf of persistence modules (Figure \ref{fig:psheaf-example_main}). The significance of this theorem is that, as illustrated in Figure \ref{fig_pers_shv}, it allows for a compressed representation of a persistent sheaf.

\begin{theorem}\cite[Theorem 3.6]{russold} \label{equ_shvgr_shvper}
A persistent sheaf $\Vec{F}\colon \mathbb{N}_0\rightarrow \mathbf{Shv}\big(X,\mathbf{vec}\big)$ of finite type corresponds to a sheaf of finitely generated $\mathbb{N}_0$-graded $\field[t]$-modules $\mathcal{M}\Vec{F}$. More precisely, if $X$ is a finite poset, there is an equivalence of categories
\begin{equation*}
\mathcal{M}\colon\mathbf{Fun}_{ft}\Big(\mathbb{N}_0,\mathbf{Shv}\big(X,\mathbf{vec}\big) \Big) \xlongrightarrow{\cong} \mathbf{Shv}\big(X,\mathbf{grMod}_{\field[t]}\big) \hspace{2pt}.
\end{equation*}
\end{theorem}
When $X$ is a single point, $\mathbf{Shv}(X,\mathbf{C})\cong\mathbf{C}$. Thus persistent sheaves over a point are equivalent to persistence modules. Hence, Proposition~\ref{equivalence_persmod} is the special case of Theorem \ref{equ_shvgr_shvper} when $X$ is a single point.
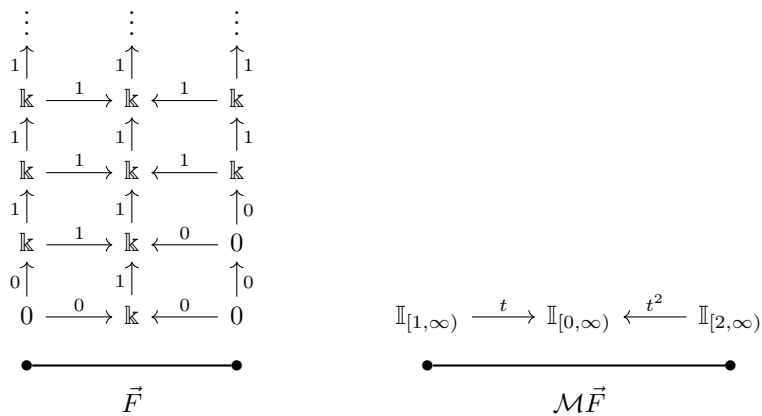
\begin{figure}
\begin{tikzcd}
\vdots & \vdots & \vdots \\[-5pt]
\field \arrow[r,"1"] \arrow[u,"1"{yshift=-1pt}] & \field \arrow[u,"1"{yshift=-1pt}] & \arrow[l,swap,"1"] \field \arrow[u,swap,"1"{yshift=-1pt}] \\[-5pt]
\field \arrow[r,"1"] \arrow[u,"1"{yshift=-1pt}] & \field \arrow[u,"1"{yshift=-1pt}] & \arrow[l,swap,"1"] \field \arrow[u,swap,"1"{yshift=-1pt}] \\[-5pt]
\field \arrow[r,"1"] \arrow[u,"1"{yshift=-1pt}] & \field \arrow[u,"1"{yshift=-1pt}] & \arrow[l,swap,"0"] 0 \arrow[u,swap,"0"{yshift=-1pt}] \\[-5pt]
0 \arrow[r,"0"] \arrow[u,"0"{yshift=-1pt}] & \field \arrow[u,"1"{yshift=-1pt}] & \arrow[l,swap,"0"] 0 \arrow[u,swap,"0"{yshift=-1pt}] & & \mathbb{I}_{[1,\infty)} \arrow[r,"t"] & \mathbb{I}_{[0,\infty)} & \mathbb{I}_{[2,\infty)} \arrow[l,swap,"t^2"] \\[-15pt]
\bullet \arrow[rr,thick,dash,shorten=-5pt] & & \bullet & & \bullet \arrow[rr,thick,dash,,shorten=-5pt] & & \bullet \\[-20pt]
& \Vec{F} & & & & \mathcal{M}\Vec{F}
\end{tikzcd}
\caption{A persistent sheaf $\Vec{F}$ and the corresponding sheaf of graded $\field[t]$-modules $\mathcal{M}\Vec{F}$.}
\label{fig_pers_shv}
\end{figure}

Our goal is to compute the persistent cohomology of $\Vec{F}$, i.e.\ the barcode of the persistence module $H^k(X,\Vec{F})\coloneqq H^k(X,-)\circ\Vec{F}\colon \N\rightarrow \mathbf{vec}$ obtained by composing $\Vec{F}$ with the sheaf cohomology functor $H^k(X,-)\colon \mathbf{Shv}(X,\mathbf{vec})\rightarrow \mathbf{vec}$. 

If $K$ is a finite simplicial complex, the sheaf cohomology persistence module $H^k(K,\Vec{F})$ can be computed by pointwise computation of all cohomology vector spaces $H^k\big(K,\Vec{F}(i)\big)$ and induced morphisms $H^k\big(K,\Vec{F}(i\leq i+1)\big)$ from the cochain complexes given in~\eqref{shv_cochain_morphism}.  
The following theorem suggests an alternative method of computation. 

\begin{theorem}\cite[Theorem 3.8]{russold} \label{russo-thm}
The persistent cohomology of $\Vec{F}$ corresponds to the cohomology of the sheaf of graded $\field[t]$-modules $\mathcal{M}\Vec{F}$, i.e.,\
\begin{equation*}
\mathcal{M}H^k(X,\Vec{F})\cong H^k(X,\mathcal{M}\Vec{F}).
\end{equation*}
\end{theorem}

\noindent
Over a simplicial complex, the cohomology of $\mathcal{M}\Vec{F}$ can be computed from a cochain complex in the same way as the cohomology of a sheaf of vector spaces by replacing the sheaf of vector spaces $F$ in~\eqref{cochain_complex} and \eqref{cochain_complex_def} by the sheaf of graded modules $\mathcal{M}\Vec{F}$ \cite{russold}. In Section~\ref{sec_pers_shv_simp} we present an algorithm motivated by this viewpoint.


\section{Cosheaves over persistent spaces} \label{app_cosheaves}

The notion of a cosheaf is dual to the one of a sheaf. Over finite posets they are equivalent \cite[Definition 6.2.8]{curry}. Every result for sheaves over finite posets can be translated to cosheaves by dualization. 

\begin{definition}[Cosheaves and cosheaf morphisms over poset]
A cosheaf $F$ on a finite poset $X$ with values in a category $\mathbf{C}$ is a functor $F\colon X^\text{op}\rightarrow \mathbf{C}$. In other words, $F$ is an assignment of an object $F(\sigma)\in\mathbf{C}$ to every element $\sigma\in X$ and an assignment of a morphism $F(\tau)\xrightarrow{F(\sigma\leq\tau)}F(\sigma)$ to every relation $\sigma\leq \tau\in X$.

A morphism of cosheaves $\phi\colon F\rightarrow G$ on $X$ is a natural transformation of the functors $F$ and $G$, i.e.\ a collection of morphisms $\phi(\sigma)\colon F(\sigma)\rightarrow G(\sigma)$ such that $G(\sigma\leq\tau)\circ \phi(\tau)=\phi(\sigma)\circ F(\sigma\leq\tau)$ for all $\sigma\leq\tau\in X$.
\end{definition}

\noindent
In this section we want to study the behavior of a "fixed" cosheaf over a varying space. 

\begin{definition}[Persistent space]
A \emph{persistent space} is a functor 
\begin{equation*}
\Vec{X}\colon \overline{\mathbb{N}}_0\rightarrow \mathbf{Top} \hspace{1pt} 
\end{equation*}
where $\overline{\mathbb{N}}_0\coloneqq \mathbb{N}_0\cup \{\infty\}$.
The persistent space $\Vec{X}$ is of \emph{finite type}, if $\exists m\in \mathbb{N}_0$ such that $\Vec{X}(i\leq j)$ is an isomorphism $\forall m\leq i\leq j$.
\end{definition}

\begin{example}
An abstract simplicial complex viewed as a poset equipped with the Alexandrov topology can be viewed as a finite topological space. Let $K_0\subseteq K_1\subseteq \cdots \subseteq K_m$ be a filtered simplicial complex. Then, $\Vec{K}(i)\coloneqq K_i$ for $0\leq i< m$ and $\Vec{K}(i)\coloneqq K_m$ for $m\leq i$ defines a persistent space of finite type.
\end{example}

\noindent
Let $\Vec{X}$ be a persistent space and $F=F^\infty$ a cosheaf on $\Vec{X}(\infty)$. In the following, we use the notation $X_i\coloneqq \Vec{X}(i)$ and $f_i\coloneqq \Vec{X}(i\leq i+1)$. The cosheaf $F^\infty$ defines a cosheaf $F^i$ on $X_i$ for all $i\in\mathbb{N}_0$ via the pull back $F^i\coloneqq \Vec{X}(i\leq \infty)^* F^\infty$. Since $\Vec{X}(i\leq \infty)=\Vec{X}(j\leq\infty)\circ\Vec{X}(i\leq j)$ for all $i\leq j\leq \infty$, we obtain $F^i\cong\Vec{X}(i\leq \infty)^* F^\infty\cong\Vec{X}(i\leq j)^*\Vec{X}(j\leq \infty)^*F^\infty\cong\Vec{X}(i\leq j)^*F^j$. Hence, we can also define the $F^i$ by iteratively pulling back $F^\infty$ to every space $X_i$ as depicted in the following diagram
\begin{equation} \label{pullback_diag}
\begin{tikzcd} 
F^0 & \arrow[l,swap,mapsto,"f_0^*"] F_1 & \arrow[l,swap,mapsto,"f_1^*"] \cdots & \arrow[l,swap,mapsto,"f_{i\minus 1}^*"] F^{i} & \arrow[l,swap,mapsto,"f_{i}^*"] F^{i+1} & \arrow[l,swap,mapsto,"f_{i+1}^*"] \cdots & \arrow[l,mapsto] F^\infty \\[-10pt]
X_0 \arrow[r,"f_0"] & X_1 \arrow[r,"f_1"] & \cdots \arrow[r,"f_{i\minus 1}"] & X_{i} \arrow[r,"f_{i}"] & X_{i+1} \arrow[r,"f_{i+1}"] & \cdots \arrow[r] & X_\infty
\end{tikzcd}
\end{equation}
Our goal is to compute the persistent homology of $F^\infty$ over $\Vec{X}$ which we define in the following way: Let $H_k(f_i)\colon H_k(X_{i},f_{i}^*F^{i+1})\rightarrow H_k(X_{i+1},F^{i+1}\big)$ denote the induced morphism in homology. Then~\eqref{pullback_diag} induces the following persistence module
\begin{equation} \label{cohomology_copersistence_module}
\begin{tikzcd}[column sep=large]
H_k(X_0,F^0) \arrow[r,"H_k(f_0)"] & H_k(X_1,F^1) \arrow[r,"H_k(f_1)"] &  H_k(X_2,F^2) \arrow[r,"H_k(f_2)"] & \cdots
\end{tikzcd}
\end{equation}
denoted by $H_k(\Vec{X},F)\colon \mathbb{N}_0\rightarrow \mathbf{vec}$. We call $H_k(\Vec{X},F)$ the homology persistence module of $(\Vec{X},F)$. If $\Vec{X}$ is of finite type, then \eqref{cohomology_copersistence_module} is determined up to isomorphism by the finite persistence module of $F^m$ over the finite persistent space in \eqref{cohomology_copersistence_module_finite}
\begin{equation} \label{cohomology_copersistence_module_finite}
\begin{tikzcd}[column sep=large]
F^0 & \arrow[l,swap,mapsto,"f_0^*"] F^1 & \arrow[l,swap,mapsto,"f_1^*"] \cdots &[8pt] \arrow[l,swap,mapsto,"f_{m\minus 1}^*"] F^m \\[-10pt]
X_0 \arrow[r,"f_0"] & X_1 \arrow[r,"f_1"] & \cdots \arrow[r,"f_{m\minus 1}"] & X_m \\[-10pt]
H_k(X_0,F^0) \arrow[r,"H_k(f_0)"] & H_k(X_1,F^1) \arrow[r,"H_k(f_1)"] & \cdots \arrow[r,"H^k(f_{m\minus 1})"] & H_k(X_{m},F^{m}) 
\end{tikzcd}
\end{equation}
In practice, the finite persistent space in \eqref{cohomology_copersistence_module_finite} will most likely be a simplicial tower, i.e.\ a sequence of arbitrary simplicial maps. In this case, we can explicitly construct chain complexes and morphisms  computing the persistent homology. Let $f\colon K\rightarrow L$ be a simplicial map and $F$ be a cosheaf on $L$. The pullback $f^*F$ on $K$ is defined by $f^*F(\sigma)\coloneqq F\big(f(\sigma)\big)$. The $k$-chains of $F$ and $f^*F$ are defined by
\begin{equation*}
\begin{aligned}
C_k(L,F) &\coloneqq \bigoplus_{\tau\in L^k} F(\tau) \\ 
C_k(K,f^*F) &\coloneqq\bigoplus_{\sigma\in K^k} f^*F(\sigma)= \bigoplus_{\sigma\in K^k} F\big(f(\sigma)\big)
\end{aligned}
\end{equation*}
The induced map $C_k(f)\colon C_k(K,f^*F)\rightarrow C_k(L,F)$ is defined by the block-matrix with the following entries: For $\sigma\in K^k$ and $\tau\in L^k$
\begin{equation} \label{induced_maps_cochain}
\begin{aligned}
& C_k(f)_{\sigma\tau}\colon F\big(f(\sigma)\big)\rightarrow F(\tau) \\[5pt]
& C_k(f)_{\sigma\tau}=\begin{cases} \text{id} \hspace{5pt} \text{ : if } f(\sigma)=\tau \\ 0 \hspace{5pt} \text{ : else} \end{cases}
\end{aligned}
\end{equation}
The homology of a cosheaf over a simplicial complex is computed by a chain complex dual to the one in~\eqref{cochain_complex}. Given a simplicial tower as in \eqref{cohomology_copersistence_module_finite}, the cosheaf chain complexes together with the chain morphisms of ~\eqref{induced_maps_cochain} induced by the $f_i$ yield the following chain complex of persistence modules 
\begin{equation} \label{cochain_comp_typeT}
\begin{tikzcd}[column sep=huge]
C_{k+1}\big(K_0,F^0\big) \arrow[d,"\partial_{k+1}^0"] \arrow[r,"C_{k+1}(f_{0})"] & C_{k+1}\big(K_1,F^1\big) \arrow[d,"\partial_{k+1}^1"] \arrow[r,"C_{k+1}(f_{1})"] & \cdots \arrow[r,"C_{k+1}(f_{m\minus 1})"] & C_{k+1}\big(K_{m},F^m\big) \arrow[d,"\partial_{k+1}^m"] \\
C_{k}\big(K_0,F^0\big) \arrow[d,"\partial_{k}^0"] \arrow[r,"C_{k}(f_{0})"] & C_{k}\big(K_1,F^1\big) \arrow[d,"\partial_{k}^1"] \arrow[r,"C_{k}(f_{1})"] & \cdots \arrow[r,"C_{k}(f_{m\minus 1})"] & C_{k}\big(K_{m},F^m\big) \arrow[d,"\partial_{k}^m"] \\
C_{k\minus 1}\big(K_0,F^0\big) \arrow[r,"C_{k\minus 1}(f_{0})"] &  C_{k\minus 1}\big(K_1,F^1\big) \arrow[r,"C_{k\minus 1}(f_{1})"] & \cdots \arrow[r,"C_{k\minus 1}(f_{m\minus 1})"] & C_{k\minus 1}\big(K_{m},F^m\big)
\end{tikzcd}
\end{equation}
which computes the persistent homology of $F$ on $\Vec{K}$. By extending \eqref{cochain_comp_typeT} by isomorphisms we view 
\begin{equation*} 
\begin{tikzcd}
C_k(K_0,F^0) \arrow[r,"C_k(f_{0})"] &[10pt] C_k(K_1,F^1) \arrow[r,"C_k(f_{1})"] &[8pt] \cdots \arrow[r,"C_k(f_{m\minus 1})"] &[12pt] C_k(K_m,F^m) \arrow[r,"\cong"] &[-3pt] \cdots
\end{tikzcd}
\end{equation*}
denoted by $C^k(\Vec{K},F)$ as a persistence module of finite type and denote the corresponding chain complex of persistence modules by
\begin{equation} \label{typeT_pers_complex}
\begin{tikzcd}
C_{k+1}(\Vec{K},F) \arrow[r,"\partial_{k+1}"] & C_{k}(\Vec{K},F) \arrow[r,"\partial_{k}"] & C_{k\minus 1}(\Vec{K},F)
\end{tikzcd}
\end{equation}
As in the case of persistent sheaves, we have two options to compute the persistent homology. The first option is to compute the homology $H_k(K_i,F^i)$, $H_k(K_{i+1},F^{i+1})$ and the induced morphism $H_k(f_i)\colon H_k(K_{i},F^{i})\rightarrow H_k(K_{i+1},F^{i+1})$ pointwise. The second option is to compute the homology of the chain complex of persistence modules~\eqref{typeT_pers_complex} in the category of graded modules.

\begin{example}
Let $\Vec{K}$ be a simplicial tower as in \eqref{cohomology_copersistence_module_finite} and $F$ be the constant cosheaf on $K_m$ defined by $F(\sigma)=\field$ and $F(\sigma\leq\tau)=\text{id}$ for all $\sigma\leq\tau\in K_m$. In this case $f_i^*F^i$ will be the constant cosheaf on $K_i$ for all $0\leq i<m$. It is easy to see that for constant cosheaves the chain complexes dual to the ones in \eqref{cochain_complex} and the chain complex morphisms defined by \eqref{induced_maps_cochain} are just the chain complexes of the underlying simplicial complexes and simplicial maps. Hence, the homology of the constant cosheaf computes the homology of the underlying space and we obtain
\begin{equation}\label{constant_cosheaf_homology}
\begin{tikzcd}[column sep=large]
H_k(K_i,F^i) \arrow[r,"H_k(f_i)"] \arrow[d,swap,"\cong"] & H_k(K_{i+1},F^{i+1}) \arrow[d,"\cong"] \\
H_k(K_i) \arrow[r,"H_k(f_i)"] & H_k(K_{i+1})
\end{tikzcd}
\end{equation}
Moreover, \eqref{constant_cosheaf_homology} implies that the persistent homology of the constant cosheaf over the tower is isomorphic to the persistent homology of the underlying tower. This shows that persistent homology is just a special case of persistent cosheaf homology. 
\end{example}


\cancel{

\section{Proofs in Section \ref{sec_cohomology_presentations}} \label{app_cohomology_presentations}
\cancel{
\begin{proposition} \label{appendix_prop_kernel}
The following exact sequence
\begin{equation*}
\begin{tikzcd}
0 & \ker\psi \arrow[l] & \ker(g_0\minus r) \arrow[l,swap,"\kappa"] &[15pt] Q_1 \arrow[l,swap,"\begin{pmatrix}q \\ g_1\end{pmatrix}"]
\end{tikzcd} 
\end{equation*}
is a presentation of $\ker\psi$.
\end{proposition}
}

\begin{proof}[Proof of Proposition~\ref{presentation_kernel}]
~
\setstretch{1.2}
\begin{enumerate}
\item Let $a\in\ker\psi$. Since $\mu$ is an epimorphism, $\exists x\in Q_0$ s.t.\ $\mu(x)=a$ and $\psi\circ \mu(x)=0=\nu\circ g_0(x)$. We get $g_0(x)\in\ker \nu=\im r$ and $\exists y\in R_1$ s.t.\ $r(y)=g_0(x)$. Thus, $(g_0\minus r)(x,y)=0$ and $\kappa(x,y)=\mu(x)=a$. Hence, $\kappa$ is an epimorphism.
\item Let $(x,y)\in \ker \kappa$. Then, $\kappa(x,y)=\mu(x)=0$ and $x\in\ker \mu=\im q$. Thus, $\exists a\in Q_1$ s.t.\ $q(a)=x$. Since, $g_0(x)=g_0\circ q(a)=r\circ g_1(a)=r(y)$, we get $g_1(a)-y\in\ker r=0$. This implies, $g_1(a)=y$ and $\begin{pmatrix}q \\ g_1\end{pmatrix}(a)=(q(a),g_1(a))=(x,y)$. Hence, $\ker \kappa\subseteq \im \begin{pmatrix}q \\ g_1\end{pmatrix}$.

Conversely, let $(x,y)\in\im \begin{pmatrix}q \\ g_1\end{pmatrix}$. Then $\exists a\in Q_1$ s.t.\ $\begin{pmatrix}q \\ g_1\end{pmatrix}(a)=(q(a),g_1(a))=(x,y)$. We get $\kappa(x,y)=\kappa(q(a),g_1(a))=\mu\circ q(a)=0$. Hence, $(x,y)\in\ker \kappa$ and $\ker \kappa=\im \begin{pmatrix}q \\ g_1\end{pmatrix}$.
\end{enumerate}
\end{proof}

\cancel{
\begin{proposition} \label{appendix_prop_commutativity}
The following diagram commutes:
\begin{equation*}
\begin{tikzcd}
0 & L \arrow[l] \arrow[d,"\phi"] &[10pt] P_0 \arrow[l,swap,"\lambda"] \arrow[d,"\begin{pmatrix}f_0 \\ 0\end{pmatrix}"] &[15pt] P_1 \arrow[l,swap,"p"] \arrow[d,"f_1"] \\[10pt]
0 & \ker\psi \arrow[l] & \ker(g_0\minus r) \arrow[l,swap,"\kappa"] & Q_1 \arrow[l,swap,"\begin{pmatrix}q \\ g_1\end{pmatrix}"]
\end{tikzcd} 
\end{equation*}
\end{proposition}
}

\begin{proof}[Proof of Proposition~\ref{kernel_lemma}]
~
\setstretch{1.2}
\begin{enumerate}
\item Let $x\in P_0$. Since $g_0\circ f_0=0$, we have $(g_0\minus r)(f_0(x),0)=g_0(f_0(x))=0$ and $\im\begin{pmatrix}f_0 \\ 0\end{pmatrix}\subseteq \ker(g_0\minus r)$.
\item Let $x\in P_1$. Since $g_1\circ f_1=0$ we get $\begin{pmatrix}q \\ g_1\end{pmatrix}\circ f_1(x)=\big(q(f_1(x)),g_1(f_1(x))\big)=\big(f_0(p(x)),0\big)=\begin{pmatrix}f_0 \\ 0\end{pmatrix}\circ p(x)$. Hence, the right square commutes.
\item Let $x\in P_0$. Then $\kappa\circ \begin{pmatrix}f_0 \\ 0\end{pmatrix}(x)=\kappa(f_0(x),0)=\mu(f_0(x))=\phi\circ \lambda(x)$. Thus, also the left square commutes.
\end{enumerate}
\end{proof}

\cancel{
\begin{theorem} \label{appendix_theorem_cohomology_presentation}
The following exact sequence 
\begin{equation*}
\begin{tikzcd}
0 & \ker\psi / \im\phi \arrow[l] &[5pt] \ker(g_0\minus r) \arrow[l,swap,"\pi\circ \kappa"] &[35pt] P_0\times Q_1 \arrow[l,swap,"\begin{pmatrix}f_0 \hspace{7pt} \minus q \\ 0 \hspace{9pt} \minus g_1 \end{pmatrix}"]
\end{tikzcd} 
\end{equation*}
is a presentation of $\ker\psi/\im\phi$.
\end{theorem}
}

\begin{proof}[Proof of Theorem~\ref{thm_cohomology_presentation}]
~
\setstretch{1.2}
\begin{enumerate}
\item Since $\kappa$ is an epimorphism by Proposition
\ref{presentation_kernel}
and $\pi$ is an epimorphism as a quotient map, $\pi\circ \kappa$ is also an epimorphism.
\item Let $(x,y)\in \ker(\pi\circ \kappa)$ i.e.\ $\pi\circ \kappa(x,y)=0$ and $\kappa(x,y)\in\ker\pi=\im\phi$. Hence, $\exists a\in L$ s.t.\ $\phi(a)=\kappa(x,y)$ and, moreover, $\exists b\in P_0$ s.t.\ $\lambda(b)=a$. By Proposition
\ref{kernel_lemma},
we get $\phi\circ \lambda(b)=\kappa\circ \begin{pmatrix}f_0 \\ 0\end{pmatrix}(b)=\kappa(x,y)$ and $\begin{pmatrix}f_0 \\ 0\end{pmatrix}(b)-(x,y)\in\ker \kappa=\im\begin{pmatrix}q \\ g_1\end{pmatrix}$. Thus, $\exists z\in Q_1$ s.t.\ $\begin{pmatrix}q \\ g_1\end{pmatrix}(z)=\begin{pmatrix}f_0 \\ 0\end{pmatrix}(b)-(x,y)$ and $(x,y)=\begin{pmatrix}f_0 \\ 0\end{pmatrix}(b)-\begin{pmatrix}q \\ g_1\end{pmatrix}(z)=\begin{pmatrix}f_0 & \minus q \\ 0 & \minus g_1 \end{pmatrix}(b,z)$. Therefore, $(x,y)\in\im\begin{pmatrix}f_0 & \minus q \\ 0 & \minus g_1 \end{pmatrix}$ and $\ker \kappa\subseteq \im\begin{pmatrix}f_0 & \minus q \\ 0 & \minus g_1 \end{pmatrix}$.

Conversely, let $(x,y)\in\im\begin{pmatrix}f_0 & \minus q \\ 0 & \minus g_1 \end{pmatrix}$ i.e.\ $(x,y)=\begin{pmatrix}f_0 & \minus q \\ 0 & \minus g_1 \end{pmatrix}(a,b)$ for some $(a,b)\in P_0\times Q_1$. We get $\kappa(x,y)=\kappa(f_0(a),0)-\kappa(q(b),g_1(b))=\mu\circ f_0(a)-\mu\circ q(b)=\phi\circ \lambda(a)\in\im\phi$. Hence, $\pi\circ \kappa(x,y)=0$, $(x,y)\in\ker\pi\circ \kappa$ and $\ker \kappa=\im\begin{pmatrix}f_0 & \minus q \\ 0 & \minus g_1 \end{pmatrix}$.
\end{enumerate}
\end{proof}
\cancel{
\begin{theorem}
Given a complex of presentations as in \eqref{comp_presentations}, represented by the matrices $p,q,r,f_0,f_1,g_0,g_1$ of \tamal{dimensions $O(n)\times O(n)$, the} algorithm described in Section \ref{sec_cohomology_presentations} computes the barcode of the cohomology in $O(n^3)$ time.
\end{theorem}
}
\begin{proof}[Proof of Theorem~\ref{thm:perscohom}]
From the input we can construct the matrices $\begin{pmatrix}g_0 & \minus r\end{pmatrix}$ and $\begin{pmatrix}f_0 & \minus q \\ 0 & \minus g_1 \end{pmatrix}$ and order the columns by degree of the generators. After a column reduction from left to right, the zero columns of the reduced matrix span $\ker\begin{pmatrix}g_0 & \minus r\end{pmatrix}$. Removing the rows of  $\begin{pmatrix}f_0 & \minus q \\ 0 & \minus g_1 \end{pmatrix}$ corresponding to the non-zero columns of the reduced matrix yields 
\begin{equation} \label{appendix_presentation_matrix_cohomology}
\ker \begin{pmatrix}g_0 & \minus r\end{pmatrix} \xleftarrow{\begin{pmatrix}f_0 & \minus q \\ 0 & \minus g_1 \end{pmatrix}} P_0\times Q_1
\end{equation}
Note that we don't have to do any transformations on $\begin{pmatrix}f_0 & \minus q \\ 0 & \minus g_1 \end{pmatrix}$. When we perform the change of basis, we never add a zero column of $\ker\begin{pmatrix}g_0 & \minus r\end{pmatrix}$ to another column. Therefore, in the other matrix there is never a row added to a row corresponding to a zero column. Hence the only rows we care about stay unchanged.

Theorem~\ref{thm_cohomology_presentation}
shows that \eqref{appendix_presentation_matrix_cohomology} is a presentation of the homology $\ker\psi/\im\phi$. After column reducing the presentation matrix in \eqref{appendix_presentation_matrix_cohomology} from left to right, every column is either zero or has a unique pivot. We could now row reduce the presentation matrix and bring it to canonical form. After this step we could read off the elementary summands defining the barcode. Every row $r_j$ corresponds to a bar that is born at $\text{deg}(r_j)$. If it is a pivot row with pivot in column $c_{j_l}$ the bar dies at index $\text{deg}(c_{j_l})$. If it is a zero row the bar lives forever. Since the row reduction does not change the pivots and the bars only depend on the pivots, we can avoid the row reduction and read of the bars directly. Hence, after the column reduction every pivot row contributes a bar $[\text{deg}(r_j),\text{deg}(c_{j_l})\big)$ and every non-pivot row contributes a bar $[\text{deg}(r_j),\infty\big)$.    
\end{proof}

}


\section{Proofs in Section \ref{sec_presentation_algorithm}} \label{app_presentation_algo}

\textbf{Constructions:}
Consider a morphism of canonical presentations: 
\begin{equation} \label{morph123}
\begin{tikzcd}
P_0 \arrow[d,swap,"f_0"] & \arrow[l,swap,"p"] P_1 \arrow[d,"f_1"] \\
Q_0 & \arrow[l,swap,"q"] Q_1
\end{tikzcd}
\end{equation}
We can construct the birth- and death-time annotations for $f_0$ in the following way: If $c_j$ is a column of $f_0$, then the $j$-th row of $p$ has at most one non-zero entry. If there is a non-zero entry in column $g_{l_j}$ we set $b(c_j)=\text{deg}(c_j)$ and $d(c_j)=\text{deg}(g_{l_j})$. Otherwise, we set $b(c_j)=\text{deg}(c_j)$ and $d(c_j)=\infty$. Similarly we can construct the annotations of the rows of $f_0$. 

Conversely, if we have a matrix $f_0$ with birth- and death-time annotations we can construct a morphism of presentations as in \eqref{morph123} in the following way: Let $P_0=\text{span}\{c_1,\ldots,c_v\}$ and $Q_0=\text{span}\{r_1,\ldots,r_w\}$ where $c_j$ and $r_k$ are the columns and rows of $f_0$ with $\text{deg}(c_j)=b(c_j)$ and $\text{deg}(r_k)=b(r_k)$. Let $c_{j_s}$ and $r_{k_t}$ be the columns and rows of $f_0$ with $d(c_{j_s})<\infty$ and $d(r_{k_t})<\infty$. Define $P_1\coloneqq \text{span}\{g_s\}_s$, $Q_1\coloneqq \text{span}\{h_t\}_t$ with $\text{deg}(g_s)=d(c_{j_s})$, $\text{deg}(h_t)=d(r_{k_t})$ and $p(g_s)\coloneqq c_{j_s}$, $q(h_t)\coloneqq r_{k_t}$. Finally, let $f_1$ be the unique morphism making the diagram commute.

 Given the matrix $f_0$ with birth- and death-time annotations, we can construct the corresponding morphism of persistence modules \eqref{MPM} in the following way: First we define the columns $c_j$ and rows $r_k$ that are active at index $i$:
\begin{equation*}
\begin{aligned}
& G_M^i\coloneqq\{c_j\vert b(c_j)\leq i \text{ and } d(c_j)>i\} \\
& G_N^i\coloneqq\{r_k\vert b(r_k)\leq i \text{ and }d(r_k)>i\} 
\end{aligned}
\end{equation*}
Given a column $c_j$ of $f_0$, we have $f_0(c_j)=\sum_{l=1}^w \lambda_l^j r_l$. Using this relation, we define 
\begin{equation*}
\begin{aligned}
& M(i) \coloneqq \text{span }G_M^i \\
& N(i) \coloneqq \text{span }G_N^i \\
& M(i\leq i+1)(c_j)\coloneqq \begin{cases} c_j \text{ if }c_j\in G_M^{i+1} \\ 0 \text{ else} \end{cases} \\
& N(i\leq i+1)(r_k)\coloneqq \begin{cases} r_k \text{ if }r_k\in G_N^{i+1} \\ 0 \text{ else} \end{cases} \\
& \phi(i)(c_j)\coloneqq  \sum_{l: r_l\in G_N^i} \lambda_l^j r_l
\end{aligned}
\end{equation*}
It is easy to see that the morphism of persistence modules defined in this way corresponds to the morphism of graded modules $\text{coker }p\xrightarrow{} \text{coker }q$ induced by the morphism of presentations defined by $f_0$.

\cancel{
\begin{theorem}
The algorithm described in Section \ref{sec_presentation_algorithm} is correct. If $\text{dim}(M_i),\text{dim}(N_i)=O(n_i)$ and $\sum_{i=0}^m n_i=n$ the algorithm takes $O(n^3)$ time.
\end{theorem}
}

\begin{proof}[Proof of Theorem~\ref{thm:bar-from-matrices}]
Given a morphism of persistence modules $\phi\colon M\rightarrow N$ of finite type 
\begin{equation} \label{morphism_persistence_modules_app}
\begin{tikzcd}
M_0 \arrow[r,"A_0"] \arrow[d,"C_0"] & M_1 \arrow[r,"A_1"] \arrow[d,"C_1"] & \cdots \arrow[r,"A_{m\minus 1}"] &[3pt] M_m \arrow[d,"C_m"] \arrow[r,"\cong"] & \cdots \\
N_0 \arrow[r,"B_0"] & N_1 \arrow[r,"B_1"] & \cdots \arrow[r,"B_{m\minus 1}"] & N_m \arrow[r,"\cong"] & \cdots
\end{tikzcd}
\end{equation}
represented by the sequences of matrices $(A_0,\ldots, A_{m\minus 1})$, $(B_0,\ldots, B_{m\minus 1})$ and $(C_0,\ldots,C_m)$. We show by induction that the algorithm maintains a canonical presentation 
\begin{equation} \label{ith_presentation_app}
\begin{tikzcd}
P_0^l \arrow[d,swap,"f_0^l"] & \arrow[l,swap,"p^l"] P_1^l \arrow[d,"f_1^l"] \\
Q_0^l & \arrow[l,swap,"q^l"] Q_1^l
\end{tikzcd}
\end{equation}
inducing a persistence module via the construction above   
\begin{equation} \label{ith_restriction_iso_app}
\begin{tikzcd}
M^l_0 \arrow[r,"A^l_0"] \arrow[d,"C^l_0"] & \cdots \arrow[r,"A^l_{l\minus 1}"] & M^l_l \arrow[d,"C^l_l"] \arrow[r,"\text{id}"] & M^l_l \arrow[r,"\text{id}"] \arrow[d,"C^l_l"] & \cdots  \\
N^l_0 \arrow[r,"B^l_0"]  & \cdots \arrow[r,"B^l_{l\minus 1}"] & N^l_l \arrow[r,"\text{id}"] & N^l_l \arrow[r,"\text{id}"] & \cdots
\end{tikzcd}
\end{equation}
that is isomorphic to 
\begin{equation} \label{ith_restriction_app}
\begin{tikzcd}
M_0 \arrow[r,"A_0"] \arrow[d,"C_0"] & \cdots \arrow[r,"A_{l\minus 1}"] & M_l \arrow[d,"C_l"] \arrow[r,"\text{id}"] & M_l \arrow[r,"\text{id}"] \arrow[d,"C_l"] & \cdots  \\
N_0 \arrow[r,"B_0"]  & \cdots \arrow[r,"B_{l\minus 1}"] & N_l \arrow[r,"\text{id}"] & N_l \arrow[r,"\text{id}"] & \cdots
\end{tikzcd}
\end{equation}
while processing the morphism of persistence modules \eqref{morphism_persistence_modules_app} from left $l=0$ to right $l=m$. 

\vspace{0.1in}
\textbf{Induction base $l=0$:} 
Let $M_0=\text{span}\{c_1,\ldots,c_{d_0}\}$, $N_0=\text{span}\{r_1,\ldots,r_{d'_0}\}$ and $C_0$ be given w.r.t.\ these bases. We set $M^0_0=M_0$, $N^0_0=N_0$ and  $C^0_0=C_0$. With these definitions \eqref{ith_restriction_iso_app} is clearly isomorphic to \eqref{ith_restriction_app} for $l=0$. In the algorithm we set $f^0_0=C_0$, $b(c_j)=0$, $d(c_j)=\infty$ and $b(r_k)=0$, $d(r_k)=\infty$ for the columns $c_j$ and rows $r_k$ of $f_0^0$. As discussed in the constructions above, on the presentation level this is equivalent to defining $P^0_0\coloneqq \text{span}\{c_1,\ldots,c_{d_0}\}$ with $\text{deg}(c_j)=0$, $Q^0_0\coloneqq \text{span}\{r_1,\ldots,r_{d'_0}\}$ with $\text{deg}(r_k)=0$, $f^0_0=C^0_0$ and $P^0_1, Q^0_1, p^0, q^0, f^0_1=0$. With these definitions \eqref{ith_presentation_app} is clearly a presentation inducing \eqref{ith_restriction_iso_app} for $l=0$. The algorithm proceeds with $f^0_0$, $A^t_0=A_0$ and $B^t_0=B_0$ to the next step.

\vspace{0.1in}
\textbf{Induction step $l=i$ to $l=i+1$:}  
We are given $f^i_0$, $A^t_i$ and $B^t_i$ from the previous step of the algorithm where $f^i_0$ defines a canonical presentation \eqref{ith_presentation_app} for $l=i$ and the columns and rows of $f^i_0$ are ordered w.r.t\ increasing birth-time. Our inductive hypothesis is that this presentation induces a persistence module \eqref{ith_restriction_iso_app} isomorphic to \eqref{ith_restriction_app} and that the matrices $A^t_i$ and $B^t_i$ are given w.r.t.\ the bases induced on $M^i_i$ and $N^i_i$ by the presentation. Since $A^t_i$ and $B^t_i$ are given w.r.t.\ the bases induced by $f^i_0$, every column of $A^t_i$ or $B^t_i$ corresponds to a column or row of $f^i_0$ and the columns of $A^t_i$ and $B^t_i$ inherit the column and row order of $f^i_0$. In particular, since \eqref{ith_presentation_app} is a presentation of \eqref{ith_restriction_iso_app}, every generator that is alive at index $i$ lives forever. Thus every column $c_j$ or row $r_k$ of $f^i_0$ corresponding to a column of $A^t_i$ or $B^t_i$ satisfies $d(c_j)=\infty$ or $d(r_k)=\infty$. The first step of the algorithm is to column reduce $A^t_i$ and $B^t_i$ from left to right while doing the corresponding column and row operations on $f^i_0$. The matrices $A^t_i$ and $B^t_i$ have size $O(n_{i+1})\times O(n_i)$. Hence, the reductions take $O(n_i^2\cdot n_{i+1})=O(n_i^2\cdot n)$ time. The matrix $f^i_0$ is of size $O(n)\times O(n)$ but since we are only adding columns or rows corresponding to columns of $A^t_i$ or $B^t_i$ we perform at most $O(n_i^2)$ column and row additions. Therefore, the operations on $f^i_0$ take $O(n^2_i\cdot n)$ time. After this step we obtain reduced matrices $A^r_i$ and $B^r_i$ and a new matrix $\overline{f}^i_0$. The column and row operations on $f^i_0$ correspond to a change of basis  $P^i_0\xrightarrow{\cong} \overline{P}^i_0$ and $Q^i_0\xrightarrow{\cong}\overline{Q}^i_0$. Since all the columns and rows of $f^i_0$ corresponding to columns of $A^t_i$ and $B^t_i$ satisfy $d(c_j),d(r_k)=\infty$, we can set $\overline{P}^i_1=P^i_1$, $\overline{Q}^i_1=Q^i_1$, $\overline{p}^i=p^i$, $\overline{q}^i=q^i$ and $\overline{f}^i_1=f^i_1$ and we obtain a presentation 
\begin{equation} \label{new_ith_pres}
\begin{tikzcd}
\overline{P}_0^i \arrow[d,swap,"\overline{f}_0^i"] & \arrow[l,swap,"\overline{p}^i"] \overline{P}_1^i \arrow[d,"\overline{f}_1^i"] \\
\overline{Q}_0^i & \arrow[l,swap,"\overline{q}^i"] \overline{Q}_1^i
\end{tikzcd}
\end{equation}
isomorphic to \eqref{ith_presentation_app}. Since the new presentation is obtained by a change of basis, it induces a persistence module 
\begin{equation*} \label{new_ith_restriction}
\begin{tikzcd}
\overline{M}^i_0 \arrow[r,"\overline{A}^i_0"] \arrow[d,"\overline{C}^i_0"] & \cdots \arrow[r,"\overline{A}^i_{i\minus 1}"] & \overline{M}^i_i \arrow[d,"\overline{C}^i_i"] \arrow[r,"\text{id}"] & \overline{M}^i_i \arrow[r,"\text{id}"] \arrow[d,"\overline{C}^i_i"] & \cdots  \\
\overline{N}^i_0 \arrow[r,"\overline{B}^i_0"]  & \cdots \arrow[r,"\overline{B}^i_{i\minus 1}"] & \overline{N}^i_i \arrow[r,"\text{id}"] & \overline{N}^i_i \arrow[r,"\text{id}"] & \cdots
c\end{tikzcd}
\end{equation*}
isomorphic to \eqref{ith_restriction_iso_app}. We can now attach the matrices $A^r_i$ and $B^r_i$ to obtain a persistence module  
\begin{equation*} \label{new_i+1th_restriction}
\begin{tikzcd}
\overline{M}^i_0 \arrow[r,"\overline{A}^i_0"] \arrow[d,"\overline{C}^i_0"] & \cdots \arrow[r,"\overline{A}^i_{i\minus 1}"] & \overline{M}^i_i \arrow[d,"\overline{C}^i_i"] \arrow[r,"A^r_i"] & M_{i+1} \arrow[r,"\text{id}"] \arrow[d,"C_{i+1}"] & M_{i+1} \arrow[r,"\text{id}"] \arrow[d,"C_{i+1}"] & \cdots  \\
\overline{N}^i_0 \arrow[r,"\overline{B}^i_0"]  & \cdots \arrow[r,"\overline{B}^i_{i\minus 1}"] & \overline{N}^i_i \arrow[r,"B^r_i"] & N_{i+1} \arrow[r,"\text{id}"] & N_{i+1} \arrow[r,"\text{id}"] & \cdots
\end{tikzcd}
\end{equation*}
isomorphic to \eqref{ith_restriction_app} for $l=i+1$. Next we row reduce $A^r_i$ and $B^r_i$ and order the rows w.r.t.\ the order of the pivot columns to obtain the matrices $A^{i+1}_{i}$ and $B^{i+1}_{i}$. This corresponds to a change of bases $M_{i+1}\xrightarrow{\cong}M^{i+1}_{i+1}$ and $N_{i+1}\xrightarrow{\cong} N^{i+1}_{i+1}$. Moreover, let $A^{i+1}_{i+1}$, $B^{i+1}_{i+1}$ and $C^{i+1}_{i+1}$ be the matrices $A_{i+1}$, $B_{i+1}$ and $C_{i+1}$ w.r.t.\ these new bases. We now set $M^{i+1}_j\coloneqq \overline{M}^i_j$, $N^{i+1}_j\coloneqq \overline{N}^i_j$ and $C^{i+1}_j=\overline{C}^i_j$ for $j=0,\ldots, i$ and $A^{i+1}_j=\overline{A}^i_j$, $B^{i+1}_j=\overline{B}^i_j$ for $j=0,\ldots, i-1$ to define the persistence module \eqref{ith_restriction_iso_app} for $l=i+1$. The matrices $A^{i+1}_{i}$ and $B^{i+1}_{i}$ are now completely reduced. Let $ca_1,\ldots,ca_s$ and $cb_1,\ldots, cb_t $ be the zero columns of $A^{i+1}_i$ and $B^{i+1}_i$ and $\overline{c}_{j_1},\ldots, \overline{c}_{j_s}$ and $\overline{r}_{k_1},\ldots, \overline{r}_{k_t}$ be the corresponding columns and rows of $\overline{f}^i_0$. Furthermore, let $ra_1,\ldots ra_u$ and  $rb_1,\ldots, rb_v$ be the zero rows of $A^{i+1}_i$ and $B^{i+1}_i$. We now define the presentation \eqref{ith_presentation_app} for $l=i+1$ in the following way:
\begin{equation*}
\begin{aligned}
& P^{i+1}_0\coloneqq \overline{P}^i_0\oplus \text{span}\{ra_1,\ldots,ra_u\} \text{ with } \text{deg}(ra_j)=i+1 \\
& Q^{i+1}_0\coloneqq \overline{Q}^i_0\oplus \text{span}\{rb_1,\ldots,rb_v\} \text{ with } \text{deg}(rb_k)=i+1 \\
& P^{i+1}_1\coloneqq \overline{P}^i_1\oplus \text{span}\{\gamma_1,\ldots,\gamma_s\} \text{ with } \text{deg}(\gamma_j)=i+1 \\
& Q^{i+1}_1\coloneqq \overline{Q}^i_1\oplus \text{span}\{\rho_1,\ldots,\rho_t\}  \text{ with } \text{deg}(\rho_k)=i+1 \\
& p^{i+1}\coloneqq \overline{p}^i\oplus \{\gamma_1\rightarrow \overline{c}_{j_1},\ldots,\gamma_s\rightarrow \overline{c}_{j_s}\} \\
& q^{i+1}\coloneqq \overline{q}^i\oplus \{\rho_1\rightarrow \overline{r}_{k_1},\ldots,\rho_t\rightarrow \overline{r}_{k_t}\} \\
& f^{i+1}_0\coloneqq \overline{f}^i_0\oplus C^{i+1}_{i+1}|_{ra_1,\ldots,ra_u} \\
& f^{i+1}_1\coloneqq \text{ unique map making \eqref{ith_presentation_app} commute}
\end{aligned}
\end{equation*} 
where $C^{i+1}_{i+1}|_{ra_1,\ldots,ra_u}$ denotes the matrix $C^{i+1}_{i+1}$ restricted to the columns corresponding to the zero-rows of $A^{i+1}_i$. By construction this new presentation induces the persistence module \eqref{ith_restriction_iso_app} for $l=i+1$ as defined above which is isomorphic to \eqref{ith_restriction_app} for $l=i+1$. The actual algorithm does not perform the row reduction on $A^r_i$ and $B^r_i$. It just updates the matrix $f^{i+1}_{0}$ by setting $d(\overline{c}_{j_w})=i+1$ and $d(\overline{r}_{k_w})=i+1$ and adding the transformed matrix $C^{i+1}_{i+1}$ restricted to the columns corresponding to non-pivot rows $ra_1,\ldots ra_u$. By the construction discussed in the beginning of this section, this is enough to recover the canonical presentation defined above. Since the non-pivot rows are exactly the ones that get reduced to zero in a row reduction, this does not change the result. The algorithm has to perform the transformations, corresponding to the row reduction of $A^r_i$ and $B^r_i$, of the matrices $A_{i+1}$, $B_{i+1}$ and $C_{i+1}$ which are of size $O(n_{i+2})\times O(n_{i+1})$ and $O(n_{i+1})\times O(n_{i+1})$, respectively. Since we would have to do $O(n_{i+1}^2)$ row additions on $A^r_i$ and $B^r_i$ in the reduction, we have to do $O(n_{i+1}^2)$ column additions on $A_{i+1}$ and $B_{i+1}$ and $O(n_{i+1}^2)$ column and row additions on $C_{i+1}$. Hence, this step takes $O(n_{i+1}^2\cdot n)$ time. Finally the algorithm proceeds with the new presentation matrix $f^{i+1}_0$ and the matrices $A^t_{i+1}\coloneqq A^{i+1}_{i+1}$ and $B^t_{i+1}\coloneqq B^{i+1}_{i+1}$ given w.r.t.\ the bases $M^{i+1}_{i+1}$ and $N^{i+1}_{i+1}$. 

At every step of the algorithm we spend $O(n_i^2\cdot n)+O(n_{i+1}^2\cdot n)$ time. Thus, overall it takes time  
\begin{equation*}
\begin{aligned}
& \phantom{=}\sum_{i=0}^{m\minus 1} O(n_i^2\cdot n)+O(n_{i+1}^2\cdot n)  \\
& =\sum_{i=0}^{m\minus 1}O(n_i^2)\cdot O(n)+\sum_{i=1}^{m}O(n_{i}^2)\cdot O(n) \\
& = O(n)\cdot \sum_{i=0}^{m}O(n_{i}^2)+O(n)\cdot\sum_{i=0}^{m}O(n_{i}^2) \\
& = O(n^3).
\end{aligned}
\end{equation*}
\end{proof}


\section{Proofs in Section \ref{sec_comp_presentations}} \label{app_comp_presentations}

\cancel{

\begin{proposition} \label{appendix_prop_uniqueness_morphism}
Given a commutative diagram with exact rows of the form 
\begin{equation*}
\begin{tikzcd}
0 & \mathbb{I}_{[a,b)} \arrow[l] \arrow[d,"\phi"] & \mathbb{I}_{[a,\infty)} \arrow[l,swap,"\cdot 1"] \arrow[d,"f_0"] & \mathbb{I}_{[b,\infty)} \arrow[l,swap,"\cdot t^{b\minus a}"] \arrow[d,"f_1"] & 0 \arrow[l] \\
0 & \mathbb{I}_{[c,d)} \arrow[l] & \mathbb{I}_{[c,\infty)} \arrow[l,swap,"\cdot 1"] & \mathbb{I}_{[d,\infty)} \arrow[l,swap,"\cdot t^{d\minus c}"] & 0 \arrow[l] 
\end{tikzcd}  
\end{equation*}
such that $a<b$ and $c<d$. There are three cases:
\begin{enumerate}
    \item If $c\leq a<d\leq b$, then $(\phi=0 \iff f_0=0 \iff f_1=0)$.
    \item If $c<d\leq a<b$, then $\phi=0$ and $(f_0=0 \iff f_1=0)$.
    \item Else, $\phi=f_0=f_1=0$.
\end{enumerate}
\end{proposition}

\begin{proof}
As defined in \eqref{hom_pmod}, if $c\leq a<d\leq b$, there is a unique non-zero morphism, up to multiplication with a scalar, between to interval modules. We denote this morphism w.r.t.\ the unit $1$ of the field $\field$ by $(\cdot 1)$ and $(\cdot t^r)$. 
\begin{enumerate}
\item Since $a<b$, $c\leq a$, $d\leq b$ and $c<d$, we have $(\cdot 1)\circ f_0=0$ and $f_0\circ (\cdot t^{b\minus a})=0$ if and only if $f_0=0$ and $(\cdot t^{d\minus c})\circ f_1=0$ if and only if $f_1=0$. Moreover, since $a<d$, we have $\phi\circ (\cdot 1)=0$ if and only if $\phi=0$. 
\item Since $d\leq a$, we get $\phi=0$. Similiar to (1), we obtain $f_0=0 \iff f_1=0$.
\item If we are not in case (1), then $\phi=0$. If $a<c$, then $f_0=0$. If $f_1\neq 0$, then $d\leq b$ and $(\cdot t^{d\minus c})\circ f_1\neq 0=f_0\circ (\cdot t^{b\minus a})$. A contradiction. If $b<d$, then $f_1=0$. If $f_0\neq 0$, then $c\leq a$ and $f_0\circ (\cdot t^{b\minus a})\neq 0=(\cdot t^{d\minus c})\circ f_1$. A contradiction. If $c\leq a$ and $d\leq b$ we are in case (1) or (2).
\end{enumerate}
\end{proof}
\cancel{
\begin{proposition} \label{appendix_prop_uniqueness_morphism_2}
Given a commutative diagram of the form 
\begin{equation*}
\begin{tikzcd}
0 & \mathbb{I}_{[c,d)} \arrow[l] & \mathbb{I}_{[c,\infty)} \arrow[l,swap,"\cdot 1"] &[5pt] \mathbb{I}_{[d,\infty)} \arrow[l,swap,"\cdot t^{d-c}"] & 0 \arrow[l] \\
0 & N \arrow[l] \arrow[u,"\pi"] & Q_0 \arrow[l,swap,"\nu"] \arrow[u,"\pi_0"] & Q_1 \arrow[l,swap,"q"] \arrow[u,"\pi_1"] & 0 \arrow[l] \\
0 & M \arrow[l] \arrow[u,"\phi"] & P_0 \arrow[l,swap,"\mu"] \arrow[u,"f_0"] & P_1 \arrow[l,swap,"p"] \arrow[u,"f_1"] & 0 \arrow[l] \\
0 & \mathbb{I}_{[a,b)} \arrow[l] \arrow[u,"\iota"] & \mathbb{I}_{[a,\infty)} \arrow[l,swap,"\cdot 1"] \arrow[u,"\iota_0"] & \mathbb{I}_{[b,\infty)} \arrow[l,swap,"\cdot t^{b-a}"] \arrow[u,"\iota_1"] & 0 \arrow[l]
\end{tikzcd}  
\end{equation*}
If not $c<d\leq a<b$, then $\pi_0\circ f_0\circ \iota_0$ and $\pi_1\circ f_1\circ \iota_1$ are uniquely determined by $\phi$.
\end{proposition}
}

\begin{proof}[Proof of Proposition~\ref{lemma1}]
By Proposition \ref{appendix_prop_uniqueness_morphism}, in case (1), we have $\pi\circ\phi\circ\iota=0\iff \pi_0\circ f_0\circ\iota_0=0\iff \pi_1\circ f_1\circ\iota_1=0$. Moreover, in case (3), we get $\pi\circ\phi\circ\iota=\pi_0\circ f_0\circ\iota_0=\pi_1\circ f_1\circ\iota_1=0$. 
\end{proof}

\cancel{
\begin{theorem} \label{modification_thm}
Let $L\xlongrightarrow{\phi} M \xlongrightarrow{\psi} N$ be a complex of graded $\field[t]$-modules, then there exists a complex of presentations presenting $L\xlongrightarrow{\phi} M \xlongrightarrow{\psi} N$. 
\end{theorem}
}

}

\begin{proof}[Proof of Theorem~\ref{modification_thm}]
Let  
\begin{equation} \label{complex_cannonical_presentation}
\begin{tikzcd}
0 & \mathbb{I}_{[a,b)} \arrow[l] \arrow[d,"\iota"] & \mathbb{I}_{[a,\infty)} \arrow[l,swap,"\cdot 1"] \arrow[d,"\iota_0"] & \mathbb{I}_{[b,\infty)} \arrow[l,swap,"\cdot t^{b\minus a}"] \arrow[d,"\iota_1"] & 0 \arrow[l] \\
0 & L \arrow[l] \arrow[d,"\phi"] & P_0 \arrow[l,swap,"\lambda"] \arrow[d,"f_0"] & P_1 \arrow[l,swap,"p"] \arrow[d,"f_1"] & 0 \arrow[l] \\
0 & M \arrow[l] \arrow[d,"\psi"] & Q_0 \arrow[l,swap,"\mu"] \arrow[d,"g_0"] & Q_1 \arrow[l,swap,"q"] \arrow[d,"g_1"] & 0 \arrow[l] \\
0 & N \arrow[l] \arrow[d,"\pi"] & R_0 \arrow[l,swap,"\nu"] \arrow[d,"\pi_0"] & R_1 \arrow[l,swap,"r"] \arrow[d,"\pi_1"] & 0 \arrow[l] \\
0 & \mathbb{I}_{[c,d)} \arrow[l] & \mathbb{I}_{[c,\infty)} \arrow[l,swap,"\cdot 1"] &[5pt] \mathbb{I}_{[d,\infty)} \arrow[l,swap,"\cdot t^{d\minus c}"] & 0 \arrow[l] 
\end{tikzcd}
\end{equation}
be a canonical presentation of $L\xlongrightarrow{\phi} M \xlongrightarrow{\psi} N$. Suppose that $g_0\circ f_0\neq 0$. Then, by Proposition~\ref{lemma1}, 
there exist $\mathbb{I}_{[a,b)}\leq L$ and $\mathbb{I}_{[c,d)}\leq N$ such that $\pi_0\circ g_0\circ f_0\circ \iota_0\neq 0$ and $c<d\leq a<b$. W.l.o.g.\ assume $\pi_0\circ g_0\circ f_0\circ \iota_0=(\cdot \eta t^{a-c})$. Let $c<d\leq e\leq a<b$. Define $Q'_0\coloneqq Q_0\oplus \mathbb{I}_{[e,\infty)}$, $Q'_1\coloneqq Q_1\oplus \mathbb{I}_{[e,\infty)}$ and 
\begin{equation} \label{modified_representation}
\begin{aligned}
& q'\colon Q'_1\rightarrow Q'_0 \hspace{2pt},\hspace{5pt} q'\coloneqq \begin{pmatrix} q & 0 \\ 0 & \text{id} \end{pmatrix} \\
& \mu'\colon Q'_0\rightarrow M \hspace{2pt},\hspace{5pt} \mu'\coloneqq \begin{pmatrix} \mu & 0  \end{pmatrix} \\
& f'_0\colon P_0\rightarrow Q'_0 \hspace{2pt},\hspace{5pt} f'_0\coloneqq \begin{pmatrix} f_0 \\ (\mathbb{I}_{[a,\infty)}\xrightarrow{\cdot t^{a\minus e}} \mathbb{I}_{[e,\infty)}) \end{pmatrix} \\
& f'_1\colon P_1\rightarrow Q'_1 \hspace{2pt},\hspace{5pt} f'_1\coloneqq \begin{pmatrix} f_1 \\ (\mathbb{I}_{[b,\infty)}\xrightarrow{\cdot t^{b\minus e}} \mathbb{I}_{[e,\infty)}) \end{pmatrix} \\
& g'_0\colon Q'_0\rightarrow R_0 \hspace{2pt},\hspace{5pt} g'_0\coloneqq \begin{pmatrix} g_0 & (\mathbb{I}_{[e,\infty)}\xrightarrow{\cdot (\minus\eta)t^{e\minus c}} \mathbb{I}_{[c,\infty)}) \end{pmatrix} \\
& g'_1\colon Q'_1\rightarrow R_1 \hspace{2pt},\hspace{5pt} g'_1\coloneqq \begin{pmatrix} g_1 & (\mathbb{I}_{[e,\infty)}\xrightarrow{\cdot (\minus\eta)t^{e\minus d}} \mathbb{I}_{[d,\infty)}) \end{pmatrix} 
\end{aligned}
\end{equation}
After replacing $Q_0$ and $Q_1$ by $Q'_0$ and $Q'_1$ we obtain 
\begin{equation} \label{pres_alt}
\begin{tikzcd}
0 & \mathbb{I}_{[a,b)} \arrow[l] \arrow[d,"\iota"] & \mathbb{I}_{[a,\infty)} \arrow[l,swap,"\cdot 1"] \arrow[d,"\iota_0"] & \mathbb{I}_{[b,\infty)} \arrow[l,swap,"\cdot t^{b\minus a}"] \arrow[d,"\iota_1"] & 0 \arrow[l] \\
0 & L \arrow[l] \arrow[d,"\phi"] & P_0 \arrow[l,swap,"\lambda"] \arrow[d,"f'_0"] & P_1 \arrow[l,swap,"p"] \arrow[d,"f'_1"] & 0 \arrow[l] \\
0 & M \arrow[l] \arrow[d,"\psi"] & Q'_0 \arrow[l,swap,"\mu'"] \arrow[d,"g'_0"] & Q'_1 \arrow[l,swap,"q'"] \arrow[d,"g'_1"] & 0 \arrow[l] \\
0 & N \arrow[l] \arrow[d,"\pi"] & R_0 \arrow[l,swap,"\nu"] \arrow[d,"\pi_0"] & R_1 \arrow[l,swap,"r"] \arrow[d,"\pi_1"] & 0 \arrow[l] \\
0 & \mathbb{I}_{[c,d)} \arrow[l] & \mathbb{I}_{[c,\infty)} \arrow[l,swap,"\cdot 1"] &[5pt] \mathbb{I}_{[d,\infty)} \arrow[l,swap,"\cdot t^{d\minus c}"] & 0 \arrow[l] 
\end{tikzcd}
\end{equation}
We now show that \eqref{pres_alt} is still a reduced presentation of $L\xlongrightarrow{\phi} M \xlongrightarrow{\psi} N$:
\begin{enumerate}
\item $\text{im }\mu'=\text{im }\mu=M$.
\item $\text{ker }\mu'=\text{ker }\mu\oplus \mathbb{I}_{[e,\infty)}=\text{im }q\oplus \mathbb{I}_{[e,\infty)}=\text{im }q'$.
\item $\text{ker }q'=\text{ker }q=0$.
\item $\mu'\circ f'_0=\begin{pmatrix} \mu & 0  \end{pmatrix}\circ \begin{pmatrix} f_0 \\ \mathbb{I}_{[a,\infty)}\xrightarrow{\cdot t^{a\minus e}} \mathbb{I}_{[e,\infty)} \end{pmatrix}=\mu\circ f_0=\phi\circ \lambda$.
\item $q'\circ f'_1=\begin{pmatrix} q & 0 \\ 0 & \text{id} \end{pmatrix}\circ \begin{pmatrix} f_1 \\ \mathbb{I}_{[b,\infty)}\xrightarrow{\cdot t^{b\minus e}} \mathbb{I}_{[e,\infty)} \end{pmatrix}=\begin{pmatrix} q\circ f_1 \\ (\mathbb{I}_{[b,\infty)}\xrightarrow{\cdot t^{b\minus e}} \mathbb{I}_{[e,\infty)}) \end{pmatrix} \\=\begin{pmatrix} f_0\circ p \\ (\mathbb{I}_{[a,\infty)}\xrightarrow{\cdot t^{a\minus e}} \mathbb{I}_{[e,\infty)})\circ p \end{pmatrix}=\begin{pmatrix} f_0 \\ (\mathbb{I}_{[a,\infty)}\xrightarrow{\cdot t^{a\minus e}} \mathbb{I}_{[e,\infty)}) \end{pmatrix}\circ p=f'_0\circ p$.
\item $\nu\circ g'_0=\nu\circ \begin{pmatrix} g_0 & (\mathbb{I}_{[e,\infty)}\xrightarrow{\cdot (\minus\eta)t^{e\minus c}} \mathbb{I}_{[c,\infty)}) \end{pmatrix}=\begin{pmatrix} \nu\circ g_0 & \nu\circ (\mathbb{I}_{[e,\infty)}\xrightarrow{\cdot (\minus\eta)t^{e\minus c}} \mathbb{I}_{[c,\infty)}) \end{pmatrix} \\ = \begin{pmatrix} \psi\circ \mu & 0 \end{pmatrix}=\psi\circ \mu'$. \\ Note that, since $d\leq e$, $\text{im }(\mathbb{I}_{[e,\infty)}\xrightarrow{\cdot (\minus\eta)t^{e\minus c}} \mathbb{I}_{[c,\infty)})\subseteq \mathbb{I}_{[d,\infty)}\subseteq \text{ker }\nu$.
\item $r\circ g'_1=r\circ \begin{pmatrix} g_1 & (\mathbb{I}_{[e,\infty)}\xrightarrow{\cdot (\minus\eta)t^{e\minus d}} \mathbb{I}_{[d,\infty)}) \end{pmatrix}=\begin{pmatrix} r\circ g_1 & r\circ (\mathbb{I}_{[e,\infty)}\xrightarrow{\cdot (\minus\eta)t^{e\minus d}} \mathbb{I}_{[d,\infty)}) \end{pmatrix} \\= \begin{pmatrix} g_0\circ q & (\mathbb{I}_{[e,\infty)}\xrightarrow{\cdot (\minus\eta)t^{e\minus c}} \mathbb{I}_{[c,\infty)}) \end{pmatrix}=\begin{pmatrix} g_0 & (\mathbb{I}_{[e,\infty)}\xrightarrow{\cdot (\minus\eta)t^{e\minus c}} \mathbb{I}_{[c,\infty)}) \end{pmatrix}\circ\begin{pmatrix} q & 0 \\ 0 & \text{id} \end{pmatrix}=g'_0\circ q'$.
\end{enumerate}
Moreover, we obtain 
\begin{equation}
\begin{aligned}
& \pi_0\circ g'_0\circ f'_0\circ \iota_0=\pi_0\circ \begin{pmatrix} g_0 & (\mathbb{I}_{[e,\infty)}\xrightarrow{\cdot (\minus\eta)t^{e\minus c}} \mathbb{I}_{[c,\infty)}) \end{pmatrix}\circ \begin{pmatrix} f_0 \\ (\mathbb{I}_{[a,\infty)}\xrightarrow{\cdot t^{a\minus e}} \mathbb{I}_{[e,\infty)}) \end{pmatrix}\circ \iota_0 \\ & =\pi_0\circ g_0\circ f_0\circ \iota_0 - \pi_0\circ (\mathbb{I}_{[a,\infty)}\xrightarrow{\cdot \eta t^{a\minus c}} \mathbb{I}_{[c,\infty)})\circ \iota_0=0
\end{aligned}
\end{equation}
since $\pi_0\circ g_0\circ f_0\circ \iota_0=(\mathbb{I}_{[a,\infty)}\xrightarrow{\cdot \eta t^{a\minus c}} \mathbb{I}_{[c,\infty)})$.
If $\mathbb{I}_{[a',\infty)}\neq \mathbb{I}_{[a,\infty)}$ or $\mathbb{I}_{[c',\infty)}\neq \mathbb{I}_{[c,\infty)}$, then $\pi'_0\circ g'_0\circ f'_0\circ \iota'_0=\pi'_0\circ g_0\circ f_0\circ \iota'_0$ since $\pi'_0\circ (\mathbb{I}_{[a,\infty)}\xrightarrow{\cdot \eta t^{a\minus c}} \mathbb{I}_{[c,\infty)}) \circ \iota'_0=0$. By commutativity, $\pi_0\circ g'_0\circ f'_0\circ \iota_0=0$ implies $\pi_1\circ g'_1\circ f'_1\circ \iota_1=0$. Therefore, we have reduced the number of pairs $\mathbb{I}_{[a,b)}$, $\mathbb{I}_{[c,d)}$ such that $\pi_0\circ g_0\circ f_0\circ \iota_0\neq 0$ by one. Since there is only a finite number of such pairs, we can repeat this process until $g'_0\circ f'_0=0$ and $g'_1\circ f'_1=0$.
\end{proof}

\cancel{
\begin{corollary}
Given a complex of graded $\field[t]$-modules $L\xlongrightarrow{\phi} M \xlongrightarrow{\psi} N$ and a canonical presentation as in \eqref{complex_cannonical_presentation} of size $n\coloneqq\text{max}\{\text{rank}(P_0),\text{rank}(Q_0),\text{rank}(S_0)\}$. Then we can modify \eqref{complex_cannonical_presentation} such that the obtained presentation is a complex of presentations of size $O(n)$.
\end{corollary}
}

\begin{proof}[Proof of Corollary~\ref{cor:preserve-rank}]
In the proof of Theorem \ref{modification_thm}, we show that given $\mathbb{I}_{[a,b)}\leq L$ and $\mathbb{I}_{[c,d)}\leq N$ such that $c<d\leq a<b$ and $\pi_0\circ g_0\circ f_0\circ \iota_0\neq 0$, we can modify the presentations as in \eqref{modified_representation} such that $\pi_0\circ g'_0\circ f'_0\circ \iota_0=0$ and for all $\mathbb{I}_{[a',b')}\neq\mathbb{I}_{[a,b)} \leq L$ and $\mathbb{I}_{[c',d')}\neq \mathbb{I}_{[c,d)}\leq N$ we have $\pi'_0\circ g'_0\circ f'_0\circ \iota'_0=\pi'_0\circ g_0\circ f_0\circ \iota'_0$. Suppose $\mathbb{I}_{[c',d')}\leq N$ is another bar such that $c'<d'\leq a<b$ and $\pi'_0\circ g'_0\circ f'_0\circ \iota_0=\xi t^{a\minus c'}\neq0$. We could repeat the same process again to resolve this issue but we don't have to add another generator to $Q'_0$. We can assume that $e=a$ and modify the $g$-maps in the following way:
\begin{equation*}
\begin{aligned}
& g''_0\colon Q'_0\rightarrow R_0 \hspace{2pt},\hspace{5pt} g''_0\coloneqq \begin{pmatrix} g_0 & (\mathbb{I}_{[e,\infty)}\xrightarrow{\cdot (\minus\eta)t^{e\minus c}} \mathbb{I}_{[c,\infty)})+(\mathbb{I}_{[e,\infty)}\xrightarrow{\cdot (\minus\xi)t^{e\minus c'}} \mathbb{I}_{[c',\infty)}) \end{pmatrix} \\
& g''_1\colon Q'_1\rightarrow R_1 \hspace{2pt},\hspace{5pt} g''_1\coloneqq \begin{pmatrix} g_1 & (\mathbb{I}_{[e,\infty)}\xrightarrow{\cdot (\minus\eta)t^{e\minus d}} \mathbb{I}_{[d,\infty)})+(\mathbb{I}_{[e,\infty)}\xrightarrow{\cdot (\minus\xi)t^{e\minus d'}} \mathbb{I}_{[d',\infty)}) \end{pmatrix} 
\end{aligned}
\end{equation*}
After this modification we have 
\begin{equation*}
\begin{aligned}
& \pi_0\circ g''_0\circ f'_0\circ \iota_0=\pi_0\circ \begin{pmatrix} g_0 & (\mathbb{I}_{[e,\infty)}\xrightarrow{\cdot (\minus\eta)t^{e\minus c}} \mathbb{I}_{[c,\infty)})+(\mathbb{I}_{[e,\infty)}\xrightarrow{\cdot (\minus\xi)t^{e\minus c'}} \mathbb{I}_{[c',\infty)}) \end{pmatrix} \\ & \circ \begin{pmatrix} f_0 & (\mathbb{I}_{[a,\infty)}\xrightarrow{\cdot t^{a\minus e}} \mathbb{I}_{[e,\infty)}) \end{pmatrix}\circ \iota_0=\pi_0\circ g_0\circ f_0\circ \iota_0 - \pi_0\circ (\mathbb{I}_{[a,\infty)}\xrightarrow{\cdot \eta t^{a\minus c}} \mathbb{I}_{[c,\infty)})\circ \iota_0 \\ & -\pi_0\circ (\mathbb{I}_{[a,\infty)}\xrightarrow{\cdot \xi t^{a\minus c'}} \mathbb{I}_{[c',\infty)})\circ \iota_0=0
\end{aligned}
\end{equation*}
since $\pi_0\circ g_0\circ f_0\circ \iota_0=(\mathbb{I}_{[a,\infty)}\xrightarrow{\cdot \eta t^{a-c}} \mathbb{I}_{[c,\infty)})$ and $\pi_0\circ (\mathbb{I}_{[a,\infty)}\xrightarrow{\cdot \xi t^{a-c'}} \mathbb{I}_{[c',\infty)}) \circ \iota_0=0$.
Moreover, by the same argument, we get $\pi'_0\circ g''_0\circ f'_0\circ \iota_0=0$. If $\mathbb{I}_{[a'',b'')}\neq\mathbb{I}_{[a,b)}\leq L$ or $\mathbb{I}_{[c'',d'')}\neq\mathbb{I}_{[c,d)},\mathbb{I}_{[c',d')}\leq L$, then $\pi''_0\circ g''_0\circ f'_0\circ \iota''_0=\pi''_0\circ g_0\circ f_0\circ \iota''_0$ since $\pi''_0\circ (\mathbb{I}_{[a,\infty)}\xrightarrow{\cdot \eta t^{a-c}} \mathbb{I}_{[c,\infty)}) \circ \iota''_0=\pi''_0\circ (\mathbb{I}_{[a,\infty)}\xrightarrow{\cdot \xi t^{a-c'}} \mathbb{I}_{[c',\infty)}) \circ \iota''_0=0$. 

Using this construction method, given a generator $\mathbb{I}_{[a,b)}\leq L$, we can modify the presentation such that $\pi_0\circ g_0\circ f_0\circ \iota_0=0$ for all $\mathbb{I}_{[c,d)}\leq N$ by adding a single generator to $Q_0$ and $Q_1$. Since, $\text{rank}(P_0)\leq n$, to fix the whole presentation, we have to add at most $n$ generators to $Q_0$. This implies after the modification is completed, we have $\text{rank}(Q'_0)\leq 2n$. 

\end{proof}


\section{Proofs in Section \ref{sec_application_tower}} \label{app_application_tower}

\begin{proposition} \label{prop_modified_presentation_algo}
Consider a simplicial tower 
\begin{equation} \label{appendix_application_tower}
\begin{tikzcd}
K_0 \arrow[r,"f_0"] & K_1 \arrow[r,"f_1"] & \cdots \arrow[r,"f_{m\minus 1}"] &[3pt] K_m
\end{tikzcd}
\end{equation} 
where each $f_i$ is an elementary inclusion or collapse and $n$ is the number of elementary inclusions. Then, using a modified version of the algorithm in Section \ref{sec_presentation_algorithm}, we can compute a presentation of 
\begin{equation} \label{appendix_chain_modules}
\begin{tikzcd}[column sep=large]
C_k(\Vec{K}): \arrow[d,"\partial_k"] &[-35pt] C_k(K_0) \arrow[r,"C_k(f_0)"] \arrow[d,"\partial_k^0"] & C_k(K_1) \arrow[r,"C_k(f_1)"] \arrow[d,"\partial_k^1"] & \cdots \arrow[r,"C_k(f_{m\minus 1})"] &[10pt] C_k(K_m) \arrow[d,"\partial_k^{m}"] \\
C_{k\minus 1}(\Vec{K}): & C_{k\minus 1}(K_0) \arrow[r,"C_{k\minus 1}(f_0)"] & C_{k\minus 1}(K_1) \arrow[r,"C_{k\minus 1}(f_1)"] & \cdots \arrow[r,"C_{k\minus 1}(f_{m\minus 1})"] & C_{k\minus 1}(K_m)
\end{tikzcd}
\end{equation}
in $O(n^2)$ time.
\end{proposition}

\begin{proof}
We already proved that the algorithm in Section \ref{sec_presentation_algorithm} is correct. This algorithm builds a canonical presentation 
\begin{equation*}
\begin{tikzcd}
0 & \arrow[l] C_k(\Vec{K}) \arrow[d,"\partial_k"] & P_0 \arrow[d,"f_0"] \arrow[l] & P_1 \arrow[d,"f_1"] \arrow[l,swap,"p"] \\
0 & \arrow[l] C_{k\minus 1}(\Vec{K}) & Q_0 \arrow[l] & Q_1 \arrow[l,swap,"q"]
\end{tikzcd}
\end{equation*}
by processing the morphism of persistence modules \eqref{appendix_chain_modules} from left to right. In this process the algorithm reduces the matrices $C_k(f_i)$ while performing the corresponding operations on $f_0$. The simple structure of the matrices $C_k(f_i)$ allows us to simplify the algorithm. We distinguish two cases: 

First we assume that $f_i$ is an elementary inclusion. This means that at step $i+1$ a single simplex $\sigma$ is inserted into the tower. If $\sigma$ is neither a $k$- nor a $(k\minus 1)$-simplex then $C_k(f_i)$ and $C_{k\minus 1}(f_i)$ are just identity matrices. In this case our algorithm does nothing at this step.  If $\sigma$ is a $k$-simplex, then $C_{k\minus 1}(f_i)$ is an identity matrix and $C_k(f_i)$ is an identity matrix with one additional zero-row. Both matrices are already column- and row-reduced so our algorithm will not perform any operations here. The zero-row of the matrix $C_k(f_i)$ implies the birth of a new bar corresponding to the inserted simplex. The algorithm takes account of that by adding a new column $\partial_k^{i+1}(\sigma)$ to $f_0$ with $b(\text{new col})=i+1$. Similarly, if $\sigma$ is a $(k\minus 1)$-simplex, the algorithm adds a new zero-row to $f_0$ with $b(\text{new row})=i+1$. Note that we don't have to build or reduce any of the $C_k(f_i)$ matrices to execute this step we just have to add a new row or column to $f_0$ or do nothing.  

Now we assume that $f_i$ is an elementary collapse. This means that $f_i$ is surjective, $f_i(b)=a$ for some $b\neq a\in K_i^0$ and $f_i$ is injective on $K_i^0\setminus b$. If $\tau$ is a $k$-simplex in $K_{i+1}$, there are two cases: 1) $\tau$ is the image of exactly one $k$-simplex $\sigma$ in $K_i$. 2) $\tau$ is the image of exactly two $k$-simplices $\tau$ and $\sigma$ in $K_i$, i.e.\ $\sigma$ is merged into $\tau$ and $\tau$ is already in $K_i$. This can be seen in the following way: By surjectivity there exists a simplex $\sigma\in K_i^k$ such that $f(\sigma)=\tau$. If $\sigma=(v_0,\ldots,v_k)\neq (w_0,\ldots, w_k)=\kappa$ and $f(\sigma)=f(\kappa)=\tau$, there exists a $0\leq j\leq k$ such that $v_i=w_i$ for all $i\neq j$ and $v_j\neq w_j$ and $f(v_i)=f(w_i)$ for all $0\leq i \leq k$. This is only possible if either $v_j=a$ and $w_j=b$ or $v_j=b$ and $w_j=a$. Hence, there are at most two $k$-simplices in $K_i$ that map to $\tau$ and if there are two, then one of them is $\tau$ itself. This implies that every row of the matrix $C_k(f_i)$ has either one or two non-zero entries. Since every $k$-simplex can map to at most one other $k$-simplex, each column of $C_k(f_i)$ has at most one non-zero entry. If $\sigma$ is a $k$-simplex in $K_i$ there are two cases: 1) $\sigma$ is mapped to a $k$-simplex $\tau$, i.e.\ $C_k(f_i)_\sigma$ has a single non-zero entry in row $\tau$. 2) $\sigma$ is mapped to a $(k\minus 1)$-simplex, i.e.\ $C_k(f_i)_\sigma$ is a zero-column. To reduce $C_k(f_i)$ the algorithm has to add exactly one column to each column corresponding to a simplex that is merged into another simplex. Note that we have to add the column of the simplex which is born first to the one which is born later, i.e., if $\sigma$ and $\tau$ are merged and $b(\sigma)<b(\tau)$ we have to add the column of $\sigma$ to the column of $\tau$. After that the matrix is column- and row-reduced. For each zero-column, i.e.\ for each simplex that is merged into another simplex which is born before or mapped to a $(k\minus 1)$-simplex, the algorithm sets $d(col)=i+1$ for the corresponding column in $f_0$. As in the case of elementary inclusions, we don't have to construct or reduce the matrices $C_k(f_i)$ to execute the algorithm. We just have to add the column corresponding to $\sigma$ in $f_0$ to the column corresponding to $\tau$ in $f_0$ if $\tau$ and $\sigma$ are merged and $b(\sigma)<b(\tau)$ for all simplices that are merged and set $d(col)=i+1$ for each column corresponding to a simplex that is merged into a simplex born before or mapped to a $(k\minus 1)$-simplex. We can proceed in a similar way for the $(k\minus 1)$-simplices and the rows of $f_0$. 

Since there are $n$ elementary inclusions we add at most $n$ columns and rows to $f_0$. This implies that $f_0$ is of size $O(n)\times O(n)$. For each simplex that is merged into another simplex born before we have to add exactly one column or row in $f_0$. Such an addition takes $O(n)$ time. Since there are at most $n$ simplices and each simplex can be merged only once into a simplex born before, the overall cost of these operations is $O(n^2)$. 
\end{proof}
\cancel{
\begin{theorem}
Given a simplicial tower \eqref{appendix_application_tower} where each $f_i$ is an elementary inclusion or collapse. Let $n$ be the number of elementary inclusions, then, using the modified version of the algorithm of Section \ref{sec_presentation_algorithm} discussed in Proposition \ref{prop_modified_presentation_algo} and the algorithm of Section \ref{sec_cohomology_presentations}, we can compute the persistent homology barcode of the tower in $O(n^3)$ time.
\end{theorem}
}

\begin{proof}[Proof of Theorem~\ref{thm:compute-tower}]
As shown in Proposition \ref{prop_modified_presentation_algo}, we can compute a presentation of the complex of persistence modules
\begin{equation} \label{appendix_simp_chain_comp_modules}
\begin{tikzcd}
C_{k+1}(\Vec{K}) \arrow[r,"\partial_{k+1}"] & C_k(\Vec{K}) \arrow[r,"\partial_{k}"] & C_{k\minus 1}(\Vec{K})
\end{tikzcd}
\end{equation}
in $O(n^2)$ time. The proposition also shows that the matrices in this presentation are of size $O(n)\times O(n)$. The methods developed in Sections  \ref{sec_cohomology_presentations} and \ref{sec_comp_presentations} now allow us to compute the homology of \eqref{appendix_simp_chain_comp_modules} in $O(n^3)$ time. 
\end{proof}


\section{Proofs in Section \ref{sec_applications_cosheaf_tower}} \label{sec_appendix_cosheaf_tower}

\begin{proposition} \label{prop_modified_presentation_algo_cosheaf}
Given a simplicial tower $\Vec{K}$ and a cosheaf $F$ on $K_m$
\begin{equation} \label{appendix_application_cosheaf_tower}
\begin{tikzcd}
F^0 & F^1 \arrow[l,swap,mapsto,"f_{0}^*"] & \cdots \arrow[l,swap,mapsto,"f_{1}^*"] & F^m \arrow[l,swap,mapsto,"f_{m\minus 1}^*"] \\[-13pt]
K_0 \arrow[r,"f_0"] & K_1 \arrow[r,"f_1"] & \cdots \arrow[r,"f_{m\minus 1}"] & K_m
\end{tikzcd}
\end{equation}
where each $f_i$ is an elementary inclusion or collapse. Let $n=\text{\# elementary inclusions}\times \underset{\sigma\in K_m}{\text{max}}\text{dim}\big(F(\sigma)\big)$, then, using a modified version of the algorithm in Section \ref{sec_presentation_algorithm}, we can compute a presentation of 
\begin{equation*} \label{appendix_chain_modules_cosheaf}
\begin{tikzcd}[column sep=large]
C_k(\Vec{K},F): \arrow[d,"\partial_k"] &[-40pt] C_k(K_0,F^0) \arrow[r,"C_k(f_0)"] \arrow[d,"\partial_k^0"] &[-3pt] C_k(K_1,F^1) \arrow[r,"C_k(f_1)"] \arrow[d,"\partial_k^1"] &[-3pt] \cdots \arrow[r,"C_k(f_{m\minus 1})"] &[8pt] C_k(K_m,F^m) \arrow[d,"\partial_k^{m}"] \\
C_{k\minus 1}(\Vec{K},F): & C_{k\minus 1}(K_0,F^0) \arrow[r,"C_{k\minus 1}(f_0)"] & C_{k\minus 1}(K_1,F^1) \arrow[r,"C_{k\minus 1}(f_1)"] & \cdots \arrow[r,"C_{k\minus 1}(f_{m\minus 1})"] & C_{k\minus 1}(K_m,F^m)
\end{tikzcd}
\end{equation*}
in $O(n^2)$ time.
\end{proposition}

\begin{proof}
The modification of the algorithm of Section \ref{sec_presentation_algorithm} and the proof is analogous to the proof of Proposition \ref{prop_modified_presentation_algo}. As defined in \eqref{induced_maps_cochain}, the matrices $C_k(f_i)$ viewed as block matrices have the same structure as the ones in the tower case. The only difference is that a single $1$ entry is replaced by an identity block. Since, a single simplex can contribute at most $\underset{\sigma\in K_m}{\text{max}}\text{dim}\big(F(\sigma)\big)$ columns or rows we obtain the claimed complexity.   
\end{proof}

\cancel{
\begin{theorem}
Given an instance $(\Vec{K},F)$ as in \eqref{appendix_application_cosheaf_tower} where each $f_i$ is an elementary inclusion or collapse and $n=\text{\# elementary inclusions}\cdot \underset{\sigma\in K_m}{\text{max}}\text{dim}\big(F(\sigma)\big)$, we can compute the persistent cosheaf homology barcode of \eqref{appendix_application_cosheaf_tower} in $O(n^3)$ time. 
\end{theorem}
}

\begin{proof}[Proof of Theorem~\ref{thm:compute-sheaf}]
As shown in Proposition \ref{prop_modified_presentation_algo_cosheaf}, we can compute a presentation of the complex of persistence modules
\begin{equation} \label{appendix_simp_chain_comp_modules_cosheaf}
\begin{tikzcd}
C_{k+1}(\Vec{K},F) \arrow[r,"\partial_{k+1}"] & C_k(\Vec{K},F) \arrow[r,"\partial_{k}"] & C_{k\minus 1}(\Vec{K},F)
\end{tikzcd}
\end{equation}
in $O(n^2)$ time. The proposition also shows that the matrices in this presentation are of size $O(n)\times O(n)$. The methods developed in Sections  \ref{sec_cohomology_presentations} and \ref{sec_comp_presentations} now allow us to compute the homology of \eqref{appendix_simp_chain_comp_modules_cosheaf} in $O(n^3)$ time. 
\end{proof}

\cancel{
\section{Proofs Section \ref{sec_pers_shv_simp}} \label{sec_appendix_pers_shv_simp}

\begin{theorem}
Given a persistent sheaf of finite type 
\begin{equation}
\begin{tikzcd}
F_0 \arrow[r,"\phi_0"] & F_1 \arrow[r,"\phi_1"] & \cdots \arrow[r,"\phi_{m\minus 1}"] & F_m \arrow[r,"\cong"] & \cdots
\end{tikzcd}
\end{equation}
\tamal{over a simplicial complex $X$, of size $n$ as defined in Section \ref{sec_pers_shv_simp}, the procedure} described in Section \ref{sec_pers_shv_simp}, we can compute the persistent sheaf cohomology barcode in $O(n^3)$ time.   
\end{theorem}

\begin{proof}
This follows directly from the results and algorithms of Section \ref{sec_presentation_algorithm}, \ref{sec_comp_presentations} and \ref{sec_cohomology_presentations}.
\end{proof}
}


\section{Proofs in Section \ref{sec_poset}} \label{app_poset}

\subsection{Computation by order complex} \label{app_order_complex}

\begin{proposition} \label{fullyfaithful}
Let $X$ be a finite poset and $f\colon \mathcal{K}(X)\rightarrow X$ the projection from the order complex, then $f^*$ is fully faithful. 
\end{proposition}

\begin{proof}
Let $F$ and $G$ be sheaves on $X$ and $H_{F,G}\colon \text{Hom}(F,G)\rightarrow \text{Hom}(f^*F,f^*G)$ the map defined by $H_{F,G}(\phi)\coloneqq f^*\phi$. Let $\phi,\psi\in \text{Hom}(F,G)$ such that $f^*\phi=f^*\psi$. Let $x\in X$, then $(x)\in \mathcal{K}(X)$ and $f\big((x)\big)=x$. Hence, $\phi(x)=\phi\big(f\big((x)\big)\big)=f^*\phi\big((x)\big)=f^*\psi\big((x)\big)=\psi\big(f\big((x)\big)\big)=\psi(x)$. This implies $H_{F,G}$ is injective. Let $\phi\in \text{Hom}(f^*F,f^*G)$. For every $x\in X$ define $\psi(x)\coloneqq \phi\big((x)\big)$. Note that, since $\phi$ is a sheaf morphism and $f^*F\big((x_0<\cdots<x_n)\big)=F(x_n)$, we obtain the following commutative diagram for every $(x_0<\cdots <x_n)\in\mathcal{K}(X)$
\begin{equation*}
\begin{tikzcd}[column sep=huge,row sep=large]
f^*F\big((x_n)\big) \arrow[r,"\phi\big((x_n)\big)"] \arrow[d,swap,"\text{id}"] &[15pt] f^*G\big((x_n)\big) \arrow[d,"\text{id}"] \\
f^*F\big((x_0<\cdots<x_n)\big) \arrow[r,"\phi\big((x_0<\cdots<x_n)\big)"] & f^*G\big((x_0<\cdots<x_n)\big)
\end{tikzcd}
\end{equation*}
Hence, $\phi\big((x_0<\cdots<x_n)\big)=\phi\big((x_n)\big)$. If $x<y\in X$, we obtain $(x)\leq (x<y)\in\mathcal{K}(X)$ and the following commutative diagram
\begin{equation*}
\begin{tikzcd}[column sep=large,row sep=large]
F(x) \arrow[r,"\psi(x)"] \arrow[d,swap,"\text{id}"] & G(x) \arrow[d,"\text{id}"] \\
f^*F\big((x)\big) \arrow[r,"\phi\big((x)\big)"] \arrow[d,swap,"f^*F\big((x)\leq (x<y)\big)"] & f^*G\big((x)\big) \arrow[d,"f^*G\big((x)\leq (x<y)\big)"] \\
f^*F\big((x<y)\big) \arrow[r,"\phi\big((x<y)\big)"] \arrow[d,swap,"\text{id}"] & f^*G\big((x<y)\big) \arrow[d,"\text{id}"] \\
F(y) \arrow[r,"\psi(y)"] & G(y)
\end{tikzcd}
\end{equation*}
Since, $f^*F\big((x)\leq (x<y)\big)=F(x<y)$, the morphism $\psi\colon F\rightarrow G$ is well defined. Moreover, for $(x_0<\cdots<x_n)\in\mathcal{K}(X)$ we obtain $f^*\psi\big((x_0<\cdots<x_n)\big)=\psi(x_n)=\phi(x_n)=\phi\big((x_0<\cdots<x_n)\big)$. Therefore, $H_{F,G}(\psi)=f^*\psi=\phi$ and $H_{F,G}$ is surjective. This implies that $f^*$ is fully faithful.
\end{proof}

\begin{proposition} \label{equivalence}
Let $X$ be a finite poset and $f\colon \mathcal{K}(X)\rightarrow X$ the projection from the order complex, then the unit $\eta\colon \text{id}\rightarrow f_*f^*$ of the adjunction $f^*\dashv f_*$ is an isomorphism.
\end{proposition}

\begin{proof}
The functor $f^*$ is left adjoint to $f_*$. By \cite[Proposition 3.4]{nlab:adjoint_functor},
given adjoint functors $f^*\dashv f_*$, the unit $\eta\colon \text{id}\rightarrow f_*f^*$ of the adjunction is an isomorphism if and only if $f^*$ is fully faithful. By Proposition \ref{fullyfaithful}, $f^*$ is fully faithful.
\end{proof}

\begin{proposition} \label{order_proj_iso}
Let $X$ be a finite poset, $f\colon \mathcal{K}(X)\rightarrow X$ the projection from the order complex and $F$ a sheaf on $X$, then the induced morphism $H^k(f)\colon H^k(X,F)\rightarrow H^k(\mathcal{K}(X),f^*F)$ is an isomorphism for all $k\geq 0$.
\end{proposition}

\begin{proof}
Let $0\rightarrow F\rightarrow I^\bullet$ be an injective resolution. By Proposition \ref{equivalence}, we obtain the following isomorphism 
\begin{equation*}
\begin{tikzcd}
0 \arrow[r] & F \arrow[r] \arrow[d,"\eta_F"] & I^\bullet \arrow[d,"\eta_{I^\bullet}"] \\
0 \arrow[r] & f_*f^*F \arrow[r] & f_*f^*I^\bullet 
\end{tikzcd}
\end{equation*}
This isomorphism induces an isomorphism in cohomology $H^k\Gamma\eta_{I^\bullet}\colon H^k\Gamma I^\bullet\rightarrow H^k\Gamma f_*f^*I^\bullet$ for all $k\geq 0$. By the proof of  \cite[Lemma 3.15]{brown2022discrete}, $0\rightarrow f^*F\rightarrow f^* I^\bullet$ is an acyclic resolution of $f^*F$. Let $0\rightarrow f^*F\rightarrow J^\bullet$ be an injective resolution of $f^*F$, then, by Theorem 6.2 in Chapter 20 of \cite{lang2005algebra}, there exists a morphism $g^\bullet$ 
\begin{equation*}
\begin{tikzcd}
0 \arrow[r] & f^*F \arrow[r] \arrow[d,"\text{id}"] & f^*I^\bullet \arrow[d,"g^\bullet"] \\
0 \arrow[r] & f^*F \arrow[r] & J^\bullet 
\end{tikzcd}
\end{equation*}
lifting the identity, such that $H^k\Gamma g^\bullet\colon H^k\Gamma f^*I^\bullet\rightarrow H^k\Gamma J^\bullet$ is an isomorphism for all $k\geq 0$. Therefore, $H^k(f)=H^k\Gamma g^\bullet\circ H^k\Gamma \eta_{I^\bullet}$ is an isomorphism for all $k\geq 0$.
\end{proof}

\cancel{
\begin{theorem}
If $X$ is a finite poset, $\Vec{F}$ a persistent sheaf on $X$ and $f\colon \mathcal{K}(X)\rightarrow X$ the projection from the order complex, then $H^k(X,\Vec{F})\cong H^k(\mathcal{K}(X),f^*\Vec{F})$.
\end{theorem}
}

\begin{proof}[Proof of Theorem~\ref{thm:poset}]
By \cite[A.5]{russold}, we have the following commutative diagram
\begin{equation*}
\begin{tikzcd}[row sep=large,column sep=huge]
H^k(X,\Vec{F}_i) \arrow[d,swap,"H^k(f\text{,}\Vec{F}_i)"] \arrow[r,"H^k(X\text{,}\Vec{F}_i^{i+1})"] &[20pt] H^k(X,\Vec{F}_{i+1}) \arrow[d,"H^k(f\text{,}\Vec{F}_{i+1})"] \\
H^k(\mathcal{K}(X),f^*\Vec{F}_i) \arrow[r,"H^k(\mathcal{K}(X)\text{,}f^*\Vec{F}_i^{i+1})"] & H^k(\mathcal{K}(X),f^*\Vec{F}_{i+1}) 
\end{tikzcd}
\end{equation*}
and, by Proposition \ref{order_proj_iso}, the vertical morphisms are isomorphisms for all $i\in\mathbb{N}_0$.
\end{proof}

\subsection{Simplification for zigzag posets} \label{app_zigzag_simplification}

\begin{proposition} \label{H0_prop}
If $X$ is a zigzag-poset, $F$ a sheaf on $X$, and $\iota\colon X'\xhookrightarrow{} X$ the inclusion of the alternating subposet, then $\text{lim }F\cong \text{lim }\iota^*F$.
\end{proposition}

\begin{proof}
A sheaf on $X$ is a functor $F\colon X\rightarrow \mathbf{vec}$. Moreover, let $F'\coloneqq\iota^*F\colon X'\rightarrow \mathbf{vec}$ and note that $F'=F|_{X'}$. Denote by $(\lim F,(\phi_x)_{x\in X})$ the limit of $F$ and $(\lim F',(\psi_x)_{x\in X'})$ the limit of $F'$. The restriction $(\lim F,(\phi_x)_{x\in X'})$ is obviously a cone of $F'$. Hence, by the universal property of the limit, there is a unique morphism $\lim F \xrightarrow{f} \lim F'$ compatible with the morphisms $\phi_x$ and $\psi_x$. We can also extend the limit of $F'$ to a cone of $F$ in the following way. Let $x<y\in X$, then there exists a unique minimal element $z\in X'$ such that $z\leq x<y$. We can now define $\overline{\psi}_x\coloneqq F(z\leq x)\circ \psi_z$ and $\overline{\psi}_y\coloneqq F(z\leq y)\circ \psi_z$ to obtain the commutative diagram
\begin{equation*}
\begin{tikzcd}[column sep=large]
F(x) \arrow[rr,"F(x<y)"] & & F(y) \\
& F(Z) \arrow[ul,swap,"F(z\leq x)"] \arrow[ur,"F(z\leq y)"] \\
& \lim F' \arrow[u,"\psi_z"] \arrow[uul,"\overline{\psi}_x"] \arrow[uur,swap,"\overline{\psi}_y"]
\end{tikzcd}
\end{equation*}
This implies that there is a unique morphism $\lim F' \xrightarrow{g} \lim F$ compatible with the morphisms $\phi_x$ and $\overline{\psi}_x$.
By combining these results, we obtain the following diagram
\begin{equation*}
\begin{tikzcd}[column sep=large]
F(x) \arrow[rr,"F(x\leq y)"] & & F(y) \\
& \lim F \arrow[ul,swap,"\phi_x"] \arrow[ur,"\phi_y"] \\
& \lim F' \arrow[u,"g"] \arrow[uul,"\overline{\psi}_x"] \arrow[uur,swap,"\overline{\psi}_y"] \\
& \lim F \arrow[u,"f"] \arrow[uuul,bend left,"\phi_x"] \arrow[uuur,swap,bend right,"\phi_y"]
\end{tikzcd}
\end{equation*}
By the properties of $f$ and $g$, we have $\overline{\psi}_x\circ f=F(z\leq x)\circ \psi_z\circ f=F(z\leq x)\circ \phi_z=\phi_x$ and $\phi_x\circ g=\overline{\psi}_x$. Hence, the diagram commutes and, by the universal property of the limit, $g\circ f=\text{id}$. In a similar fashion we obtain $f\circ g=\text{id}$. Therefore, $\lim F\cong \lim F'$. 
\end{proof}

\begin{proposition} \label{prop_inj}
Let $X$ be a zigzag-poset and $\iota\colon X'\xhookrightarrow{} X$ the inclusion of the alternating subposet, then $\iota^*$ preserves injectives.
\end{proposition}

\begin{proof}
Let $y\in X$ and $[y]$ be the elementary injective sheaf \cite[Def 7.1.3]{curry} concentrated on $y$, i.e.\
\begin{equation*}
[y](x)=\begin{cases}
\field \hspace{5pt} \text{if } x\leq y \\
0 \hspace{6pt} \text{else}
\end{cases} \hspace{2pt} .
\end{equation*}
If $y$ is not an internal maximal element in $X$, then there exists a unique minimal element $z\in X'\subseteq X$ such that $z\leq y$. This implies
\begin{equation*}
\iota^*[y](x)=[y](\iota(x))=[y](x)=\begin{cases}
\field \hspace{5pt} \text{if } x\leq y \\
0 \hspace{6pt} \text{else}
\end{cases}=\begin{cases}
\field \hspace{5pt} \text{if } x=z \\
0 \hspace{6pt} \text{else}
\end{cases}=[z] \hspace{2pt} .
\end{equation*}
If $y$ is an internal maximal element in $X$, then $y\in X'$ and we obtain
\begin{equation*}
\iota^*[y](x)=[y](\iota(x))=[y](x)=\begin{cases}
\field \hspace{5pt} \text{if } x\leq y \\
0 \hspace{6pt} \text{else}
\end{cases}=[y] \hspace{2pt} .
\end{equation*}
Hence, $\iota^*[y]$ is an elementary injective sheaf on $X'$. Since every injective sheaf on a poset is a direct sum of elementary injective sheaves and the inverse image functor is additive, $\iota^*$ preserves injectives.
\end{proof}

\begin{proposition} \label{prop_iso}
Let $X$ be a zigzag-poset, $\iota\colon X'\xhookrightarrow{} X$ the inclusion of the alternating subposet, and $\eta\colon \text{id}\rightarrow \iota_*\iota^*$ the unit of the push-pull-adjunction. Then, for any sheaf $F$ on $X$ we have that $\Gamma\eta_F\colon \Gamma F\rightarrow \Gamma\iota_*\iota^* F$ is an isomorphism. 
\end{proposition}

\begin{proof}
For any $x\in X$ we get
\begin{equation*}
\iota_*\iota^*F(x)=\underset{y\in X':\iota(y)\geq x}{\text{lim }}\iota^*F(y)=\underset{y\in X':y\geq x}{\text{lim }}F(y)=\begin{cases}
0 \hspace{5pt} \text{if x is external} \\
F(x) \hspace{5pt} \text{if x is minimal} \\
F(z) \hspace{5pt} \text{otherwise, where $z$} \\ \text{ is the unique maximal element $\geq x$ }
\end{cases}
\end{equation*}
Moreover, for $\eta_F\colon F\rightarrow \iota_*\iota^*F$, given by the induced map to the limit $\eta_F(x)\colon F(x)\rightarrow \underset{y\in X':y\geq x}{\text{lim }}F(y)$, we get
\begin{equation*}
\eta_F(x)=\begin{cases}
0 \hspace{5pt} \text{if x is external} \\
\text{id} \hspace{5pt} \text{if x is minimal} \\
F(x\leq z) \hspace{5pt} \text{otherwise, where $z$ is the unique maximal element $\geq x$ }
\end{cases}
\end{equation*}
Hence, $\eta_F(x)=F(x\leq x)=\text{id}$ if $x$ is minimal or internal maximal. Since $\Gamma F=\text{lim }F$ and, by Proposition \ref{H0_prop}, the limit over a zigzag-poset only depends on the minimal and internal maximal elements and the maps between them, we obtain $\Gamma \eta_F\colon \Gamma F\rightarrow \Gamma\iota_*\iota^* F$ is an isomorphism.
\end{proof}

\begin{proposition} \label{prop_ind}
Let $X$ be a zigzag-poset and $\iota\colon X'\xhookrightarrow{} X$ the inclusion of the alternating subposet. If $F$ is a sheaf on $X$, then the induced morphism $H^k(\iota)\colon H^k(X,F)\rightarrow H^k(X',\iota^*F)$ is an isomorphism. 
\end{proposition}

\begin{proof}
Let $0\rightarrow F\rightarrow I^\bullet$ be an injective resolution. Since, $\iota^*$ is exact and preserves injectives by Proposition \ref{prop_inj}, $0 \rightarrow \iota^*F\rightarrow \iota^*I^\bullet$ is an injective resolution of $\iota^*F$. The morphism induced in cohomology by $\iota$ is given by the following cochain morphism \cite[A.4]{russold}
\begin{equation*}
\begin{tikzcd}
0 \arrow[r] & \Gamma I^0 \arrow[r] \arrow[d,"\Gamma \eta_{I^0}"] & \Gamma I^1 \arrow[r] \arrow[d,"\Gamma \eta_{I^1}"] & \Gamma I_2 \arrow[r] \arrow[d,"\Gamma \eta_{I^2}"] & \cdots \\
0 \arrow[r] & \Gamma \iota_*\iota^*I^0 \arrow[r] \arrow[d,"\text{id}"] & \Gamma \iota_*\iota^*I^1 \arrow[r] \arrow[d,"\text{id}"] & \Gamma \iota_*\iota^*I_2 \arrow[r] \arrow[d,"\text{id}"] & \cdots \\
0 \arrow[r] & \Gamma \iota^*I^0 \arrow[r] & \Gamma \iota^*I^1 \arrow[r] & \Gamma \iota^*I_2 \arrow[r] & \cdots 
\end{tikzcd}
\end{equation*}
where $\Gamma \iota_{*}\iota^*I^k=\Gamma\iota^*I^k$ for all $k\geq 0$. By Proposition \ref{prop_iso}, $\Gamma \eta_{I_k}$ is an isomorphism for every $k\geq 0$. Thus, $(\Gamma \eta_{I^k})_{k\geq 0}$ is a cochain isomorphism and $H^k(\iota)\colon H^k(X,F)\rightarrow H^k(X',\iota^*F)$ is an isomorphism.
\end{proof}

\cancel{
\begin{theorem} \label{zigzag_reduction_thm}
Let $X$ be a zigzag-poset and $\iota\colon X'\xhookrightarrow{} X$ the inclusion of the alternating subposet. If $\Vec{F}$ is a persistent sheaf on $X$, then $ H^k(X,\Vec{F})\cong H^k(X',\iota^*\Vec{F})$. 
\end{theorem}
}

\begin{proof}[Proof of Theorem~\ref{zigzag_reduction_thm}]
By \cite[A.5]{russold}, for every $i\in\mathbb{N}_0$, the following diagram commutes
\begin{equation*}
\begin{tikzcd}[column sep=huge,row sep=large]
H^k(X,\Vec{F}_i) \arrow[r,"H^k(X\text{,}\Vec{F}_i^{i+1})"] \arrow[d,swap,"H^k(\iota\text{,}\Vec{F}_i)"] &[10pt] H^k(X,\Vec{F}_{i+1}) \arrow[d,"H^k(\iota\text{,}\Vec{F}_{i+1})"] \\
H^k(X',\iota^*\Vec{F}_i) \arrow[r,"H^k(X'\text{,}\iota^*\Vec{F}_i^{i+1})"] & H^k(X',\iota^*\Vec{F}_{i+1})
\end{tikzcd}
\end{equation*}
and, by Proposition \ref{prop_ind}, the vertical morphisms are isomorphism. Therefore, $ H^k(X,\Vec{F})\cong H^k(X',\iota^*\Vec{F})$. 
\end{proof}

\begin{corollary}
Let $X$ be a zigzag-poset and $\Vec{F}$ a persistent sheaf on $X$, then $H^k(X,\Vec{F})=0$ for all $k\geq 2$.
\end{corollary}

\begin{proof}
Let $X'$ be the alternating subposet of $X$. By Theorem \ref{zigzag_reduction_thm}, we have $H^k(X,\Vec{F})\cong H^k(X',\iota^*\Vec{F})$. Since, $X'$ is the face-relation-poset of a one-dimensional simplicial complex $H^k(X',\iota^*\Vec{F})=0$ for all $k\geq 2$.
\end{proof}


\section{Experiments}
\label{app_experiments}
We have an implementation (\url{https://github.com/TDA-Jyamiti/Algos-cplxs-pers-modules/}) of the algorithms proposed in the paper. In this section, we report the run-time results of computing persistent sheaf cohomology of inputs of different sizes. For implementation purposes, we consider (co)homology over $\mathbb{Z}_2$.
All the experiments were performed on a machine with 4 Intel(R) Xeon(R) Gold 6248R CPUs.

Given a persistent sheaf over a simplicial complex, we use the algorithm in Section \ref{sec_presentation_algorithm} to obtain presentations of the persistence modules associated with each simplex and the presentations of morphisms between these persistence modules. With this processed input, we use the algorithm presented in Section \ref{sec_cohomology_presentations} to compute the presentation (barcode) of persistent sheaf cohomology. Figure \ref{fig:app_graph_input_bars} shows the processed input of a persistent sheaf over a graph. The persistence modules on each vertex and edge are decomposed into interval modules (red and blue bars) and the green arrows show the maps between these interval modules. 

For our experiments, we consider Erd\H{o}s-R\'enyi graphs of different sizes with $p=0.4$. We generate persistence modules of the same length over each vertex and edge and maps between these persistence modules corresponding to the face relation in the graph. Note that these persistence modules and the corresponding maps are represented as a sequence of matrices (the input for the algorithm in Section \ref{sec_presentation_algorithm}). We report the results in Table \ref{tab:app_run_times}.

\begin{table}
    \centering    \begin{tabular}{|c|c|c|l|} \hline 
         \textbf{Number of simplices}&  \textbf{Length of Persistence Module}&  $\mathbf{\Sigma n_i}$&\textbf{Time}\\ \hline 
         2060&  500&  5143040&285 s 
\\ \hline 
         563&  500&  1429218&79 s\\ \hline
 534& 500&  1270016&68 s\\ \hline 
 531& 100&  221260&15 s
\\ \hline 
 98& 500&  204672 &13 s
\\ \hline 
 115& 50&  15770&3 s
\\ \hline
    \end{tabular}
    \caption{Run times for computing persistent sheaf cohomology. The third column shows the input size $\Sigma n_i$ as defined in section \ref{sec_presentation_algorithm}.}
    \label{tab:app_run_times}
\end{table}

\begin{figure}
    \centering
    \includegraphics[scale=0.5]{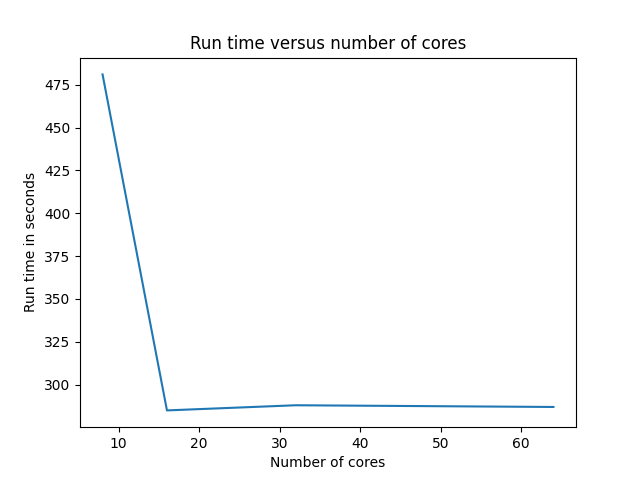}
    \caption{Plot of run time of the algorithm to compute persistent sheaf cohomology versus number of cores being used for computation.}
    \label{fig:app_num_cores_times}
\end{figure}

We can compute the presentations of persistence modules and morphism between them in parallel for each vertex-edge pair. In Figure \ref{fig:app_num_cores_times}, we plot the run times for computing persistent sheaf cohomology of a persistent sheaf over a graph with 2060 simplices with filtration length of 500 and the input size 5.1 million versus the number of cores used for computation. We see that there is a significant increase in performance from 8 to 16 cores. However, increasing the number of cores beyond 16 does not really give any performance enhancement.

We note that in these experiments, we compute degree 0 cohomology of the persistent sheaf because there are no higher dimensional simplices. Thus, according to \eqref{comp_presentations}, we only compute $f_0, f_1, p, q$ because $g_0, g_1, r$ are $0$-maps. We would like to point out that the run-time for computing the subsequent higher degree cohomology would be very similar because we can reuse $f_0, f_1, p, q$ and only compute $g_0, g_1$ and $r$ for an input that has higher dimensional simplices and the corresponding persistence modules over each of those simplex.

\begin{figure}
    \centering
    \includegraphics[width=\textwidth]{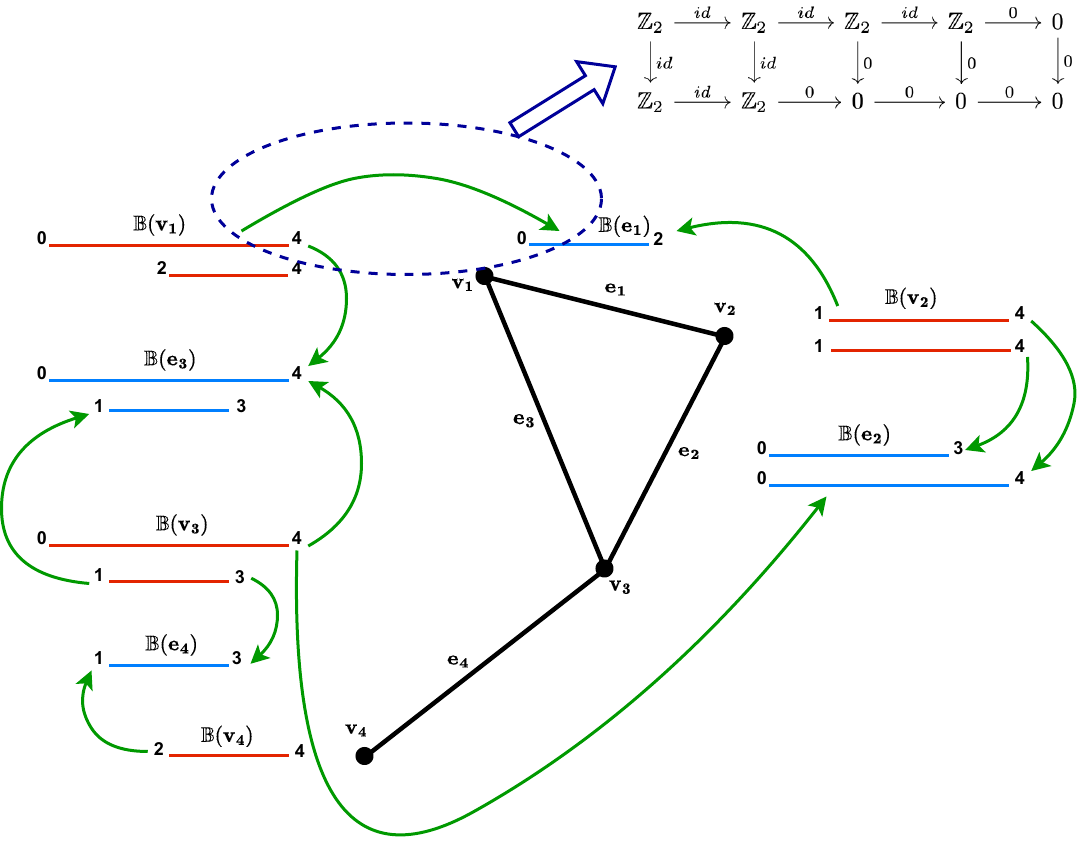}
    \caption{The figure above shows a processed persistent sheaf over a graph with $4$ vertices and $4$ edges. In the figure, the persistence modules are represented as bars/interval modules (red bars denote the persistence modules associated with vertices and blue bars denote those associated with edges). Each bar has a birth-time $b$ and a death-time $d$ shown towards the left and right of the bar respectively and is to be interpreted as a bar $[b,d)$. The green arrows show the map between persistence modules which can be represented as maps between interval modules (bars). One of the green arrows, towards the top part of the figure, has been zoomed in to show how one of the red bars (interval modules) associated with $\mathbf{v_1}$ is mapped to the blue bar associated with $\mathbf{e_1}$. $\mathbb{B}(\sigma)$ denotes the interval decomposition of the persistence module associated with $\sigma$ where $\sigma \in \{ \mathbf{v_1}, \mathbf{v_2}, \mathbf{v_3}, \mathbf{v_4}, \mathbf{e_1}, \mathbf{e_2}, \mathbf{e_3}, \mathbf{e_4}\}$.}
    \label{fig:app_graph_input_bars}
\end{figure}
\end{document}